\documentclass[letterpaper,12pt,hidelinks]{article}
\usepackage{configuration,natbib,bbm}

\begin{document}
\title{Sampling with replacement vs Poisson sampling:\\
  a comparative study in optimal subsampling
\footnote{The first two authors contributed equally to this work.}}
\author{Jing Wang
  \footnote{Department of Statistics, University of Connecticut, USA},$\quad$
  Jiahui Zou
  \footnote{School of Statistics, Capital University of Economics and Business, Beijing 100070, China}, $\quad$and$\quad$
  HaiYing Wang$^*$}
\maketitle

\begin{abstract}
Faced with massive data, subsampling is a commonly used technique to improve computational efficiency, and using nonuniform subsampling probabilities is an effective approach to improve estimation efficiency. For computational efficiency, subsampling is often implemented with replacement or through Poisson subsampling. However, no rigorous investigation has been performed to study the difference between the two subsampling procedures such as their estimation efficiency and computational convenience. This paper performs a comparative study on these two different sampling procedures. In the context of maximizing a general target function, we first derive asymptotic distributions for estimators obtained from the two sampling procedures. The results show that the Poisson subsampling may have a higher estimation efficiency. Based on the asymptotic distributions for both subsampling with replacement and Poisson subsampling, we derive optimal subsampling probabilities that minimize the variance functions of the subsampling estimators. These subsampling probabilities further reveal the similarities and differences between subsampling with replacement and Poisson subsampling. The theoretical characterizations and comparisons on the two subsampling procedures provide guidance to select a more appropriate subsampling approach in practice. Furthermore, practically implementable algorithms are proposed based on the optimal structural results, which are evaluated through both theoretical and empirical analyses.

\textit{keywords}: Algorithmic sampling; Asymptotic Distribution; Informative Sample; Massive Data. %
\end{abstract}

\section{Introduction}
\label{sec:introduction}
With fast development of technology, data collecting is becoming easier and easier, and the volumes of available data sets are increasing exponentially. To extract useful information from these massive data, a major challenge lies with the thirst for computing resources. Subsampling is a commonly used technique to reduce computational burden, and it has been an important topic in computer science and statistics with a long standing of literature, such as \cite{drineas2006fast1, drineas2006fast2, drineas2006fast3, mahoney2009cur, Drineas:11, mahoney2011randomized, %
  clarkson2013low, kleiner2014scalable, mcwilliams2014fast, yang2015randomized}.

To improve the estimation efficiency\footnote{The estimation efficiency is different from that discussed in Chapter 8 of \cite{Vaart:98}, which focuses on achieving the asymptotic lower bound of regular estimators. Here we focus on taking a subsample that better approximates the full data estimator, and we consider it with computational efficiency simultaneously.}, nonuniform subsampling probabilities are often used so that more informative data points are sampled with higher probabilities. A popular choice is the leverage-based subsampling in which the subsampling distribution is the normalized statistical leverage scores of the design matrix \citep{Drineas:12,PingMa2015-JMLR}. 
\cite{yang2015explicit} showed that if statistical leverage scores are very nonuniform, then using their normalized square roots as the subsampling distribution yields better approximation. For logistic regression, \cite{WangZhuMa2017} derived an optimal subsampling distribution that minimizes the asymptotic variance of the subsampling estimator, and \cite{wang2019more} further developed a more efficient estimation approach based on the selected subsample. 
\cite{Ting_NIPS2018_7623} investigated optimal subsampling with influence functions.
\cite{WangYangStufken2018} proposed a method called information-based optimal subdata selection
which selects data points deterministically for linear regression.
The subsampling approach has a close connection to the technique of coreset approximation \citep{campbell2018bayesian,campbell2019automated}, which also use a subset of the data with associated weights instead of the full data to reduce calculations. The coreset approximation is often used in Bayes analysis and the problem is often to better approximate the objective function in a functional space, while this paper focuses on approximating the full data estimator.

For computational efficiency, subsampling is often implemented with replacement or through Poisson subsampling. %
Subsampling with replacement needs to use all subsampling probabilities simultaneously to generate random numbers from a multinomial distribution. The resultant subsample observations are independent and identically distributed (i.i.d.) conditional on the full data, but their unconditional distributions are not independent. Poisson subsampling considers each data point and determines if it should be included in the subsample by generating a random number from the uniform distribution. If the subsampling probabilities in Poisson subsample are all equal, then the subsampling procedure is also called the Bernoulli subsampling \citep{sarndal2003model}. 
For Poisson subsampling, the resultant subsample observations do not have identical conditional distributions, but their unconditional distributions can be independent.

Although subsampling with replacement and Poisson subsampling are commonly used in practice, no rigorous investigation has been performed to compare them, especially in the context of optimal subsampling. When they perform similarly and when one is preferable to the other?
This paper studies this topic, and has the following major contributions. 1) In the context when an estimator is obtained by maximizing a target function, we first derive conditional and unconditional asymptotic distributions for estimators from both subsampling with replacement and Poisson subsampling. These asymptotic distributions accurately characterize the subsampling approximation errors, and we derive general structure results of optimal subsampling probabilities to minimize these errors for the two subsampling procedures. 
2) We systematically compare subsampling with replacement and Poisson subsampling, both theoretically and empirically. %
We identify conditions when %
the asymptotic distributions for subsampling with replacement and for Poisson subsampling are the same, and when they are different. We also discuss the similarity and difference for the two subsampling procedures in terms of the structural results of optimal subsampling probabilities. 
3) Based on the optimal subsampling probabilities, we propose practical algorithms and evaluate their performance through both theoretical analysis and numerical experiments. 

It is worth mentioning that our investigation views subsampling as a computational tool and investigates it within a statistical framework. For computer scientists, subsampling is a commonly used randomized device to speed up computing by using a subsample estimator to approximate the full data estimator \citep[e.g.,][]{mcwilliams2014fast,woodruff2014sketching}, while for statisticians resampling is widely adopted in exchangeable bootstrap schemes to build confidence regions \citep[e.g.,][]{shao1995jackknife,Poli:Roma:Wolf:subs:1999}. This paper lies in the middle of these two communities. We derive asymptotic distributions of subsampling estimators in a similar fashion to existing literature on bootstrap. However, our purpose is not to establish the bootstrap consistency. %
Instead, we utilize the asymptotic distributions to develop better subsampling probabilities so that the subsample estimator better approximate the full data estimator. In addition, we focus on data dependent subsampling probabilities for which existing investigations and techniques on bootstrap do not apply because they require data independent and exchangeable sampling weights \citep{praestgaard1993exchangeably,cheng2010bootstrap}. 
\black

 The rest of the paper is organized as follows. 
We present the model setup and asymptotic distributions in Section~\ref{sec:model setup}. In Section~\ref{sec:optim-subs-prob}, we derive optimal subsampling probabilities and propose practical algorithms. We will also obtain theoretical properties for the practical algorithms. 
In Section~\ref{sec:numer-exper}, we %
perform numerical experiments demonstrating the performance of the proposed methods. Proofs %
of our theoretical results are provided in the appendix.

Here are some notation conventions to be used in the paper. We use $^*$ to indicate subsample quantities; use $\hat{}$ to indicate full data estimator; use $\tilde{}$ to indicate subsample estimator; use $_R$ and $_P$ to indicate subsampling with replacement and Poisson subsampling, respectively; use $\dot{m}$ and $\ddot{m}$ to denote the gradient and Hessian matrix of a function $m$ with respect to the parameter $\btheta$; use $\op$ or $\Op$ to denote a sequence that converges to zero in probability or is bounded in probability, respectively; use $\rightsquigarrow$ to denote convergence in distribution; use $\|\bv\|$ to denote the Euclidean norm of a vector $\bv$; and use $\|\A\|$ to denote the Frobenius  norm of a matrix $\A$.

\section{Problem setup and asymptotic distributions}
\label{sec:model setup}
Suppose that a set of training data $\Dn=\{Z_i\}_{i=1}^n$ consists of  independent observations from the distribution that generates $Z$. To estimate some parameter $\btheta\in\mathbb{R}^d$ about the data distribution, we want to calculate $\htheta_n$, the maximizer of
\begin{equation*}%
  \M_n(\btheta)=\onen\sumn\m(Z_i,\btheta).
\end{equation*}
Here the dimension of $Z_i$ does not have to be the same as $\btheta$, e.g., in softmax regression.
Usually, there is no closed-form solution to $\htheta_n$, and an iterative algorithm is required to find the solution numerically. 
For massive data, iterative calculations on the full data of size $n$ are often too expensive, so subsampling is adopted to produce a subsampling estimator $\ttheta$ to approximate $\htheta_n$. Nonuniform subsampling probabilities are often used to improve the estimation efficiency.

Let $\bpi=\{\pi_{n,i}\}_{i=1}^n$ be a subsampling distribution such that $\pi_{n,i}>0$ and $\sumn\pi_{n,i}=1$. 
For Poisson subsampling, we further assume that $\pi_{n,i}\le s_n^{-1}$, where $s_n$ is the expected subsample size. As stated early, we use $^*$ to indicate quantities with randomness due to subsampling.
For instance, let $Z_1^*, ..., Z_{s_n}^*$ denote the resampled sample and let $\pi_{n,1}^*,...,\pi_{n,s_n}^*$ be the corresponding resampled %
  subsampling probabilities.

We present the general subsampling estimators $\ttheta_{s_n,R}$ based on subsampling with replacement and $\ttheta_{s_n,P}$ based on Poisson subsampling, comparatively, in the following Algorithm~\ref{alg:0}. %
\begin{algorithm}[H]%
  \caption{Subsampling with replacement vs Poisson subsampling}
  \label{alg:0}
\begin{minipage}[t]{.46\textwidth}
  {\bf Sampling with replacement}
  \begin{itemize}[leftmargin=3.5mm]
  \item Calculate $\bpi=\{\pi_{n,i}\}_{i=1}^n$ %
    based on $\Dn$;
  \item generate $s_n$ independent random numbers from multinomial distribution with $\bpi$ to determine a subsample
    $\mathcal{D}_{s_n}^*=\{Z_1^*,Z_2^*,...,Z_{s_n}^*\}$;
  \item record $\{\pi^ {*}_{n,1}, \pi^{*}_{n,2},...,\pi^{*}_{n,s_n}\}$ in the subsample;
  \item obtain the subsample estimator %
\begin{equation}\label{eq:27}
  \ttheta_{s_n,R}=\arg\max_{\btheta}
  \sumrr\frac{\m(Z_i^*,\btheta)}{ns_n\pi_{n,i}^*}.
\end{equation}
\end{itemize}
\end{minipage}\hspace{5mm}
\begin{minipage}[t]{0.46\textwidth}
   {\bf Poisson Sampling:}
\begin{itemize}[leftmargin=5.5mm]
\item For each $i=1, ..., n$, calculate an individual $\pi_{n,i}$ such that $\pi_{n,i}\le s_n^{-1}$ based on $Z_i$;
  \item generate $u_i\sim U(0,1)$; 
  \item if $u_i \le s_n\pi_{n,i}$, include $Z_i$ in the subsample and record $\pi_{n,i}$; %
  \item obtain the subsample estimator
\begin{equation}\label{eq:28}
\ttheta_{s_n,P}=\arg\max_{\btheta}
\sumr\frac{\m(Z_i^*,\btheta)}{ns_n^*\pi_{n,i}^*}.
\end{equation}
\end{itemize}
\end{minipage}
\end{algorithm}

\begin{remark}\normalfont
  In Algorithm~\ref{alg:0}, we see that  subsampling with replacement requires to access the whole sampling distribution $\bpi=\{\pi_{n,i}\}_{i=1}^n$, i.e., all $\pi_{n,i}$'s, because they are the parameters in the multinomial distribution from which random numbers are generated. On the other hand, Poisson subsampling only needs to access one $\pi_{n,i}$ in each sampling consideration. This makes the Poisson subsampling more convenient to implement, especially when the available memory cannot hold all $\pi_{n,i}$'s or in distributed computing platforms.   
  For subsampling with replacement, the subsample size is equal to $s_n$ and there may be replicates in the subsample. Here $\pi_{n,i}$ is the probability that observation $Z_i$ is selected when only one data point is selected, and the probability to include $Z_i$ in the subsample of size $s_n$ is $1-(1-\pi_{n,i})^{s_n}$, which is smaller than $s_n\pi_{n,i}$.
  For Poisson subsampling, the subsample size $s_n^*$ is random with $\Exp(s_n^*)=s_n$; there is no replicates in the subsample; %
 and $s_n\pi_{n,i}$ is the probability of including $Z_i$ in the subsample of expected size $s_n$. %
\end{remark}

\begin{remark}\normalfont
  Another way of implementing Poisson subsampling is to remove the condition of $\pi_{n,i}\le s_n^{-1}$ and replace $\pi_{n,i}$ with $\min(s_n\pi_{n,i},1)$. The expected subsample size from this approach would be difficult to determine as $\pi_{n,i}$'s are often calculated on the go as scanning through the full data. We only know that the expected subsample size would be smaller than $s_n$. %
  In this paper, we focus on the Poisson subsampling procedure described in Algorithm~\ref{alg:0}.  
\end{remark}

We now derive asymptotic properties of $\ttheta_{s_n,R}$ in (\ref{eq:27}) and $\ttheta_{s_n,P}$ in (\ref{eq:28}), respectively, to compare their estimation efficiency theoretically. We need some regularity assumptions listed below. 

\begin{assumption}\label{ass:1}
  The parameter $\btheta$ belongs to a compact set.
\end{assumption}
\begin{assumption}\label{ass:2}
  The function $m(Z,\btheta)$ is a concave function of $\btheta$ with a unique and finite maximum, and it satisfies that $\Exp\{ m^2(Z,\btheta)\}<\infty$ for any $\btheta$.
\end{assumption}

\begin{assumption}\label{ass:3}
The matrix $-\Exp\{\ddm(Z,\btheta)\}$ %
is positive-definite, %
$\Exp\{\ddm_{k,l}^2(Z,\btheta)\}<\infty$, and $\ddm(Z,\btheta)$ is Lipschitz continuous in $\btheta$ so that %
there exists a function $\psi(z)$ with $\Exp\{\psi^2(Z)\}<\infty$ and for every $\btheta_1$ and $\btheta_2$, %
$|\ddm_{k,l}(z, \btheta_1)-\ddm_{k,l}(z, \btheta_2)|\leq \psi(z)\|\btheta_1-\btheta_2\|$, $k,l=1,2,...,d$.
\end{assumption}

\begin{assumption}\label{ass:4}
  The matrix $\Lambda(\btheta)=\Exp\{\dm(Z,\btheta)\dm\tp(Z,\btheta)\}$ is positive-definite, and for $\btheta$ in the neighborhood of $\htheta_n$, $\onen\sumn\|\dm(Z_i,\btheta)\|^{4}=\Op$. %

\end{assumption}

\begin{assumption}\label{ass:6}
The sampling distribution $\bpi$ satisfies that  $\max_{i=1,...,n}(n\pi_{n,i})^{-1}=O_P(1)$.
\end{assumption}
Assumptions~\ref{ass:1} and \ref{ass:2} are very mild, and they assure that the target function has a finite and unique maximum. Assumptions~\ref{ass:3} and \ref{ass:4} impose some constraints on the Hessian matrix and gradient of $m(Z,\btheta)$; Assumption~\ref{ass:3} is used to prove the consistency of subsample estimators and Assumption~\ref{ass:4} is used to establish the asymptotic normality of subsample estimators. Assumption~\ref{ass:6} essentially requires that the minimum subsampling probability is at the same order of $\onen$ in probability. Here, $\pi_{n,i}$ can be random as it is allowed to depend on the data, so the notation $O_P(1)$ is used. This assumption is required so that the  objective function based on a subsample would not be dominated by data points with very small $\pi_{n,i}$'s. Very small $\pi_{n,i}$'s %
may not matter %
when characterizing the worst-case bound, e.g., \cite{Drineas:12}, but they do %
impact the statistical properties of subsampling algorithms. %
Due to this, \cite{PingMa2015-JMLR} proposed the ``shrinkage'' leverage scores to prevent the statistical performance of algorithmic leveraging algorithm from being deteriorated by very small leverage scores. 

Let $\btheta_0=\arg\max_{\btheta}\Exp\{\m(Z,\btheta)\}$ be the true parameter that generates the data. %
The following proposition is a known result \citep[see, e.g., Chapter 5 of][]{Vaart:98}.

\begin{proposition}\label{prop1}
Under Assumptions~\ref{ass:1} and \ref{ass:3}, if $\Lambda(\btheta)$ is positive-definite (the first part of Assumption~\ref{ass:4}), then
\begin{equation*}
  \sqrt{n}(\htheta_n-\btheta_0)
  \rightsquigarrow \Nor\{\0, V(\btheta_0)\},
\end{equation*}
where $V(\btheta)=\ddM^{-1}(\btheta)\Lambda(\btheta)\ddM^{-1}(\btheta)$ and  $M(\btheta)=\Exp\{m(\btheta, Z)\}$. %
\end{proposition}
\black

To assess the distributional properties of subsample estimators, we need to derive the distribution asymptotically, i.e., to assume that $s_n\rightarrow\infty$ and $n\rightarrow\infty$. We assume that $s_n<n$, because a primary goal of subsampling is to reduce the subsample size, but we do not require $s_n=o(n)$.

We define some notations for convergence in conditional distribution and probability before presenting our results.
Let $\Delta_{n,s_n}$ be a vector function of a subsample of size $s_n$ from the full data $\Dn$, e.g., a subsample estimator. We say that $\Delta_{n,s_n}$ converges in conditional probability given $\Dn$ in probability and write it as $\Delta_{n,s_n}=\opD(1)$, if $\Pr(\|\Delta_{n,s_n}\|>\delta|\Dn)=\op$ for any $\delta>0$; this can be equivalently stated as for any $\delta>0$ and $\epsilon>0$, as $s_n\rightarrow\infty$ and $n\rightarrow\infty$,
\begin{equation*}
  \Pr\Big\{\Pr(\|\Delta_{n,s_n}\|>\delta|\Dn) \le \epsilon\Big\}\rightarrow1.
\end{equation*}
We say that $\Delta_{n,s_n}$ is bounded in conditional probability given $\Dn$ in probability and write it as $\Delta_{n,s_n}=\OpD(1)$, if for any $\epsilon>0$ there exists a $0<K_{\epsilon}<\infty$ such that
as $s_n\rightarrow\infty$ and $n\rightarrow\infty$,
\begin{equation*}
  \Pr\big\{\Pr(\|\Delta_{n,s_n}\|>K_{\epsilon}|\Dn)
  \le\epsilon\big\}\rightarrow1.
\end{equation*}
We say that $\Delta_{n,s_n}$ (of dimension $d$) converges in conditional distribution to a continuous random vector $U$ given $\Dn$ in probability and denote this as $\Delta_{n,s_n}\cvdD U$, if $\Pr(\Delta_{n,s_n}\le\x|\Dn)-\Pr(U\le\x)=\op$ for every $\x\in\mathbb{R}^d$;
this can also be stated as that for any $\epsilon>0$ and every $\x\in\mathbb{R}^d$, as $s_n\rightarrow\infty$ and $n\rightarrow\infty$,
\begin{equation*}
  \Pr\Big\{\Big|\Pr(\Delta_{n,s_n}\le\x|\Dn)-\Pr(U\le\x)\Big|
  \le\epsilon\Big\}\rightarrow1.
\end{equation*}
\begin{proposition}\label{prop2}
  The following results hold for conditional convergence.
  \begin{enumerate}[(a)]
  \item \label{item:1} If $\Delta_{n,s_n}=\opD(1)$ then $\Delta_{n,s_n}=\op$, and vice versa.
  \item \label{item:2} If $\Delta_{n,s_n}=\OpD(1)$ then $\Delta_{n,s_n}=\Op$, and vice versa.
  \item \label{item:3} If $\Delta_{n,s_n}\cvdD U$ then $\Delta_{n,s_n}\cvd U$, and vice versa.
  \end{enumerate}
\end{proposition}
\black

The following Theorems~\ref{thm:1} and \ref{thm:2} present conditional asymptotic distributions of $\ttheta_{s_n,R}$ in (\ref{eq:27}) and $\ttheta_{s_n,P}$ in (\ref{eq:28}), respectively, when approximating the full data estimator $\htheta_n$.

\begin{theorem}\label{thm:1}
Under Assumptions~\ref{ass:1}-\ref{ass:6}, as $s_n\rightarrow\infty$ and $n\rightarrow\infty$, the estimator $\ttheta_{s_n,R}$ in (\ref{eq:27}) 
satisfies that 
\begin{equation}\label{eq:tthetaAsim}
  \sqrt{s_n}\{V_{n,R}(\htheta_n)\}^{-1/2}(\ttheta_{s_n,R}-\htheta_n)
  \ \cvdD\  \Nor(\0,\I_d),
\end{equation}
where $\Nor(\0,\I_d)$ is a multivariate Gaussian distribution with mean $\0$ and variance $\I_d$ (the identity matrix of dimension $d$), $V_{n,R}(\btheta)=\ddM_n^{-1}(\btheta)\Lambda_{n,R}(\btheta)\ddM_n^{-1}(\btheta)$,
\begin{equation}\label{eq:1}
  \ddM_n(\btheta)=\onen\sumn\ddm(Z_i,\btheta),
  \quad\text{and}\quad
\Lambda_{n,R}(\btheta)=\frac{1}{n^2}\sumn
      \frac{\dm(Z_i,\btheta)\dm\tp(Z_i,\btheta)}{\pi_{n,i}}.
\end{equation}
\end{theorem}

\begin{theorem}\label{thm:2}
Under Assumptions~\ref{ass:1}-\ref{ass:6}, as $s_n\rightarrow\infty$ and $n\rightarrow\infty$, the estimator $\ttheta_{s_n,P}$ in (\ref{eq:28}) satisfies that, 
\begin{equation}\label{eq:thm5}
  \sqrt{s_n}\{V_{n,P}(\htheta_n)\}^{-1/2}(\ttheta_{s_n,P}-\htheta_n)
  \ \cvdD\  \Nor\left(\0,\I\right),
\end{equation}
where $V_{n,P}(\btheta)=\ddM_n^{-1}(\btheta)\Lambda_{n,P}(\btheta)\ddM_n^{-1}(\btheta)$,  $\ddM_n(\btheta)$ is the same as in \eqref{eq:1}, and
\begin{align}
  \Lambda_{n,P}(\btheta)
  =\Lambda_{n,R}(\btheta)-\frac{s_n}{n^2}\sumn
    \dm(Z_i,\btheta)\dm\tp(Z_i,\btheta).\label{eq:2}
\end{align}
\end{theorem}
\begin{remark}\normalfont
  The asymptotic distributions in \eqref{eq:tthetaAsim} and \eqref{eq:thm5} mean that given a full data set for any $\delta>0$, the probability that $\|\ttheta_{s_n,R}-\htheta_n\|>\delta$ is accurately approximated by $\Pr(\|U_R\|>\delta)$ where $U_R\sim \Nor\{\0,V_{n,R}(\htheta_n)\}$, and the probability that $\|\ttheta_{s_n,P}-\htheta_n\|>\delta$ is accurately approximated by $\Pr(\|U_P\|>\delta)$ where $U_P\sim \Nor\{\0,V_{n,P}(\htheta_n)\}$.
Thus, a smaller variance means a smaller probability of excess error at the same error bound, or a smaller error bound for the same excess probability.  %
\end{remark}
\begin{remark}\normalfont
 Both $\ttheta_{s_n,R}$ and $\ttheta_{s_n,P}$ have Gaussian asymptotic distributions, but they have different asymptotic variances $V_{n,R}(\htheta_n)$ and $V_{n,P}(\htheta_n)$, respectively. %
  \black  
  Under Assumption~\ref{ass:4}, the second term on the right-hand-side of \eqref{eq:2} goes to zero in probability if $s_n/n\rightarrow0$, and it converges to a positive-definite matrix in probability if $s_n/n\rightarrow c>0$. Thus,
  the difference $V_{n,R}(\htheta_n)-V_{n,P}(\htheta_n)\rightarrow\bm0$ in probability if $s_n/n\rightarrow0$, and it converges to a positive-definite matrix in probability if $s_n/n$ converges to a positive constant. This means that subsampling with replacement and Poisson subsampling have the same asymptotic estimation efficiency only if the subsampling ratio $s_n/n$ goes to zero; otherwise, Poisson subsampling has a higher estimation efficiency. Thus, to obtain more accurate estimates in practice, Poisson subsampling is recommended unless the subsampling ratio $s_n/n$ is very small.
\end{remark}

\begin{remark}\normalfont
  If the sampling distribution $\bpi$ is constructed so that $\Lambda_{n,R}(\btheta)\rightarrow\Lambda(\btheta)$ in probability uniformly in a neighborhood of $\btheta_0$, then $V_{n,R}(\htheta_n)$ and $(1-c)^{-1}V_{n,P}(\htheta_n)$ both converge in probability to $V(\btheta_0)$, the scaled asymptotic variance of ${\htheta}$. This means both subsample estimators have the bootstrap consistency in this scenario. A class of sampling distributions satisfies this situation if $\bpi$  does not dependent on the data, such as the class of exchangeable
  bootstrap weights which includes the uniform sampling distribution \citep[see][]{praestgaard1993exchangeably,cheng2010bootstrap}. However if $\bpi$ depends on the data, then $\Lambda_{n,R}(\btheta)$ may not converge to $\Lambda(\btheta)$ \footnote{This is still possible in some special cases such as the local case control subsampling for logistic regression \citep{fithian2014local,wang2019more}.}, and in this case the subsample estimators do not have the bootstrap consistency. The goal of this paper is different from the line of research about bootstrap that focuses on constructing conference region nor approximating complicated distributions \citep[see][]{bickel:97,Poli:Roma:Wolf:subs:1999}, so bootstrap inconsistency is not a concern. Nevertheless, if multiple subsamples are taken, then the average of the subsample estimates is recommended and the variance can be estimated from these subsample estimates using the approach proposed in \cite{wang2021optimal}.
\end{remark}

Although the convergence in conditional distribution $\cvdD$ can be replaced by $\cvd$ because of Proposition~\ref{prop2} \ref{item:3}, Theorems~\ref{thm:1} and \ref{thm:2} are about approximating the full data estimator and they are conditional results in nature. In the following, we derive the unconditional asymptotic distribution when the true parameter is of interest to further compare the two subsampling approaches.

\begin{theoremp}{1'}\label{thm:1p}
  Under Assumptions~\ref{ass:1}-\ref{ass:6}, if $\Lambda_{n,R}(\btheta_{0})$ converges to a positive-definite matrix $\Lambda_{\pi}(\btheta_{0})$ as $s_n\rightarrow\infty$ and $n\rightarrow\infty$, %
  then the estimator $\ttheta_{s_n,R}$ in (\ref{eq:27}) satisfies that 
  \begin{equation*}
    \sqrt{s_n}(\ttheta_{s_n,R}-\btheta_{0})
    \cvd\Nor\big\{\0,\;V_{R}^{U}(\btheta_{0})\big\},
  \end{equation*}
   where $V_{R}^{U}(\btheta)=\ddM^{-1}(\btheta)\Lambda_{R}^{U}(\btheta)\ddM^{-1}(\btheta)$, $\Lambda_{R}^{U}(\btheta)=\Lambda_{\pi}(\btheta)+c\Lambda(\btheta)$, and $c=\lim\frac{s_n}{n}$.

\end{theoremp}

\begin{theoremp}{2'}\label{thm:2p}
Under Assumptions~\ref{ass:1}-\ref{ass:6}, if $\Lambda_{n,R}(\btheta_{0})$ converges to a positive-definite matrix $\Lambda_{\pi}(\btheta_{0})$ as $s_n\rightarrow\infty$ and $n\rightarrow\infty$, then the estimator $\ttheta_{s_n,P}$ in \eqref{eq:28} satisfies that
\begin{equation*}
  \sqrt{s_n}(\ttheta_{s_n,P}-\btheta_{0})
  \cvd\Nor\big\{\0,\;V_{P}^{U}(\btheta_{0})\big\},
\end{equation*}
where $V_{P}^{U}(\btheta)=\ddM^{-1}(\btheta)\Lambda_{\pi}(\btheta)\ddM^{-1}(\btheta)$.

\end{theoremp}
\begin{remark}\normalfont
In Theorems~\ref{thm:1p} and~\ref{thm:2p}, the unconditional asymptotic distributions of $\ttheta_{s_n,R}$ and $\ttheta_{s_n,P}$ for estimating the true parameter are also Gaussian with (scaled) variances $V_{R}^{U}(\btheta_0)$ and $V_{P}^{U}(\btheta_0)$, respectively.
From the two theorems, we see that $V_{R}^{U}(\btheta_0)=V_{P}^{U}(\btheta_0)+cV(\btheta_0)$, where $V(\btheta_0)$ is the scaled asymptotic variance for the full data estimator in Propositio~\ref{prop1}. Here, $V_{P}^{U}(\btheta_0)$ can be interpreted as the variation due to subsampling and $cV(\btheta_0)$ can be interpreted as the variation due to the randomness of the full data. It is interesting to note that the asymptotic variance components due to the two sources are additive for the subsampling with replacement estimator $\ttheta_{s_n,R}$, while $cV(\btheta_0)$ does not contribute to the asymptotic variance of the Poisson subsampling estimator $\ttheta_{s_n,P}$.
\end{remark}
\black
\section{Optimal subsampling probabilities}
\label{sec:optim-subs-prob}
From the results in Theorems~\ref{thm:1} and \ref{thm:2}, the asymptotic variances $V_{n,R}(\htheta_n)$ and $V_{n,P}(\htheta_n)$ depend on $\bpi=\{\pi_{n,i}\}_{i=1}^n$. To improve the estimation efficiency, we want to choose optimal $\bpi$ to minimize $V_{n,R}(\htheta_n)$ or $V_{n,P}(\htheta_n)$. Specifically, we consider the L-optimality criterion \citep[Section 10.5 of][]{atkinson2007optimum}. 
The L-optimality minimizes the trace of the asymptotic variance matrix
for some {\bf l}inear transformation, say $L$, of the parameter estimator. For our case, this is to minimize $\tr\{LV_{n,R}(\htheta_n)L\tp\}$ or $\tr\{LV_{n,P}(\htheta_n)L\tp\}$ for some matrix $L$, because $LV_{n,R}(\htheta_n)L\tp$ and $LV_{n,P}(\htheta_n)L\tp$ are the asymptotic variances of $L\ttheta_{s_n,R}$ and $L\ttheta_{s_n,R}$, respectively.
If we take $L=\I$, then the resulting criterion is also called the A-optimality; this is to minimize the {\bf a}verage of the variances for all parameter components by minimizing the trace of the variance matrix, i.e., minimizing $\tr\{V_{n,R}(\htheta_n)\}$ or $\tr\{V_{n,P}(\htheta_n)\}$. 
If we take $L=\ddM_n(\htheta_n)$, then the resultant criterion is to minimize $\tr\{\Lambda_{n,R}(\htheta_n)\}$ or $\tr\{\Lambda_{n,P}(\htheta_n)\}$. This has a computational advantage compared with other choices, so we focus more on this choice in this paper. %
The following Theorems~\ref{thm:3} and \ref{thm:4} present the optimal subsampling probabilities for subsampling with replacement and Poisson subsampling, respectively. %

\begin{theorem}\label{thm:3}
  For the subsampling with replacement estimator in (\ref{eq:27}), %
  the L-optimal subsampling probabilities %
  with $L=\ddM_n(\htheta_n)$ that minimize $\tr\{\Lambda_{n,R}(\htheta_n)\}$ are
\begin{equation}\label{eq:SSPmMSE}
  \pi_{n,Ri}^{\mvc}
  =\frac{\|\dm(Z_i,\htheta_n)\|}{\sumjn\|\dm(Z_j,\htheta_n)\|},
  \quad i=1, ..., n.
   \end{equation}
\end{theorem}

\begin{theorem}\label{thm:4}
  For the Poisson subsampling estimator in (\ref{eq:28}),
  the L-optimal subsampling probabilities %
 with $L=\ddM_n(\htheta_n)$ that minimize $\tr\{\Lambda_{n,P}(\htheta_n)\}$ are
\begin{align}\label{eq:11}
  \pi_{n,Pi}^{\mvc}
  =\frac{\|\dm(Z_i,\htheta_n)\|\wedge H}
  {\sumjn\{\|\dm(Z_j,\htheta_n)\|\wedge H\}},
  \quad i=1, ..., n,
\end{align}
where $a\wedge b=\min(a,b)$,
\begin{align}
H=\frac{\sum_{i=1}^{n-g}\|\dm(Z,\htheta_n)\|_{(i)}}{s_n-g},\label{eq:7}
\end{align}
$\|\dm(Z,\htheta_n)\|_{(1)}\le ... \le \|\dm(Z,\htheta_n)\|_{(n)}$ are the order statistics of $\|\dm(Z_1,\htheta_n)\|, ..., \|\dm(Z_n,\htheta_n)\|$, and $g$ is an integer such that
\begin{align}
  \frac{\|\dm(Z,\htheta_n)\|_{(n-g)}}
  {\sum_{i=1}^{n-g}\|\dm(Z,\htheta_n)\|_{(i)}}
  &<\frac{1}{s_n-g}
  \quad\text{and}\quad  
  \frac{\|\dm(Z,\htheta_n)\|_{(n-g+1)}}
  {\sum_{i=1}^{n-g+1}\|\dm(Z,\htheta_n)\|_{(i)}}
  \geq\frac{1}{s_n-g+1}\label{eq:9},
\end{align}
in which we define $\|\dm(Z,\htheta_n)\|_{(n+1)}=\infty$.
\end{theorem}

\begin{remark}\normalfont
  For a general choice of $L$, we can obtain optimal subsampling probabilities by replacing $\|\dm(Z_i,\htheta_n)\|$ with $\|\dm(Z_i,\htheta_n)\|_L=\|L\ddM^{-1}_n(\htheta_n)\dm(Z_i,\htheta_n)\|$. However, these quantities require $O(nd^2)$ time to compute when $\ddM_n^{-1}(\htheta_n)$ and $\dm(Z_i,\htheta_n)$ are available, where $n$ is the full data sample size and $d$ is dimension of $\htheta_n$. 
On the other hand, it only \black takes $O(nd)$ time to compute all $\|\dm(Z_i,\htheta_n)\|$'s.  %
  Thus the choice of $L=\ddM_n(\htheta_n)$ has a significant computational advantage. 
\end{remark}

\begin{remark}\normalfont
  In Theorems~\ref{thm:3} and \ref{thm:4}, $\pi_{n,Ri}^{\mvc}$ in \eqref{eq:SSPmMSE} and $\pi_{n,Pi}^{\mvc}$ in \eqref{eq:11} have both similarities and differences. Assuming that $\|\dm(Z_i,\htheta_n)\|>0$ for all $i$, then $0<\pi_{n,Ri}^{\mvc}<1$ while $0<\pi_{n,Pi}^{\mvc}\le\oner$. This means that the inclusion of any data point through optimal subsampling with replacement is random, while the inclusion of data points with $\pi_{n,Pi}^{\mvc}=\oner$ is deterministic through optimal Poisson subsampling. The order statistics constraint in \eqref{eq:9} indicates that if there are data points such that $\frac{s_n}{n}\|\dm(Z_i,\htheta_n)\|>\onen\sumjn\|\dm(Z_j,\htheta_n)\|$, then $\pi_{n,Ri}^{\mvc}$ and $\pi_{n,Pi}^{\mvc}$ are different. This means that if the subsampling ratio $\frac{s_n}{n}$ is larger or if the tail of the distribution of $\|\dm(Z,\htheta_n)\|$ is heavier, then optimal probabilities for Poisson subsampling and subsampling with replacement are more likely to be different.   
 If $s_n\|\dm(Z,\htheta_n)\|_{(n)}<\sumn\|\dm(Z_i,\htheta_n)\|$, then $\pi_{n,Ri}^{\mvc}$ and $\pi_{n,Pi}^{\mvc}$ are identical. This condition is true with probability approaching one under some conditions, e.g., when $s_n\log n=o(n)$ and the distribution of $\|\dm(Z,\htheta_n)\|$ has a sub-Gaussian tail because in this case $\frac{s_n}{n}\|\dm(Z,\htheta_n)\|_{(n)}=\op$ and  $\onen\sumn\|\dm(Z_i,\htheta_n)\|$ goes to a positive constant in probability.
\end{remark}

\begin{remark}\normalfont
  In Theorem~\ref{thm:4}, $H$ is the threshold so that all $\pi_{n,Pi}^{\mvc}$ are no larger than $\oner$, and it satisfies that
  \begin{equation}\label{eq:15}
    \|\dm(Z,\htheta_n)\|_{(n-g)} < H \le \|\dm(Z,\htheta_n)\|_{(n-g+1)}.
  \end{equation}
  Here $g$ is the number of cases that $\pi_{n,Pi}^{\mvc}=\oner$, i.e., the number of data points that will be included in the subsample for sure.
\end{remark}

Now we discuss an example to illustrate the optimal structural results. Additional examples are available in Section~\ref{sec:addit-exampl-optim} of the Appendix.

\begin{example}[Binary response models]\normalfont
  \label{example1}
Consider a binary classification model such that %
\begin{equation*}
  \Pr(y_i=1)=p(\x_i,\btheta), \quad i=1, ..., n,
\end{equation*}
where $y_i\in\{0,1\}$ is the binary class label, $\x_i$ is the covariate, and $\btheta$ is the unknown parameter. %
To estimate $\btheta$ using the maximum likelihood estimator (MLE), let $Z_i=(\x_i, y_i)$ and
\begin{align*}
  \m(Z_i,\btheta)=y_i\log\{p(\x_i,\btheta)\}
    +(1-y_i)\log\{1-p(\x_i,\btheta)\}.
\end{align*}
Direct calculations yield that
\begin{align}
  \dm(Z_i,\htheta_n)
  &=\frac{y_i-\hat{p}_i}{\hat{p}_i(1-\hat{p}_i)}\hat{\dot{p}}_i,
    \quad\text{and}\quad
  \|\dm(Z_i,\htheta_n)\|
  =\frac{|y_i-\hat{p}_i|}{\hat{p}_i(1-\hat{p}_i)}\|\hat{\dot{p}}_i\|,
  \label{eq:17}
\end{align}
where $\hat{p}_i=p(\x_i,\htheta_n)$, and $\hat{\dot{p}}_i=\dot{p}(\x_i,\htheta_n)$ is the gradient of $p(\x_i,\btheta)$ evaluated at $\htheta_n$. We can obtain optimal sampling probabilities by inserting the expression in~\eqref{eq:17} 
into Theorems~\ref{thm:3} and \ref{thm:4}.

To obtain the general L-optimal subsampling probabilities with any $L$, the Hessian matrix $\ddm(Z_i,\hat\btheta)$ of $\m(Z_i,\hat\btheta)$ is
\begin{align}\label{eq:31}
  \ddm(Z_i,\hat\btheta)
  &=\frac{y_i-\hat{p}_i}{\hat{p}_i(1-\hat{p}_i)}\hat{\ddot{p}}_i
    -\Big\{\frac{y_i}{\hat{p}_i^2}+\frac{1-y_i}{(1-\hat{p}_i)^2}\Big\}
    \hat{\dot{p}}_i\hat{\dot{p}}_i\tp,
\end{align}
where $\hat{\ddot{p}}_i=\ddot{p}(\x_i,\htheta_n)$ is the Hessian matrix of $p(\x_i,\btheta)$ evaluated at $\htheta_n$. 
Thus, we obtain the general L-optimal sampling probabilities by using
\begin{align}\label{eq:33}
  \|\dm(Z_i,\htheta_n)\|_L
  &=\frac{|y_i-\hat{p}_i|}{\hat{p}_i(1-\hat{p}_i)}
    \|L\ddM_n^{-1}(\htheta_n)\hat{\dot{p}}_i\|,
\end{align}
to replace $\|\dm(Z_i,\htheta_n)\|$ in Theorems~\ref{thm:3} and \ref{thm:4}, for any $L$, where
\begin{align}\label{eq:40}
  \ddM_n(\htheta_n)
  &=\sumn\frac{y_i-\hat{p}_i}{\hat{p}_i(1-\hat{p}_i)}\hat{\ddot{p}}_i
    -\sumn\Big\{\frac{y_i}{\hat{p}_i^2}
    +\frac{1-y_i}{(1-\hat{p}_i)^2}\Big\}
    \hat{\dot{p}}_i\hat{\dot{p}}_i\tp.
\end{align}
Under some regularity conditions,  $\onen\sumn\frac{y_i-\hat{p}_i}{\hat{p}_i(1-\hat{p}_i)}\hat{\ddot{p}}_i$ is a small term in \eqref{eq:40}, and therefore $\ddM_n(\htheta_n)$ in \eqref{eq:33} can be replaced by
\begin{align}\label{eq:16}
    \ddM_n^a(\htheta_n)=-\onen\sumn       \Big\{\frac{y_i}{\hat{p}_i^2}+\frac{1-y_i}{(1-\hat{p}_i)^2}\Big\}
    \hat{\dot{p}}_i\hat{\dot{p}}_i\tp.
\end{align}
Thus, there is no need to calculate the Hessian matrix $\hat{\ddot{p}}_i$.

From~\eqref{eq:17} or \eqref{eq:33},
the optimal subsampling probabilities are proportional to $|y_i-\hat{p}_i|$. Thus if $y_i=1$, data points with smaller values of $\hat{p}_i$ are sampled with higher probabilities; if $y_i=0$, data points with larger values of $\hat{p}_i$ are sampled with higher probabilities. The optimal subsampling probabilities give higher preference to data points that are closer to the class boundary. This increases the classification accuracy %
because if these data points can be classified correctly, then other data points are easier to classify.

Specifically for Logistic regression in which
\begin{equation*}
  p(\x_i,\btheta)=\frac{e^{\x_i\tp\btheta}}{(1+e^{\x_i\tp\btheta})},
\end{equation*}
we have $\hat{\dot{p}}_i=\hat{p}_i(1-\hat{p}_i)\x_i$ and $\hat{\ddot{p}}_i=\hat{p}_i(1-\hat{p}_i)(1-2\hat{p}_i)\x_i\x_i\tp$. Thus, for this case
\begin{align}
  \|\dm(Z_i,\htheta_n)\|
  &=|y_i-\hat{p}_i|\|\x_i\|, \quad\text{ and }\quad\label{eq:32}\\
  \|\dm(Z_i,\htheta_n)\|_L
  &=|y_i-\hat{p}_i|\|L\ddM_n^{-1}(\htheta_n)\x_i\|,
    \quad\text{with}\quad\label{eq:38}
  \ddM_n(\htheta_n)=-\onen\sumn\hat{p}_i(1-\hat{p}_i)\x_i\x_i\tp.
\end{align}
If the expression %
in \eqref{eq:32}, or the expression in \eqref{eq:38} with $L=\I$, is used in Theorems~\ref{thm:3}, the structural results for optimal probabilities of subsampling with replacement are identical to those in \cite{WangZhuMa2017}. If \eqref{eq:16} is used, then the expression of $\ddM_n^a(\htheta_n)$ is 
$\ddM_n^a(\htheta_n)=-\onen\sumn(y_i-\hat{p}_i)^2\x_i\x_i\tp$, which has the same limit as $\ddM_n(\htheta_n)$ in \eqref{eq:38}.

From Theorem~\ref{thm:4} we see that if there are data points such that $\frac{s_n}{n}|y_i-\hat{p}_i|\|\x_i\|>\onen\sumjn|y_j-\hat{p}_j|\|\x_j\|$, then optimal probabilities for Poisson subsampling are different from that for subsampling with replacement.

\end{example}

\subsection{Practical algorithms}\label{sec:pract-impl}
The optimal subsampling probabilities depend on the full data estimator $\htheta_n$, so the structural results in the previous section do not translate into useful algorithms directly. We need a pilot estimator to approximate the optimal subsampling probabilities in order to obtain practically implementable algorithms.
This can be done by taking a pilot subsample of size $s_0$ through a subsampling distribution that does not depend on $\htheta_n$. For simplicity, we use the uniform subsampling distribution $\bpi^{\uni}=\{\pi_{n,i}=\onen\}_{i=1}^n$, and present the approximated optimal subsampling with replacement procedure in Algorithm~\ref{alg:3}.

Compared with the exact $\pi_{n,Ri}^{\mvc}$, the approximated $\tilde{\pi}_{n,Ri}^{\mvc}$ in \eqref{eq:13} are subject to additional disturbance due to the randomness of $\ttheta_{s_0,R}^{0*}$, the maximizer of \eqref{eq:12}. From Theorem~\ref{thm:1},  %
the subsampling probabilities are in the denominators of $\Lambda_{n,R}(\htheta_n)$.  %
Thus the additional disturbance may be amplified for data points with $\pi_{n,Ri}^{\mvc}$ being close to zero, and this may inflate the asymptotic variance of the subsample estimator.
To protect the estimator from these data points, we adopt the idea of defensive importance sampling \citep{hesterberg1995weighted, owen2000safe} and mix the approximated optimal subsampling distribution with the uniform subsampling distribution. Specifically, we use $\tilde{\pi}_{n,R\alpha i}^{\mvc}$ instead of $\tilde{\pi}_{n,Ri}^{\mvc}$ in \eqref{eq:13} to perform the subsampling. The same idea was also adopted in \cite{PingMa2015-JMLR}.

In $\tilde{\bm{\pi}}_{R\alpha i}=\{\tilde{\pi}_{n,R\alpha i}^{\mvc}\}_{i=1}^n$, $\alpha$ controls the proportion of mixture, and $\tilde{\bm{\pi}}_{R\alpha i}$ is close to the optimal subsampling distribution if $\alpha$ is close to 0 while it is close to the uniform subsampling distribution if $\alpha$ is close to 1. If $\alpha>0$, then $n\pi_{n,R\alpha i}^{\mvc}$ are bounded away from zero, which add to robustness of the subsampling estimator.

\begin{algorithm}[H]%
\caption{Practical %
  algorithm based on optimal subsampling with replacement}
\label{alg:3}
\begin{algorithmic}
  \STATE $\bullet$ {\it Pilot subsampling}:
  use sampling with replacement %
  with $\bm\pi^{\uni}$ to obtain $\{Z_1^{0*}, ..., Z_{s_{0}}^{0*}\}$;   %
obtain $\ttheta_{s_0,R}^{0*}$ through maximizing\-\\[-5mm]
  \begin{equation}
    M_R^{0*}(\btheta)
    =\sum_{i=1}^{s_0}\frac{\m(Z_i^{0*},\btheta)}{s_0}.\label{eq:12}
  \end{equation}
  \STATE \-\\[-12mm] $\bullet$ {\it Approximated optimal subsampling:}
  \STATE calculate the whole subsampling distribution $\tilde{\bm{\pi}}_{R\alpha i}=\{\tilde{\pi}_{n,R\alpha i}^{\mvc}\}_{i=1}^n$, where $\alpha\in(0,1)$,
   \begin{align}
     \tilde{\pi}_{n,Ri}^{\mvc}
     &=\frac{\|\dm(Z_i,\ttheta_{s_0,R}^{0*})\|}
       {\sumjn\|\dm(Z_j,\ttheta_{s_0,R}^{0*})\|}%
     \quad\text{and}\quad
     \tilde{\pi}_{n,R\alpha i}^{\mvc}
     =(1-\alpha)\tilde{\pi}_{n,Ri}^{\mvc}+\alpha\onen;\label{eq:13}
   \end{align}
  \STATE \-\\[-12mm] use $\tilde{\bm{\pi}}_{R\alpha i}$ to take a subsample $\{Z_1^*, ..., Z_{s_n}^*\}$, and record the corresponding probabilities $\{\tilde{\pi}_{R\alpha 1}^{\mvc*}, ..., \tilde{\pi}_{R\alpha s}^{\mvc*}\}$.
  \STATE $\bullet$ {\it Estimation:}
 obtain $\ttheta_{s_n,R}^{\alpha}$ through maximizing\-\\[-3mm]
\begin{equation}\label{solution-subsample}
  M_{R\alpha}^{*}(\btheta) =%
  \sumrr\frac{\m(Z_i^*,\btheta)}{ns_n\tilde{\pi}_{n,R\alpha i}^{\mvc*}}.
\end{equation}
\end{algorithmic}
\end{algorithm}
\begin{algorithm}[H]%
\caption{Practical %
  algorithm based on optimal Poisson subsampling}
\label{alg:4}
\begin{algorithmic}
  \STATE $\bullet$ {\it Pilot subsampling}: 
  use Poisson sampling with $\bpi^{\uni}$ to obtain $\{Z_1^{0*}, ..., Z_{s_{0}^*}^{0*}\}$;
  \STATE
  obtain $\ttheta_{s_0,P}^{0*}$ through maximizing\-\\[-6mm]
  \begin{equation}
    M^{0*}_P(\btheta)
    =\sum_{i=1}^{s_0^*}\frac{\m(Z_i^{0*},\btheta)}{s_0^*};\label{eq:14}
  \end{equation}
  \STATE\-\\[-12mm]
     calculate\-\\[-9mm]
    \begin{align}
      H^{0*}&=\|\dm(Z_i^{0*},\ttheta_{s_0,P}^{0*})\|_{\frac{s}{bn}},
       \quad\text{and}\quad
    \Psi^{0*}
     =\sum_{i=1}^{s_0^*}\frac{\{\|\dm(Z_i^{0*},\ttheta_{s_0,P}^{0*})\|\wedge
      H^{0*}\}}{s_0^*}.\label{eq:5}
  \end{align}
  \STATE\-\\[-12mm] $\bullet$ {\it Approximated optimal subsampling:}
   For each $i$ of $i=1, ..., n$, 
   \STATE calculate\-\\[-9mm]
   \begin{align}
     \tilde{\pi}_{n,Pi}^{\mvc}
     &=\frac{\|\dm(Z_i,\ttheta_{s_0,P}^{0*})\|\wedge H^{0*}}{n\Psi^{0*}},
     \quad\text{and}\quad
     \tilde{\pi}_{n,P\alpha i}^{\mvc}
     =(1-\alpha)\tilde{\pi}_{n,Pi}^{\mvc}+\alpha\onen;\label{eq:4}
   \end{align}
   \STATE\-\\[-12mm]
   generate $u_i\sim U(0,1)$; 
   \STATE
   if {$u_i \le s_n\tilde{\pi}_{n,P\alpha i}^{\mvc}$}, 
   include $Z_i$ in the subsample and record $\tilde{\pi}_{n,P\alpha i}^{\mvc}$.
  \STATE $\bullet$ {\it Estimation:}
 obtain $\ttheta_{s_n,P}^{\alpha}$ through maximizing\-\\[-3mm]
\begin{equation}\label{solution-subsample25}
  M_{P\alpha}^{*}(\btheta) =\onen\sumr
  \frac{\m(Z_i^*,\btheta)}{(s_n\tilde{\pi}_{n,P\alpha i}^{\mvc*})\wedge1}.
\end{equation}

\end{algorithmic}
\end{algorithm}

For the optimal Poisson subsampling probability $\pi_{n,Pi}^{\mvc}$, we also need to use the pilot subsample to approximate $H$ and $\Psi=\onen\sumn\{\|\dm(Z_i,\htheta_n)\|\wedge H\}$ in order to determine the inclusion probability based on each data point itself, as described in Algorithm~\ref{alg:4}.
From \eqref{eq:15}, $H$ is between the $(n-g)$-th and the $(n-g+1)$-th order statistics of $\{\|\dm(Z_i,\htheta_n)\|\}_{i=1}^{n}$, and $g$ is between $0$ and $s_n$,
so we can roughly approximate $H$ with $\|\dm(Z_i^{0*},\ttheta_{s_0,P}^{0*})\|_{\frac{s_n}{bn}}$, the upper $\frac{s_n}{bn}$-th sample quantile of  $\{\|\dm(Z_i^{0*},\ttheta_{s_0,P}^{0*})\|\}_{i=1}^{s_0}$, where $b\ge1$ is a tuning parameter. Since $g$ is typically %
closer to $0$ and farther from $s_n$,
  taking $b=1$ %
  underestimates $H$ and the resulting subsampling probabilities lean towards the uniform subsampling probability (if $H\le\|\dm(Z,\htheta_n)\|_{(1)}$,  then $\pi_{n,Pi}^{\mvc}$ would be all equal to $\onen$).
When subsampling from massive data, $s_n$ is often much smaller than $n$ and %
the number of cases for $\|\dm(Z_i,\htheta_n)\|$ to be larger than $H$ is small. %
For this scenario, one may simply ignore $H$ and
use $\infty$ to replace $H$. This simple option in general overestimates $H$, but it may perform reasonably well for small subsampling ratios. %
For $\Psi$, it can be approximated by $\Psi^{0*}$ defined in \eqref{eq:5}.

When we use $\Psi^{0*}$ and $H^{0*}$ to replace $\Psi$ and $H$ in \eqref{eq:4}, it is possible that some  $\tilde{\pi}_{n,Pi}^{\mvc}$ in \eqref{eq:4} are larger than $\oner$ and thus $s_n\tilde{\pi}_{n,Pi}^{\mvc}$ are larger than one. %
Thus, we use one as a threshold in the denominator of \eqref{solution-subsample25}.

\begin{remark}\normalfont\label{rmk:6}
  In Algorithm~\ref{alg:3}, $\ttheta_{s_0,R}^{0*}$ and $\ttheta_{s_n,R}^{\alpha}$ can be combined to obtain an aggregated estimator,
\begin{align*}
  \ctheta_R
  =\big(s_0\ddM_R^{0*}+s\ddM_R^{*}\big)^{-1}
  \times \big(s_0\ddM_R^{0*}\times\ttheta_{s_0,R}^{0*}
    +s\ddM_R^{*}\times\ttheta_{s_n,R}^{\alpha}\big),
\end{align*}
where $\ddM_R^{0*}$ is the Hessian matrix of $\M_R^{0*}(\btheta)$ in \eqref{eq:12} evaluated at $\ttheta_{s_0,R}^{0*}$ and $\ddM_R^{*}$ is the Hessian matrix of $\M_R^{*}(\btheta)$ in \eqref{solution-subsample} evaluated at $\ttheta_{s_n,R}^{\alpha}$. Similarly, in Algorithm~\ref{alg:4}, $\ttheta_{s_0,P}^{0*}$ and $\ttheta_{s_n,P}^{\alpha}$ can be combined to obtain an aggregated estimator,
\begin{align*}
  \ctheta_P
  =\big(s_0^*\ddM_P^{0*}+s_n\ddM_P^{*}\big)^{-1}
  \times \big(s_0^*\ddM_P^{0*}\times\ttheta_{s_0,P}^{0*}
    +s_n\ddM_P^{*}\times\ttheta_{s_n,P}^{\alpha}\big),
\end{align*}
where $\ddM_P^{0*}$ is the Hessian matrix of $\M_P^{0*}(\btheta)$ in \eqref{eq:14} evaluated at $\ttheta_{s_0,P}^{0*}$ and $\ddM_P^{*}$ is the Hessian matrix of $\M_P^{*}(\btheta)$ in \eqref{solution-subsample25} evaluated at $\ttheta_{s_n,P}^{\alpha}$. 
Here,  $\ctheta_R$ is obtained as a linear combination of $\ttheta_{s_0,R}^{0*}$ and $\ttheta_{s_n,R}^{\alpha}$, and $\ctheta_P$ is obtained as a linear combination of $\ttheta_{s_0,P}^{0*}$ and $\ttheta_{s_n,P}^{\alpha}$ in a way similar to the aggregation step in the divide-and-conquer method \citep{LinXie2011,schifano2016online}. This further improves the estimation efficiency. %
\end{remark}

\subsection{Theoretical analysis of practical algorithms}
We obtain the following distributional results in Theorems~\ref{thm:5} and \ref{thm:6} for Algorithms~\ref{alg:3} and \ref{alg:4}, respectively.
\begin{theorem}\label{thm:5}
  For $\ttheta_{s_n,R}^{\alpha}$ obtained from Algorithm~\ref{alg:3}, under Assumptions~\ref{ass:1}-\ref{ass:4}, as $s_0$, $s_n$, and $n$ get large, the following result holds. Given $\Dn$ and $\ttheta_{s_0,R}^{0*}$ in probability,
\begin{equation*}
  \sqrt{s_n}\{V_{n,R}^{\alpha}(\htheta_n)\}^{-1/2}(\ttheta_{s_n,R}^{\alpha}-\htheta_n)
  \rightarrow \Nor\left(\0,\I\right),
  \end{equation*}
in conditional distribution, where $V_{n,R}^{\alpha}(\htheta_n)=\ddM_n^{-1}(\htheta_n)\Lambda_R^{\alpha}(\htheta_n)\ddM_n^{-1}(\htheta_n)$,
  \begin{align*}
    \Lambda_R^{\alpha}(\htheta_n)
    &=\frac{1}{n^2}\sumn\frac{\dm(Z_i,\htheta_n)\dm\tp(Z_i,\htheta_n)}
      {\pi_{n,R\alpha i}^{\mvc}(\htheta_n)},
      \quad\text{and}\quad
    \pi_{n,R\alpha i}^{\mvc}(\htheta_n)
    =(1-\alpha)\pi_{n,Ri}^{\mvc}(\htheta_n)+\alpha\onen.
\end{align*}
\end{theorem}

\begin{theorem}\label{thm:6}
  For $\ttheta_{s_n,P}^{\alpha}$ obtained from Algorithm~\ref{alg:4}, %
  under Assumptions~\ref{ass:1}-\ref{ass:4}, as $s_0$, $s_n$, and $n$ get large, if $s_0=o(n)$, $\varrho_n=s_n/(bn)\rightarrow\varrho\in[0,1)$, and the distribution of $Z$ is continuous, the following result hold. If $\varrho=0$, then given $\Dn$ and the pilot estimates in probability,  %
\begin{equation*}
  \sqrt{s_n}\{V_{n,P}^{\alpha}(\htheta_n)\}^{-1/2}(\ttheta_{s_n,P}^{\alpha}-\htheta_n)
  \rightarrow \Nor\left(\0,\I\right),
  \end{equation*}
in conditional distribution, where $V_{n,P}^{\alpha}(\htheta_n)=\ddM_n^{-1}(\htheta_n)\Lambda_{n,P}^{\alpha}(\htheta_n)\ddM_n^{-1}(\htheta_n)$,
\begin{align}\label{eq:30}
    \Lambda_{n,P}^{\alpha}(\htheta_n)
    &=\frac{1}{n^2}\sumn\frac{\{1-s_n\pi_{n,P\alpha i}^{\mvc}(\htheta_n)\}
      \dm(Z_i,\htheta_n)\dm\tp(Z_i,\htheta_n)}
      {\pi_{n,P\alpha i}^{\mvc}(\htheta_n)}
\end{align}
and $\pi_{n,P\alpha i}^{\mvc}(\htheta_n) =(1-\alpha)\pi_{n,Pi}^{\mvc}(\htheta_n)+\alpha\onen$.
If $\varrho>0$, then $\pi_{n,Pi}^{\mvc}(\htheta_n)$ in \eqref{eq:30} is replaced by $\pi_{n,Pi}^{\mvc}
  =\frac{\|\dm(Z_i,\htheta_n)\|\wedge H_{\varrho_n}}
  {\sumjn\{\|\dm(Z_j,\htheta_n)\|\wedge H_{\varrho_n}\}}$, where $H_{\varrho_n}$ is the $\varrho_n$-th upper sample quantile of $\|\dm(Z_i,\htheta_n)\|$'s for $i=1, ..., n$.
\end{theorem}
\begin{remark}\normalfont\label{remark8}
  Denote $\Lambda_R^\opt(\htheta_n)$ and $\Lambda_{n,P}^\opt(\htheta_n)$ as  $\Lambda_{n,R}(\htheta_n)$ and $\Lambda_{n,P}(\htheta_n)$ with optimal subsampling probabilities that produce the minimum trace values, respectively.
  In Theorems~\ref{thm:5} and \ref{thm:6}, $\Lambda_R^{\alpha}(\htheta_n)$ and $\Lambda_{n,P}^{\alpha}(\htheta_n)$ are different from $\Lambda_R^\opt(\htheta_n)$ and $\Lambda_{n,P}^\opt(\htheta_n)$, respectively. However, it can be shown that
{\small\begin{equation*}
  \tr\{\Lambda_R^\opt(\htheta_n)\}
  < \tr\{\Lambda_R^{\alpha}(\htheta_n)\}
  < \frac{\tr\{\Lambda_R^\opt(\htheta_n)\}}{1-\alpha},
  \;\text{ and }\;\;
  \tr\{\Lambda_{n,P}^\opt(\htheta_n)\}
  <\tr\{\Lambda_{n,P}^{\alpha}(\htheta_n)\}
  < \frac{\tr\{\Lambda_{n,P}^\opt(\htheta_n)\}}{1-\alpha}.
\end{equation*}}
Thus, if $\alpha$ is small enough, $\tr\{\Lambda_R^{\alpha}(\htheta_n)\}$ and $\tr\{\Lambda_R^\opt(\htheta_n)\}$ can be arbitrarily close, and $\tr\{\Lambda_{n,P}^{\alpha}(\htheta_n)\}$ and $\tr\{\Lambda_{n,P}^\opt(\htheta_n)\}$ can be arbitrarily close.
\end{remark}

\begin{remark}\normalfont
    If the pilot subsample size is much smaller than the approximated optimal subsample size, i.e., $s_0=o(s_n)$, then the aggragated estimator $\ctheta_R$ and $\ctheta_P$ have the same asymptotic distributions as those for $\ttheta_{s_n,R}^{\alpha}$ and $\ttheta_{s_n,P}^{\alpha}$, respectively.
\end{remark}

\section{Numerical experiments}
\label{sec:numer-exper}
In this section, we use numerical examples to compare the optimal subsampling probabilities under the two sampling procedures considered in this paper. We will also use numerical experiments to evaluate the performance of the practical algorithms proposed in Section~\ref{sec:pract-impl}.
\subsection{Comparisons of optimal subsampling probabilities}
In this section, we use numerical examples to compare the optimal probabilities for subsampling with replacement presented in Theorem~\ref{thm:3} with the optimal Poisson subsampling probabilities presented in Theorem~\ref{thm:4}.

\begin{example}[{\bf Linear regression}]\normalfont
  Consider solving the OLS for a linear regression model $y_i=\theta_{0}+\x_i\tp\btheta_{1}+\varepsilon_i$, $i=1, ..., n$, with $n=10^5$, $\theta_0=1$, $\btheta_1$ being a 50 dimensional vector of ones, and $\varepsilon_i$ being i.i.d. $\Nor(0,1)$. For the expected subsample sizes, we consider $s_n=2\times10^3, 3\times10^3, 5\times10^3, 10^4, 2\times10^4,$ and $5\times10^4$, so that the subsampling ratios are $s_n/n=0.02, 0.03, 0.05, 0.1, 0.2,$ and $0.5$. In this example, we use the L-optimality criterion with $L=(\X\tp\X)^{1/2}$ so that the optimal subsampling probabilities are closely related to the statistical leverage scores. Specially, $\pi_{n,Ri}^{\mvc}\propto|\hat\varepsilon_i|\sqrt{h_i}$ and $\pi_{n,Pi}^{\mvc}\propto(|\hat\varepsilon_i|\sqrt{h_i})\wedge H$ for the two subsampling procedures, respectively. To generate $\x_i$'s, we used normal distribution $\Nor(\0, \bSigma)$ and multivariate $t$ distributions $t_{\nu}(\0, \bSigma)$ with degrees of freedom $v=5, 4, 3, 2$, and $1$, where $\bSigma$ is a matrix with the $(i,j)$-th element being $0.5^{I(i\neq j)}$ and $I()$ being the indicator function. For this sequence of covariate distributions, the statistical leverage scores become more and more nonuniform.

  Table~\ref{tab01} gives the values of $g$ in the expression of the optimal Poisson subsampling probabilities in Theorem~\ref{thm:4} for different combinations of the subsampling ratio $s_n/n$ and covariate distribution. Note that $g$ is the number of cases that $\|\dm(Z_i,\htheta_n)\|\propto|\hat\varepsilon_i|\sqrt{h_i}$ are truncated by $H$. Thus, $\pi_{n,Ri}^{\mvc}$ and $\pi_{n,Pi}^{\mvc}$ are more different for larger values of $g$, and they are identical if $g=0$. It is clear that $g$ increases as $s_n/n$ increases, indicating that $\pi_{n,Ri}^{\mvc}$ and $\pi_{n,Pi}^{\mvc}$ are more different as the subsampling ratio $s_n/n$ gets larger. We also see that as the tail of the covariate distribution get heavier, $g$ gets larger. This tells us that the difference between $\pi_{n,Ri}^{\mvc}$ and $\pi_{n,Pi}^{\mvc}$ is more significant if the statistical leverage scores are more nonuniform, as a heavier-tailed covariate distribution leads to more nonuniform leverage scores. 

\begin{table}[H]%
  \caption{The values of $g$ in the optimal Poisson subsampling probabilities for OLS with different expected subsample sizes $s_n$ and different distributions of $\x_i$'s. %
    The full data sample size is $n=10^5$.}
\label{tab01}
  \centering
\begin{tabular}{crrrrrr}\hline
  &\multicolumn{5}{c}{Distribution of $\x_i$'s} \\
  \cline{2-7}
$s_n/n$ & Normal & $t_5$ & $t_4$ & $t_3$ & $t_2$ & $t_1$ \\ \hline
0.02 & 0     & 0     & 0     & 0     & 16    & 120   \\ 
0.03 & 0     & 0     & 0     & 7     & 39    & 203   \\ 
0.05 & 0     & 0     & 1     & 28    & 113   & 342   \\ 
0.1  & 0     & 23    & 58    & 154   & 492   & 756   \\ 
0.2  & 15    & 584   & 762   & 1216  & 2242  & 1734  \\ 
0.5  & 14364 & 16569 & 17191 & 17954 & 19038 & 15481 \\ \hline
\end{tabular}
\end{table}
\end{example}

Figure~\ref{fig:a1} presents histograms and scatter plots of optimal probabilities for the two subsampling procedures to show more details on the distributions of $\pi_{n,Ri}^{\mvc}$'s and $\pi_{n,Pi}^{\mvc}$'s when $\x_i$'s are from the $t_3$ distribution. In each sub-figure, the left panel is the histogram for $\pi_{n,Pi}^{\mvc}$'s and the right panel is the scatter plot of $\pi_{n,Pi}^{\mvc}$'s against $\pi_{n,Ri}^{\mvc}$'s. We multiply all probabilities by $n$ for better presentations. Note that this does not change the shapes of the figures. We only create the histogram for $\pi_{n,Pi}^{\mvc}$'s, because the distribution of $\pi_{n,Ri}^{\mvc}$'s does not depend on $s_n$ and remains the same for all values of $s_n$. In addition, since $g=0$ for the case with $s_n/n=0.02$, the histogram in Figure~\ref{fig:a1}(a) is the same to the histogram for $\pi_{n,Ri}^{\mvc}$ and we can compare it with other histograms to see the difference between the distributions of $\pi_{n,Ri}^{\mvc}$'s and $\pi_{n,Pi}^{\mvc}$'s. From Figure~\ref{fig:a1} (a)-(f), we see that as $s_n/n$ increases the optimal probabilities for Poisson sampling and sampling with replacement are more different, because more larger $\pi_{n,Pi}^{\mvc}$'s are truncated to $1/s$.

\begin{figure}[H]
  \centering
  \begin{subfigure}{0.49\textwidth}
    \includegraphics[width=\textwidth,page=1]{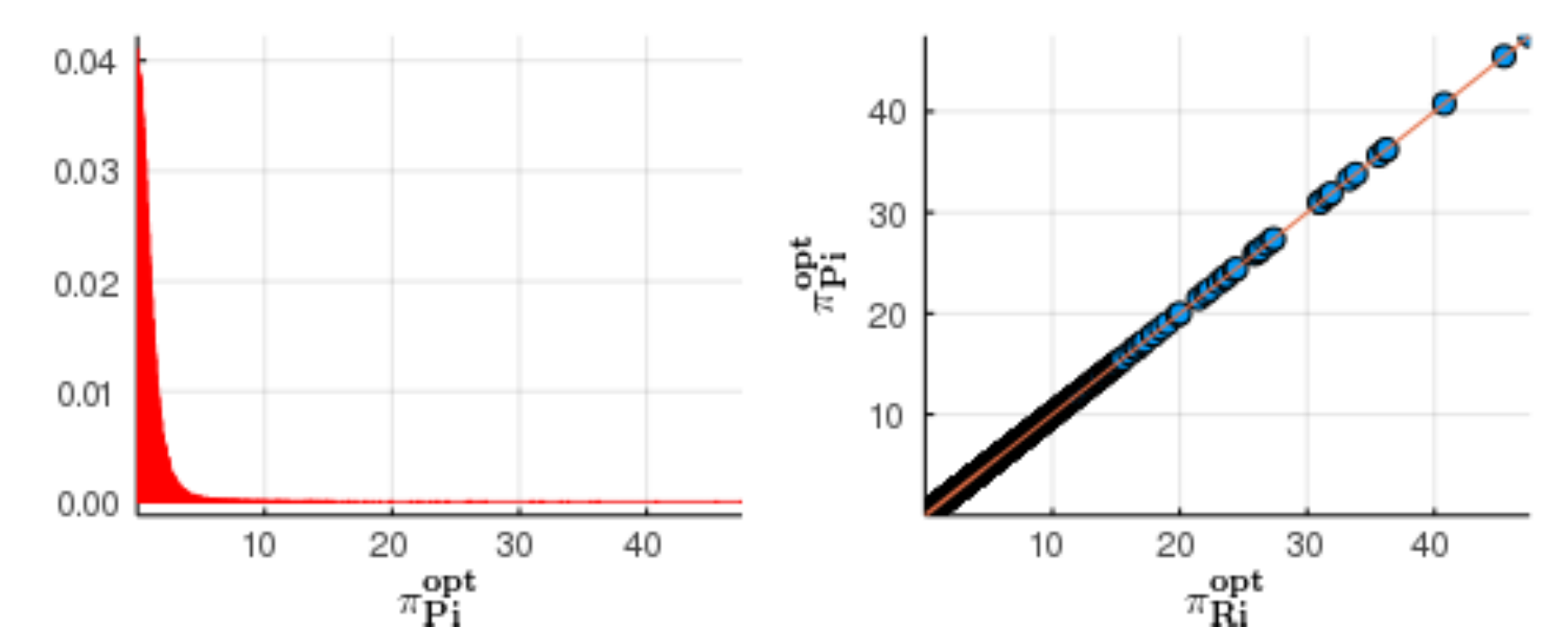}
    \caption{$s_n/n=0.02$}
  \end{subfigure}
  \begin{subfigure}{0.49\textwidth}
    \includegraphics[width=\textwidth,page=2]{0PIsChange4.pdf}
    \caption{$s_n/n=0.03$}
  \end{subfigure}
  \begin{subfigure}{0.49\textwidth}
    \includegraphics[width=\textwidth,page=3]{0PIsChange4.pdf}
    \caption{$s_n/n=0.05$}
  \end{subfigure}
  \begin{subfigure}{0.49\textwidth}
    \includegraphics[width=\textwidth,page=4]{0PIsChange4.pdf}
    \caption{$s_n/n=0.1$}
  \end{subfigure}
  \begin{subfigure}{0.49\textwidth}
    \includegraphics[width=\textwidth,page=5]{0PIsChange4.pdf}
    \caption{$s_n/n=0.2$}
  \end{subfigure}
  \begin{subfigure}{0.49\textwidth}
    \includegraphics[width=\textwidth,page=6]{0PIsChange4.pdf}
    \caption{$s_n/n=0.5$}
  \end{subfigure}
  \caption{Histograms and scatter plots of optimal probabilities for subsampling with replacement and Poisson subsampling for different subsampling ratio $s_n/n$. Here $\x_i$'s are from the $t_3$ distribution.}
  \label{fig:a1}
\end{figure}

Figure~\ref{fig:a2} presents histograms and scatter plots of optimal probabilities $\pi_{n,Ri}^{\mvc}$'s and $\pi_{n,Pi}^{\mvc}$'s for different distributions of $\x_i$'s when $s_n/n=0.1$. In each sub-figure, the upper and lower plots in the left panel are the histograms for $\pi_{n,Ri}^{\mvc}$'s and $\pi_{n,Pi}^{\mvc}$'s, respectively, and the right panel is the scatter plot of $\pi_{n,Pi}^{\mvc}$'s against $\pi_{n,Ri}^{\mvc}$'s. Again, we multiply all probabilities by $n$ for better presentations. We see that for a fixed subsampling ratio $s_n/n$, $\pi_{n,Ri}^{\mvc}$'s and $\pi_{n,Pi}^{\mvc}$'s become more different as the leverage scores become more nonuniform (the tail of the covariate distribution becomes heavier), because more large values of $\pi_{n,Pi}^{\mvc}$'s are truncated. 
  
\begin{figure}[H]
  \centering
  \begin{subfigure}{0.49\textwidth}
    \includegraphics[width=\textwidth,page=1]{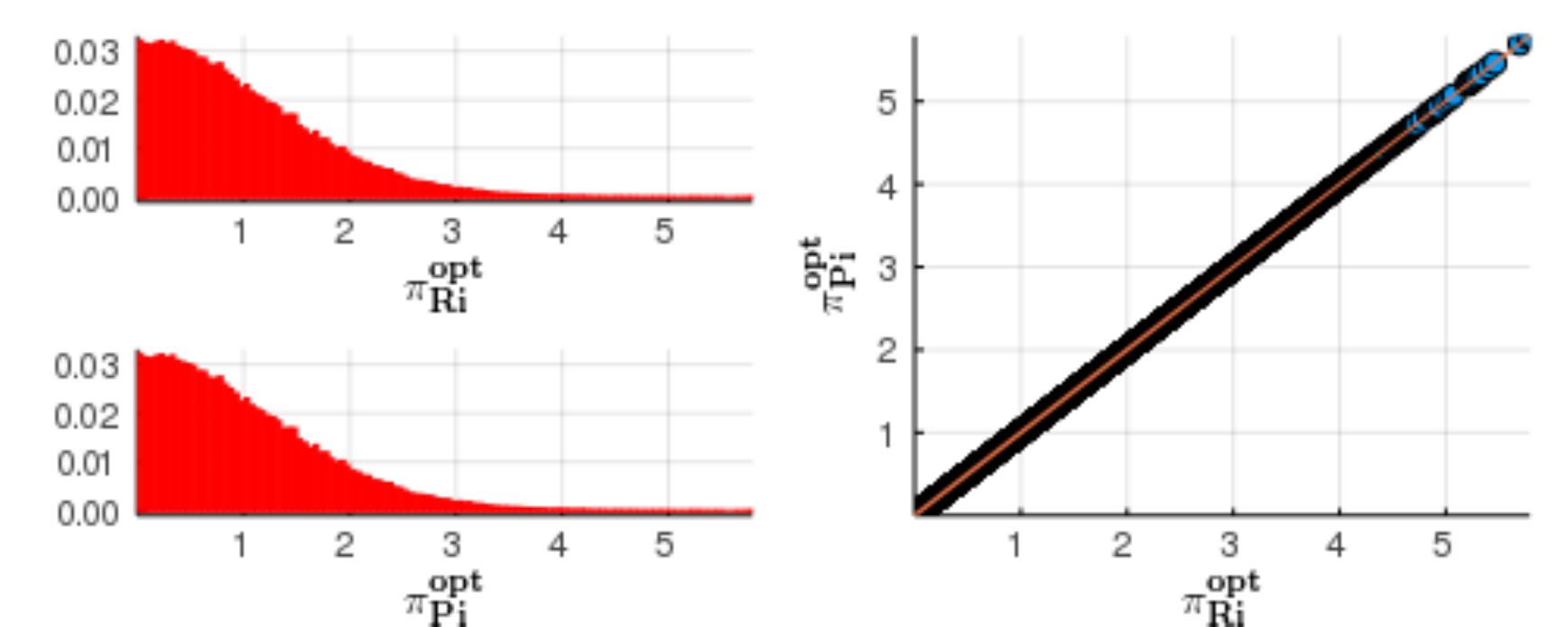}
    \caption{Normal}
  \end{subfigure}
  \begin{subfigure}{0.49\textwidth}
    \includegraphics[width=\textwidth,page=2]{0PIDiffCase-k10000.pdf}
    \caption{$t_5$}
  \end{subfigure}
  \begin{subfigure}{0.49\textwidth}
    \includegraphics[width=\textwidth,page=3]{0PIDiffCase-k10000.pdf}
    \caption{$t_4$}
  \end{subfigure}
  \begin{subfigure}{0.49\textwidth}
    \includegraphics[width=\textwidth,page=4]{0PIDiffCase-k10000.pdf}
    \caption{$t_3$}
  \end{subfigure}
  \begin{subfigure}{0.49\textwidth}
    \includegraphics[width=\textwidth,page=5]{0PIDiffCase-k10000.pdf}
    \caption{$t_2$}
  \end{subfigure}
  \begin{subfigure}{0.49\textwidth}
    \includegraphics[width=\textwidth,page=6]{0PIDiffCase-k10000.pdf}
    \caption{$t_1$}
  \end{subfigure}
  \caption{Histograms and scatter plots of optimal probabilities for subsampling with replacement and Poisson subsampling for different covariate distributions. Here the subsampling ratio $s_n/n=0.1$.}
  \label{fig:a2}
\end{figure}

\subsection{Comparison of estimation efficiency of the practical algorithms}
We compare the estimation efficiency for the two subsampling procedures using both synthetic and real data sets.
\begin{example}[{\bf Logistic regression}]\normalfont Form model $\Pr(y_i=1|\x_i)%
={e^{\theta_{0}+\x_i\tp\btheta_{1}}}/{(1+e^{\theta_{0}+\x_i\tp\btheta_{1}})}$, 
$i=1, ..., n$, we generate synthetic data sets by setting $n=10^5$, $\theta_0=0.5$, and $\btheta_1$ to be a 9 dimensional vector of 0.5. We consider the following three cases to generate $\x_i$. In Cases 1 and 3, the responses $y_i$ are balanced, while in Case 2 about 98\% of the data points are with $y_i=1$.
\begin{enumerate}[{Case }1:, itemindent=0.7cm, leftmargin=*, itemsep=-4pt]
\item \textbf{Normal}. Generate $\x_i$ from a multivariate normal distribution, $\Nor(\0, \bSigma)$, where the $(i,j)$-th element of $\bSigma$ is $\Sigma_{ij}=0.5^{I(i\neq j)}$ and $I()$ is the indicator function. This distribution is symmetric with light tails. %
\item \textbf{LogNormal}. %
  Generate $\bv_i$ from $\Nor(\0, \bSigma)$ as defined in Case 1 and then set $\x_i=e^{\bv_i}$, where the exponentiation is element-wise. This distribution is asymmetric and positively skewed.
\item $\bm{T}_3$. We generate $\x_i$ from a multivariate $t$ distribution with three degrees of freedom $t_3(\0, \bSigma)$ with $\bSigma$ defined in Case 1. This distribution is symmetric with heavy tails.
\end{enumerate}

We also consider two real data sets: the covtype data from the LIBSVM data website (\url{https://www.csie.ntu.edu.tw/~cjlin/libsvm/})
and the SUSY data \citep{Baldi2014}.
Both data sets are also available from the UCI data repository \citep{Dua:2017}. We present them as Cases 4 and 5 below.
\begin{enumerate}[{Case }1:, itemindent=0.7cm, leftmargin=*, itemsep=-4pt]
  \setcounter{enumi}{3}
\item \textbf{Covtype Data}. It has $n=581,012$ observations with about 48.76\% of the responses are $y_i=1$. We use the ten quantitative covariate variables as $\x_i$'s. %
  
\item \textbf{SUSY Data}. It has $n=5,000,000$ observations with about 54.24\% of the responses are $y_i=1$. We use the 18 kinematic features to classify whether new SUSY particles are produced. 
\end{enumerate}

To implement Algorithms~\ref{alg:3} and \ref{alg:4}, we set $\alpha=0.1$, and choose $s_0=0.01n$ and different values for $s_n$
so that the sampling ratio $(s_0+s_n)/n= 0.02$, $0.05$, $0.1$, $0.2$, and $0.5$. 
Two different options of $H^{0*}$ are considered: $H^{0*}=\|\dm(Z_i^{0*},\ttheta_{s_0,P}^{0*})\|_{\frac{s_n}{5n}}$ and $H^{0*}=\infty$. We aggregate the pilot estimator with the approximated optimal subsampling estimator using the procedure described in Remark~\ref{rmk:6}.
For comparison, we also implement the uniform subsampling method with expected subsample sizes $s_0+s_n$. Newton's method is used for optimization on all subsamples. 
We repeat the simulation for $T=1000$ times to calculate the empirical mean squared error (MSE), defined as MSE=$\frac{1}{T}\sum_{t=1}^T\|\ctheta^{(t)}-\htheta_n\|^2$, where $\ctheta^{(t)}$ is the subsampling estimate at the $t$-th repetition and $\htheta_n$ is the full data estimate.

Figure~\ref{fig:1} plots the empirical MSE (natural logarithm is taken for better presentation) against the subsampling ratio $(s_0+s_n)/n$. %
When the subsampling ratio $(s_0+s_n)/n$ is close to zero, subsampling with replacement and Poisson subsampling have similar performance for both approximated optimal subsampling and uniform subsampling.
However, when $(s_0+s_n)/n$ gets larger, Poisson subsampling outperforms subsampling with replacement, and the improvement from subsampling with replacement to Poisson subsampling is more significant for approximated optimal subsampling than for uniform subsampling. For both subsampling with replacement and Poisson subsampling, approximated optimal subsampling methods outperform the uniform subsampling method. Their performances are closer for smaller $(s_0+s_n)/n$ because the proportions of uniform subsamples are higher for smaller $(s_0+s_n)/n$. For Poisson subsampling, the results for the two choices of $H^{0*}$, $H^{0*}=\infty$ and $H^{0*}=\|\dm(Z_i^{0*},\ttheta_{s_0,P}^{0*})\|_{\frac{s_n}{5n}}$,
are similar when $(s_0+s_n)/n$ is small, but they start to differ for larger $(s_0+s_n)/n$.

\begin{figure}[H]
  \centering
  \begin{subfigure}{0.32\textwidth}
    \includegraphics[width=\textwidth,page=1]{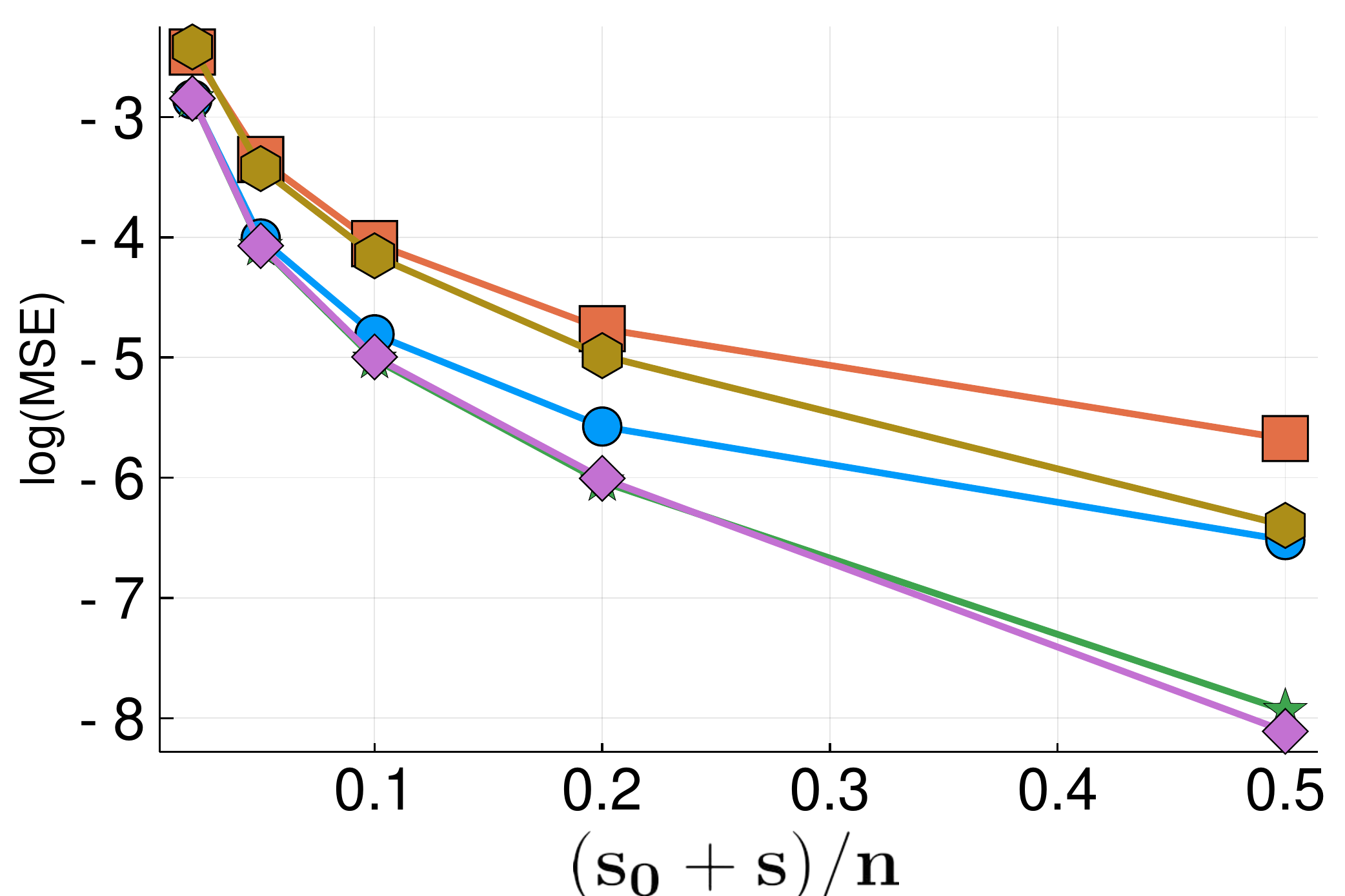}
    \caption{Case 1: \textbf{Normal}}
  \end{subfigure}
  \begin{subfigure}{0.32\textwidth}
    \includegraphics[width=\textwidth,page=2]{mse_logistic.pdf}
    \caption{Case 2: \textbf{LogNormal}}
  \end{subfigure}
  \begin{subfigure}{0.32\textwidth}
    \includegraphics[width=\textwidth,page=3]{mse_logistic.pdf}
    \caption{Case 3: $\bm{T}_3$}
  \end{subfigure}
  \begin{subfigure}{0.32\textwidth}
    \includegraphics[width=\textwidth,page=4]{mse_logistic.pdf}
    \caption{Case 4: \textbf{Covtype Data}}
  \end{subfigure}
  \begin{subfigure}{0.32\textwidth}
    \includegraphics[width=\textwidth,page=5]{mse_logistic.pdf}
    \caption{Case 5: \textbf{SUSY Data}}
  \end{subfigure}
  \begin{subfigure}{0.32\textwidth}
    \includegraphics[width=\textwidth,page=6]{mse_logistic.pdf}
    \caption*{labels}
  \end{subfigure}
\caption{Log empirical MSEs (y-axis) against subsampling ratio $(s_0+s_n)/n$ (x-axis) for logistic regression. Here, ``optR'' means optimal subsampling with replacement; ``uniR'' means uniform subsampling with replacement; ``optP$_\infty$'' means approximated optimal Poisson subsampling with $H^{0*}=\infty$; ``optP$_{b=5}$'' means approximated optimal Poisson subsampling with $H^{0*}=\|\dm(Z_i^{0*},\ttheta_{s_0,P}^{0*})\|_{\frac{s_n}{5n}}$; and``uniP'' means uniform Poisson subsampling.}
  \label{fig:1}
\end{figure}

\end{example}

\begin{example}[{\bf Linear regression}]\normalfont
We consider a linear model $y_i=\theta_{0}+\x_i\tp\btheta_{1}+\varepsilon_i$, $i=1, ..., n$, with $n=10^5$, $\theta_0=1$, $\btheta_1$ being a 50 dimensional vector of ones, and $\varepsilon_i$ being i.i.d. $\Nor(0,1)$. We use the same distributions in Cases 1-3 to generate $\x_i$ and refer them as Cases 1'-3'. 
We also consider a gas sensor data \cite{fonollosa2015reservoir} from 
the UCI data repository \citep{Dua:2017}.
We present it as Case 6 below.
\begin{enumerate}[{Case }1:, itemindent=0.7cm, leftmargin=*]
  \setcounter{enumi}{5}
\item \textbf{Gas Sensor Data}. After cleaning, the data contain $n=4,188,261$ readings on 15 sensors. We use log of readings from the last sensor as responses and log of other readings as covariates.
\end{enumerate}

To implement Algorithms~\ref{alg:3} and \ref{alg:4}, we use the same setup for $\alpha$, $s_0$, $s_n$, and $H^{0*}$, as used in logistic regression. Specifically, $\alpha=0.1$, $s_0=0.01n$ and different values for $s_n$ so that $(s_0+s_n)/n= 0.02$, $0.05$, $0.1$, $0.2$, and $0.5$. 
  We also consider both $H^{0*}=\|\dm(Z_i^{0*},\ttheta_{s_0,P}^{0*})\|_{\frac{s_n}{5n}}$ and $H^{0*}=\infty$, and aggregate the pilot estimator with the approximated optimal subsampling estimator using the procedure described in Remark~\ref{rmk:6}.
We repeat the simulation for $T=1000$ times to calculate the empirical MSE.  %

Figure~\ref{fig:3} gives results for empirical MSE from least-squares in linear regression model. %
The overall pattern in Figure~\ref{fig:3} is similar to that in Figure~\ref{fig:1}. We see that subsampling with replacement and Poisson subsampling have similar performance if the subsampling ratio $(s_0+s_n)/n$ is close to zero, while Poisson subsampling outperforms subsampling with replacement as $(s_0+s_n)/n$ gets larger. This trend is true for both approximated optimal subsampling and uniform subsampling, and we observe that the advantage of Poisson subsmapling over subsampling with replacement is more significant for  approximated optimal subsampling. Furthermore, for linear regression, the advantage of Poisson subsampling compared with subsampling with replacement is more significant. For example, in Case 4', the synthetic data sets with $\x_i$'s from the $t_3$ distribution, the uniform Poisson subsampling can even outperform the approximated optimal subsampling with replacement when $(s_0+s_n)/n=0.5$. 
We also observe that approximated optimal subsampling methods outperform the uniform subsampling methods, and the gap between their performance in terms of estimation efficiency is larger for larger $(s_0+s_n)/n$. This is because the proportions of more informative  observations in the subsample are higher for larger $(s_0+s_n)/n$. Another pattern is that when the approximated optimal subsampling probabilities are more nonuniform, their advantage over uniform subsampling is more significant. For example, from the gas sensor data set, approximated optimal subsampling methods have significantly higher estimation efficiency than the uniform subsampling methods even when $s_0=s_n=1000$. 
For Poisson subsampling, the performance with $H^{0*}=\infty$  and that with $H^{0*}=\|\dm(Z_i^{0*},\ttheta_{s_0,P}^{0*})\|_{\frac{s_n}{5n}}$
are similar for small $(s_0+s_n)/n$, but the choice with $H^{0*}=\|\dm(Z_i^{0*},\ttheta_{s_0,P}^{0*})\|_{\frac{s_n}{5n}}$ starts to show its advantage for larger $(s_0+s_n)/n$.

\begin{figure}[ht]
  \centering
  \begin{subfigure}{0.32\textwidth}
    \includegraphics[width=\textwidth,page=1]{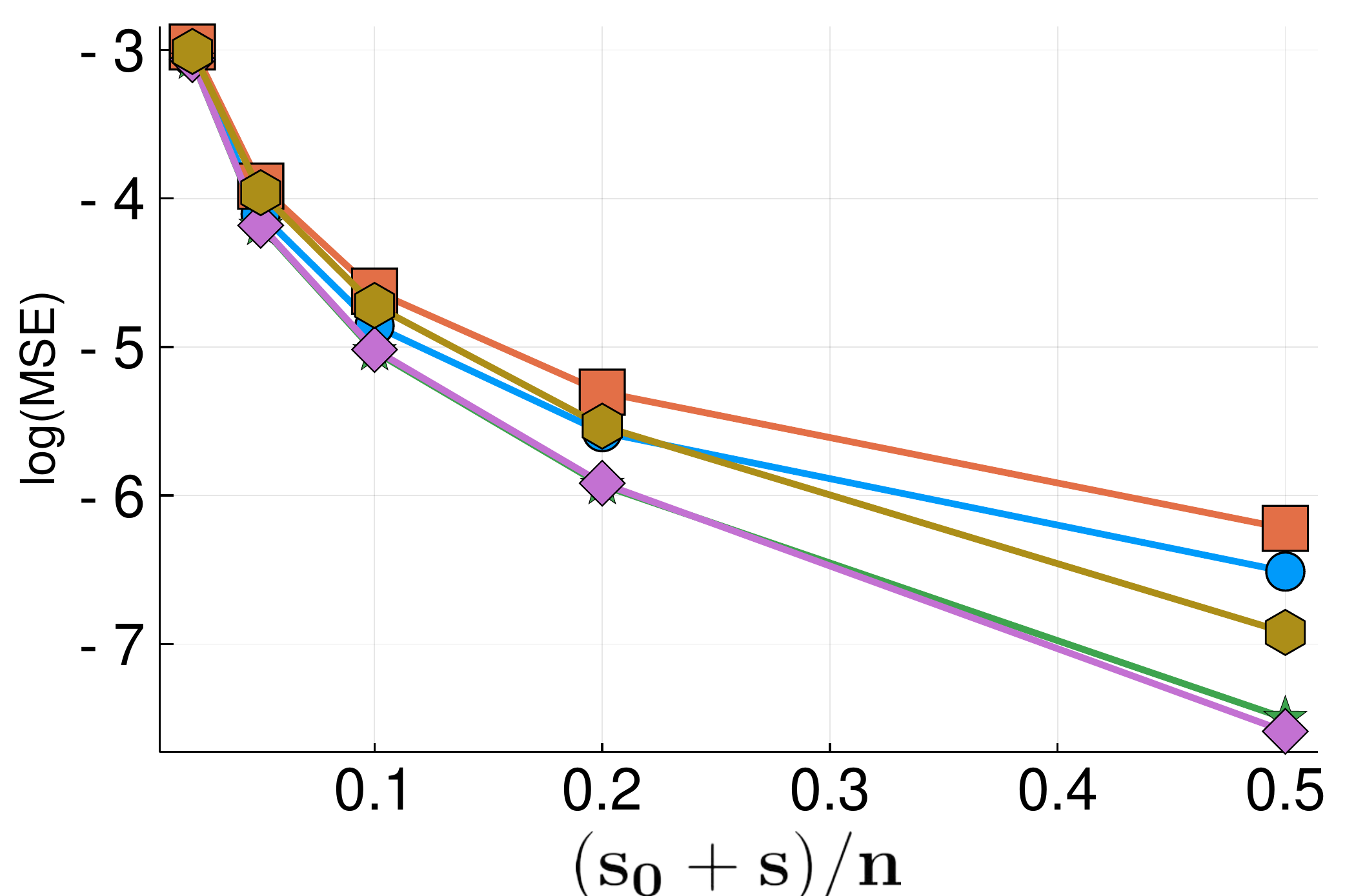}
    \caption{Case 1': \textbf{Normal}}
  \end{subfigure}
  \begin{subfigure}{0.32\textwidth}
    \includegraphics[width=\textwidth,page=2]{mse_linear.pdf}
    \caption{Case 2': \textbf{LogNormal}}
  \end{subfigure}
  \begin{subfigure}{0.32\textwidth}
    \includegraphics[width=\textwidth,page=4]{mse_linear.pdf}
    \caption{Case 3': $\bm{T}_3$}
  \end{subfigure}
  \begin{subfigure}{0.32\textwidth}
    \includegraphics[width=\textwidth,page=5]{mse_linear.pdf}
    \caption{Case 6: \textbf{Gas Sensor}}
  \end{subfigure}
  \begin{subfigure}{0.32\textwidth}
    \includegraphics[width=\textwidth,page=6]{mse_linear.pdf}
    \caption*{labels}
  \end{subfigure}
  \caption{Log Empirical MSEs (y-axis) against subsampling ratio $(s_0+s_n)/n$ (x-axis) for linear regression. Here, ``optR'' means optimal subsampling with replacement; ``uniR'' means uniform subsampling with replacement; ``optP$_\infty$'' means approximated optimal Poisson subsampling with $H^{0*}=\infty$; ``optP$_{b=5}$'' means approximated optimal Poisson subsampling with $H^{0*}=\|\dm(Z_i^{0*},\ttheta_{s_0,P}^{0*})\|_{\frac{s_n}{5n}}$; and``uniP'' means uniform Poisson subsampling.}
  \label{fig:3}
\end{figure}
\end{example}

\section{Conclusion and Discussion}
In this paper, we derived optimal subsampling probabilities in the context of maximizing an additive target function for both subsampling with replacement and Poisson subsampling. Theoretical and empirical results show that the two different subsampling procedure have similar performance when the subsampling ratio is small. However, when subsampling ratio does not converge to zero, Poisson subsampling has a higher estimation efficiency. 
One problem warrants for further investigation is how to chose the tuning parameter $b$ in Algorithm~\ref{alg:4} so that the approximated optimal subsampling probabilities produce an estimator with an asymptotic variance-covariance matrix that is near optimal even when the subsampling ratio does not converge to zero.

\section*{Acknowledgments}
The authors are deeply grateful to Professor Michael Mahoney, the Associate Editor Professor Stephane Boucheron, and two anonymous reviewers for their insightful comments, question, and suggestions that significantly improved the manuscript.
  This work was partially supported by the NSF grant CCF-2105571.  %

An early short version of the paper is presented in the AISTATS 2021 conference \citep{pmlr-v130-wang21a}. This version contains substantially more technical results such as the relationships between conditional and unconditional convergences, and the unconditional asymptotic distributions about the true parameter. It also contains additional examples and numerical comparison results.      
 \black

\newpage\appendix
\begin{center} \LARGE Appendix\\[1cm] \end{center}
\setcounter{equation}{0}
\renewcommand{\theequation}{A.\arabic{equation}}
\renewcommand{\thesection}{A.\arabic{section}}
\renewcommand{\thefigure}{A.\arabic{figure}}

\section{Proofs}
\addcontentsline{toc}{section}{Appendix: Proofs}
\label{appendix}

In this section, we prove all the theoretical results in the paper.

\subsection{Proof of Proposition~\ref{prop2}}
\begin{proof}
For Proposition~\ref{prop2}~\ref{item:1} since $\Pr(\|\Delta_{n,s_n}\|>\delta|\Dn)$ is a nonnegative and bounded random variable, from Theorem 1.3.6 of \cite{Serfling1980},  $\Pr(\|\Delta_{n,s_n}\|>\delta|\Dn)=\op$ if and only if $\Exp\{\Pr(\|\Delta_{n,s_n}\|>\delta|\Dn)\}\rightarrow0$. Note that 
  \begin{equation*}
    \Exp\{\Pr(\|\Delta_{n,s_n}\|>\delta|\Dn)\}
    =\Exp[\Exp\{I(\|\Delta_{n,s_n}\|>\delta)|\Dn\}]
    =\Exp\{I(\|\Delta_{n,s_n}\|>\delta)\}
    =\Pr(\|\Delta_{n,s_n}\|>\delta).
  \end{equation*}
  Thus $\Exp\{\Pr(\|\Delta_{n,s_n}\|>\delta|\Dn)\}\rightarrow0$ if and only if $\Pr(\|\Delta_{n,s_n}\|>\delta)\rightarrow0$, which is true if and only if $\Delta_{n,s_n}=\op$. 
  
Now we prove Proposition~\ref{prop2}~\ref{item:2}. Note that $\Delta_{n,s_n}=\OpD(1)$ means that for any $\epsilon>0$ and any $\delta>0$, there exist a finite $K_{\epsilon}>0$ and a finite $N_{\epsilon,\delta}>0$ such that $\Pr\{\Pr(\|\Delta_{n,s_n}\|>K_{\epsilon}|\Dn) > \epsilon\} < \delta$ for all $s_n>N_{\epsilon,\delta}$ and $n>N_{\epsilon,\delta}$.
Thus if $\Delta_{n,s_n}=\OpD(1)$, then for any $\epsilon>0$ and $\delta=\epsilon$, there exist a finite $K_{0.5\epsilon}>0$ and a finite $N_{0.5\epsilon,0.5\epsilon}>0$ such that for all $s_n>N_{0.5\epsilon,0.5\epsilon}$ and $n>N_{0.5\epsilon,0.5\epsilon}$,
\begin{equation*}
  \Pr\{\Pr(\|\Delta_{n,s_n}\|>K_{0.5\epsilon}|\Dn) > 0.5\epsilon\} < 0.5\epsilon.
\end{equation*}
Therefore,
\begin{align*}
  \Pr(\|\Delta_{n,s_n}\|>K_{0.5\epsilon})
  &=\Exp\big\{\Pr(\|\Delta_{n,s_n}\|>K_{0.5\epsilon}|\Dn)\big\}\\
  &\le\Exp\big[\Pr(\|\Delta_{n,s_n}\|>K_{0.5\epsilon}|\Dn)
    I\{\Pr(\|\Delta_{n,s_n}\|>K_{0.5\epsilon}|\Dn)>0.5\epsilon\}\big]+0.5\epsilon\\
  &\le\Exp\big[I\{\Pr(\|\Delta_{n,s_n}\|>K_{0.5\epsilon}|\Dn)>0.5\epsilon\}\big]
    +0.5\epsilon\\
  &\le\Pr\{\Pr(\|\Delta_{n,s_n}\|>K_{0.5\epsilon}|\Dn)>0.5\epsilon\}+0.5\epsilon\\
  &\le\epsilon,
\end{align*}
meaning that $\Delta_{n,s_n}=\Op$. 

On the other hand, if $\Delta_{n,s_n}=\Op$, then for any $\epsilon>0$ and $\delta>0$, there exist a finite $K_{\delta\epsilon}$ and a finite $N_{\delta\epsilon}$ such that
$\Pr(\|\Delta_{n,s_n}\|>K_{\delta\epsilon}) \le \delta\epsilon$ for all $s_n>N_{\delta\epsilon}$ and $n>N_{\delta\epsilon}$. Thus
\begin{equation*}
  \Pr\{\Pr(\|\Delta_{n,s_n}\|>K_{\delta\epsilon}|\Dn) > \epsilon\}
  \le\epsilon^{-1}\Exp\big\{\Pr(\|\Delta_{n,s_n}\|>K_{\delta\epsilon}|\Dn)\big\}
  =\epsilon^{-1}\Pr(\|\Delta_{n,s_n}\|>K_{\delta\epsilon}) < \delta,
\end{equation*}
which means that $\Delta_{n,s_n}=\OpD(1)$. 

Now we prove Proposition~\ref{prop2}~\ref{item:3}. Because $\Pr(\Delta_{n,s_n}\le\x|\Dn)$ is bounded, $\Pr(\Delta_{n,s_n}\le\x|\Dn)-\Pr(U\le\x)=\op$ if and only if  $\Exp\{\Pr(\Delta_{n,s_n}\le\x|\Dn)\}-\Pr(U\le\x)=\Pr(\Delta_{n,s_n}\le\x)-\Pr(U\le\x)=o(1)$.  Thus $\Delta_{n,s_n}\cvdD U$ if and only if $\Delta_{n,s_n}\cvd U$. 
\end{proof}
\black

\subsection{Proof for Theorem~\ref{thm:1}}
\addcontentsline{toc}{subsection}{Appendix A: Proof for Theorem~\ref{thm:1}}
Recall that
\begin{align*}
M_n(\btheta)&=\onen\sumn m(Z_i,\btheta).
\end{align*}
For the sampling with replacement estimator in~(\ref{eq:28}), %
let
\begin{align*}
M_{s_n}^*(\btheta)&=\oner\sumrr\frac{m(Z_i^*,\btheta)}{n\pi_{n,i}^*}.
\end{align*}
To prove Theorem~\ref{thm:1}, we first establish Lemma~\ref{lem:BrtoddM} and Lemma~\ref{lem2:DistributionofdM} in the following. 
\begin{lemma}\label{lem:BrtoddM}
Under Assumptions \ref{ass:3} and \ref{ass:6}, if $\|\ttheta_{s_n,R}-\htheta_n\|=\op$, then conditional on $\Dn$,
\begin{align*}
B_{s_n}- \ddM_n(\htheta_n)=\op,
\end{align*}
where
\begin{align*}
&\ddM_n(\htheta_n)=\onen\sumn\ddm(Z_i,\htheta_n),\\
&B_{s_n}=\int_{0}^{1}\oner \sumrr\frac{\ddm\{Z_i^*,\htheta_n+\lambda(\ttheta_{s_n,R}-\htheta_n) \}}{n\pi_{n,i}^*}\,\ud\lambda.%
\end{align*}
\end{lemma}
In Lemma~\ref{lem:BrtoddM}, the notation $o_P(1)$ means convergence to 0 in probability. Here the probability is conditional probability. From \cite{xiong2008some,cheng2010bootstrap}, a sequence converges to 0 in conditional probability is equivalent to the fact that it converges to 0 in unconditional probability. Thus we use $o_P(1)$ to indicate convergence to 0 either in unconditional or conditional probability.

\begin{proof}
Firstly, note that
\begin{align*}
  \Exp\left(\oner\sumrr\frac{\psi(Z_i^*)}{n\pi_{n,i}^*}\Big|\Dn\right)
  &=\onen\sumn\psi(Z_i)=\Exp\{\psi(Z)\}+\op, \quad\text{ and}\\
  \Var\left(\oner\sumrr\frac{\psi(Z_i^*)}{n\pi_{n,i}^*}\Big|\Dn\right)
  &=\oner\sumn\frac{\psi^2(Z_i)}{n^2\pi_{n,i}}
    \leq\max_{i=1,...,n}\left(\frac{1}{n\pi_{n,i}}\right)\frac{1}{s_nn}\sumn \psi^2(Z_i)
  =O_P(s_n^{-1}).
\end{align*}
Thus,
\begin{equation*}
  \oner\sumrr\frac{\psi(Z_i^*)}{n\pi_{n,i}^*}=O_{P|\Dn}(1).
\end{equation*}
For every $k,l=1,2,...,d$, from Lipschitz continuity, for $\lambda\in(0,1)$, we have
\begin{align}
&\left|\oner\sumrr\frac{\ddm_{k,l}\{Z_i^*,\htheta_n+\lambda(\ttheta_{s_n,R}-\htheta_n)\}}{n\pi_{n,i}^*}-
\oner\sumrr\frac{\ddm_{k,l}(Z_i^*,\htheta_n)}{n\pi_{n,i}^*} \right|\notag\\
&=\lambda\|\ttheta_{s_n,R}-\htheta_n\|\oner\sumrr\frac{\psi(Z_i^*)}{n\pi_{n,i}^*}
=o_P(1),\label{eq:29}
\end{align}
and for any fixed $\btheta$, we have
\begin{align}
  &\onen\sumn\ddm_{k,l}^2(Z_i,\htheta_n)\leq \frac{2}{n}\sumn\ddm_{k,l}^2(Z_i,\btheta)+\frac{2\|\htheta_n-\btheta\|^2}{n}\sumn\psi^2(Z_i)
    =O_P(1).
\label{eq:ddmOP1}
\end{align}
In addition, according to \eqref{eq:ddmOP1},
\begin{align*}
\Exp\left\{\oner\sumrr\frac{\ddm_{k,l}(Z_i^*,\htheta_n)}{n\pi_{n,i}^*}\Big|\Dn\right\}&=\onen\sumn\ddm_{k,l}(Z_i,\htheta_n),\\
\Var\left\{\oner\sumrr\frac{\ddm_{k,l}(Z_i^*,\htheta_n)}{n\pi_{n,i}^*}\Big|\Dn\right\}&\le\frac{1}{s_n}\sumn\frac{\ddm_{k,l}^2(Z_i,\htheta_n)}{n^2\pi_{n,i}}\\
&\leq\max_{i}\left(\frac{1}{n\pi_{n,i}}\right)\frac{1}{s_nn}\sumn \ddm_{k,l}^2(Z_i,\htheta_n)
=O_P(s_n^{-1}).
\end{align*}
Thus, by Chebyshev's inequality, we have
\begin{align}\label{eq:ddm}
\left|\oner\sumrr\frac{\ddm_{k,l}(Z_i^*,\htheta_n)}{n\pi_{n,i}^*}-\onen\sumn\ddm_{k,l}(Z_i,\htheta_n)\right|
&=O_{P|\Dn}(s_n^{-1/2})=\opD(1).
\end{align}

Combining \eqref{eq:29} and \eqref{eq:ddm}, we have
\begin{align*}
&\left|B_{s_n,k,l}-\onen\sumn\ddm_{k,l}(Z_i,\htheta_n)\right|\\
&\leq \int_{0}^{1}\left|\oner \sumrr\frac{\ddm_{k,l}\{Z_i^*,\htheta_n+\lambda(\ttheta_{s_n,R}-\htheta_n) \}}{n\pi_{n,i}^*}-\onen\sumn\ddm_{k,l}(Z_i,\htheta_n)\right|\,\ud\lambda\\
&\leq \int_{0}^{1}
\Bigg[\left|\oner\sumrr\frac{\ddm_{k,l}\{Z_i^*,\htheta_n+\lambda(\ttheta_{s_n,R}-\htheta_n)\}}{n\pi_{n,i}^*}-
\oner\sumrr\frac{\ddm_{k,l}(Z_i^*,\htheta_n)}{n\pi_{n,i}^*} \right|\Bigg]\,\ud\lambda\\
&\quad+\left|\oner\sumrr\frac{\ddm_{k,l}(Z_i^*,\htheta_n)}{n\pi_{n,i}^*}-\onen\sumn\ddm_{k,l}(Z_i,\htheta_n)\right|\\
&=\opD(1).
\end{align*}
\end{proof}

\begin{lemma}\label{lem2:DistributionofdM}
Under Assumptions \ref{ass:4}-\ref{ass:6}, given $\Dn$ in probability,
\begin{equation}\label{DistributionofdM}
\sqrt{s_n}\{ \Lambda_{n,R}(\htheta_n)\}^{-1/2}\dM^*_{s_n}(\htheta_n)\cvdD \Nor\left(\0,\I\right),
\end{equation}
in conditional distribution.
\end{lemma}
\begin{proof}
Note that
\begin{equation}\label{dMr*}
\sqrt{s_n}\dM^*_{s_n}(\htheta_n)=\frac{1}{\sqrt{s_n}}\sumrr\frac{\dm(Z_i^*,\htheta_n)}{n\pi^*_i}\equiv\frac{1}{\sqrt{s_n}}\sumrr\bm{\eta}_i
\end{equation}
Given $\Dn, \bm{\eta}_1,...,\bm{\eta}_{s_n}$ are i.i.d, with
\begin{align}
\Exp(\bm{\eta}|\Dn)&=
\onen\sumn\dm(Z_i,\htheta_n)=\0,\text{ and}\label{Expeta}\\
\Var(\bm{\eta}_i|\Dn)&=\Lambda_{n,R}(\htheta_n)=\frac{1}{n^2}\sumn\frac{\dm(Z_i,\htheta_n)\dm\tp(Z_i,\htheta_n)}{\pi_{n,i}}\notag\\
&\leq\max_{i=1,...,n}\left(\frac{1}{n\pi_{n,i}} \right)\onen\sumn\dm(Z_i,\htheta_n)\dm\tp(Z_i,\htheta_n)
=\Op,\label{Vareta}
\end{align}
where the inequality in \eqref{Vareta} is in the Loewner ordering, i.e., $\A_1\le \A_2$ means $\A_1-\A_2$ is a negative semi-definite matrix. 

Meanwhile, for every $\varepsilon>0$ and some $\delta\in(0,2]$,
\begin{align*}
&\oner\sumrr\Exp\left\{\Vert \bm{\eta}_i\Vert^2I(\| \bm{\eta}_i\|>s_n^{1/2}\varepsilon)|\Dn  \right\}
\leq\frac{1}{s_n^{1+\delta/2}\varepsilon^\delta}\sumrr\Exp\left\{\Vert \bm{\eta}_i\Vert^{2+\delta}I(\| \bm{\eta}_i\|>s_n^{1/2}\varepsilon)|\Dn  \right\}\\
&\leq\frac{1}{s_n^{1+\delta/2}\varepsilon^\delta}\sumrr\Exp\left(\| \bm{\eta}_i\|^{2+\delta}|\Dn  \right)
\leq\frac{1}{s_n^{\delta/2}n^{2+\delta}\varepsilon^\delta}\sumn\frac{\|\dm(Z_i,\htheta_n)\|^{2+\delta}}{\pi_{n,i}^{1+\delta}}\\
&=\max_{i=1,...,n}\left(\frac{1}{n\pi_{n,i}} \right)^{1+\delta}\frac{1}{ns_n^{\delta/2}\varepsilon^\delta}\sumn\|\dm(Z_i,\htheta_n)\|^{2+\delta}
=O_P(s_n^{-\delta/2}).
\end{align*}
This shows that Lindeberg's condition is satisfied in probability. From \eqref{dMr*}, \eqref{Expeta} and \eqref{Vareta}, by the Lindeberg-Feller central limit theorem (Proposition 2.27 of \cite{Vaart:98}), conditionally on $\Dn$, \eqref{DistributionofdM} follows.
\end{proof}

\begin{proof}[Proof of Theorem~\ref{thm:1}]
Based on Lemma~\ref{lem:BrtoddM} and Lemma \ref{lem2:DistributionofdM}, now we are ready to prove Theorem \ref{thm:1}.
By direct calculation, we have that for any $\btheta$,
\begin{equation*}
  \Exp\left(\M_{s_n}^*(\btheta)|\Dn\right)=\M_n(\btheta).
\end{equation*}
By Chebyshev's inequality, for any $\varepsilon>0$,
\begin{align*}
\Pr\left\{\left|\M_{s_n}^*(\btheta)-\M_n(\btheta)\right|\geq \varepsilon |\Dn\right\}
&\leq\frac{\Var\{\M_{s_n}^*(\btheta)|\Dn\}}{\varepsilon^2}
=\frac{1}{\varepsilon^2s_nn^2}\sumn \frac{m^2(Z_i,\btheta)}{\pi_{n,i}}\\
&\leq \frac{1}{\varepsilon^2 s_n}\max_{i=1,...,n}\left(\frac{1}{n\pi_{n,i}}\right)\onen\sumn m^2(Z_i,\btheta)
=O_P\left(s_n^{-1}\right).
\end{align*}
Thus, for every $\btheta$,
\begin{align}
\M_{s_n}^*(\btheta)-\M_n(\btheta)=\opD(1).
\end{align}

Note that under Assumptions \ref{ass:1}, \ref{ass:2}, the parameter space is compact and $\htheta_n$ is the unique global maximum of the continuous concave function $M_n(\btheta)$. Thus from Theorem 5.9 and its remark of \cite{Vaart:98}, conditionally on $\Dn$,
\begin{equation}\label{ConsistencyOfTtheta}
  \|\ttheta_{s_n,R}-\htheta_n\|=o_{P|\Dn}(1)=\op.
\end{equation}
The consistency ensures that $\ttheta_{s_n,R}$ is close to $\htheta_n$ as long as $s_n$ is large. By Taylor expansion,
\begin{align}\label{Taylor}
  0=\dM_{s_n}^*(\ttheta_{s_n,R})
  &=\dM_{s_n}^*(\htheta_n)
  +B_{s_n}(\ttheta_{s_n,R}-\htheta_n),
\end{align}
where
  \begin{align*}
B_{s_n}=\int_{0}^{1}\oner \sumrr\frac{\ddm\{Z_i^*,\htheta_n+\lambda(\ttheta_{s_n,R}-\htheta_n) \}}{n\pi_{n,i}^*}\,\ud\lambda.
\end{align*}

From \eqref{Taylor} and Lemma \ref{lem:BrtoddM},
\begin{equation}\label{taylor-expansion}
  0=\dM_{s_n}^*(\ttheta_{s_n,R})
  =\dM_{s_n}^*(\htheta_n)+\{\ddM_n(\htheta_n)+\op\}
    (\ttheta_{s_n,R}-\htheta_n),
\end{equation}
which shows that%
\begin{align}\label{difference}
  \ttheta_{s_n,R}-\htheta_n
  &=-\{\ddM_n(\htheta_n)+\op\}^{-1}\dM_{s_n}^*(\htheta_n)\notag\\
  &=-\frac{1}{\sqrt{s_n}}\{\ddM_n(\htheta_n)+\op\}^{-1}\{ \Lambda_{n,R}(\htheta_n)\}^{1/2}\sqrt{s_n}\{ \Lambda_{n,R}(\htheta_n)\}^{-1/2}\dM_{s_n}^*(\htheta_n).
\end{align}
By Lemma \ref{lem2:DistributionofdM} and Slutsky's theorem, we obtain that, given full data $\Dn$ in probability, 
 \begin{align}\label{eq:P-phi}
  \sqrt{s_n}\{V_{n,R}(\htheta_n)\}^{-1/2}(\ttheta_{s_n,R}-\htheta_n)\rightarrow \Nor\left(\0,\I \right),
 \end{align}
 in conditional distribution, and this finishes the proof.
\end{proof}
\subsection{Proof for Theorem~\ref{thm:2}}
\addcontentsline{toc}{subsection}{Appendix B: Proof for Theorem~\ref{thm:2}}
Let $\nu_i=1$ if the $i$-th data point is selected in the subsample and $\nu_i=0$ otherwise. The estimator in (\ref{eq:28}) %
is the same as the maximizer of 
\begin{align*}
  M^*_P(\btheta)
  &=\onen\sum_{i=1}^{s_n^*}\frac{m(Z_i^*,\btheta)}{s_n\pi_{n,i}^*}
    =\onen\sum_{i=1}^{n}\frac{\nu_im(Z_i,\btheta)}{s_n\pi_{n,i}}.
\end{align*}
Here, we use $s_n$ to replace $s_n^*$ in~\eqref{eq:28} for convenience, and the resulting estimator is identical to $\ttheta_{s_n,P}$.

To prove Theorem \ref{thm:2}, we first establish the following Lemmas \ref{lem:dM_PAsim} and \ref{lem:ddMr}. 

\begin{lemma}\label{lem:dM_PAsim}
If Assumptions \ref{ass:4}-\ref{ass:6} hold, then, given $\Dn$,
  \begin{align*}
    \sqrt{s_n}\{\Lambda_{n,P}(\htheta_n)\}^{-1/2}\dM^*_P(\htheta_n) \cvdD \Nor(\0,\I),
  \end{align*}
  in conditional distribution, where
\begin{align*}
  \Lambda_{n,P}(\htheta_n)=\frac{1}{n^2}
  \sumn\frac{(1-s_n\pi_{n,i})\dm(Z_i,\htheta_n)\dm\tp(Z_i,\htheta_n)}{\pi_{n,i}}.
\end{align*}
\end{lemma}
\begin{proof}
Write 
\begin{align*}
  \sqrt{s_n}\dM_P^*(\htheta_n)
  =\sumn\frac{\nu_i\dm(Z_i,\htheta_n)}{n\sqrt{s_n}\pi_{n,i}}\equiv \sumn\bm{\eta}_{Pi}.
  \end{align*}
By direct calculation and according to the definition of $\htheta_n$,
\begin{align*}
\Exp\left(\sumn\bm{\eta}_{Pi} \Big|\Dn \right)=\sqrt{s_n}\sumn\frac{\dm(Z_i,\htheta_n)}{n}=\0,
\end{align*}
and
\begin{align*}
  \Var\left( \sumn\bm{\eta}_{Pi}\Bigg|\Dn\right)
    &=\frac{1}{n^2}\sumn
      \frac{\Var(\nu_i|\Dn)\dm(Z_i,\htheta_n)\dm\tp(Z_i,\htheta_n)}{r\pi^2_i}\\
    &=\frac{1}{n^2}\sumn
      \frac{(1-s_n\pi_{n,i})\dm(Z_i,\htheta_n)\dm\tp(Z_i,\htheta_n)}{\pi_{n,i}}
  =\Lambda_{n,P}(\htheta_n)\\
    &\leq\Big(\max_{i}\frac{1}{n\pi_{n,i}}\Big)
      \onen\sumn\dm(Z_i,\htheta_n)\dm\tp(Z_i,\htheta_n)=\Op.
  \end{align*}

Next, we check Lindeberg's condition in conditional distribution. Note that for $\rho\in(0,2]$ and any $\varepsilon>0$,
\begin{align*}
  &\quad\sumn\Exp\left\{\|\bm{\eta}_{Pi}\|
    I(\|\bm{\eta}_{Pi}\|>\varepsilon)\Big|\Dn\right\}%
    \leq\frac{1}{\varepsilon^\rho}\sumn\Exp\left\{\|\bm{\eta}_{Pi}\|^{2+\rho}
    I(\|\bm{\eta}_{Pi}\|>\varepsilon)\Big|\Dn\right\}\\
  &\leq\frac{1}{\varepsilon^\rho}\sumn\Exp\left(\|\bm{\eta}_{Pi}\|^{2+\rho}\Big|\Dn\right)%
    =\frac{1}{\varepsilon^\rho}\Exp\left\{\sumn\frac{\nu_i^{2+\rho}\|\dm(Z_i,\htheta_n)\|^{2+\rho}}{n^{2+\rho}s_n^{1+\rho/2}\pi_{n,i}^{2+\rho}}\Bigg|\Dn\right\}\\
  &=\frac{1}{\varepsilon^\rho}\sumn\frac{\|\dm(Z_i,\htheta_n)\|^{2+\rho}}{n^{2+\rho}s_n^{\rho/2}\pi_{n,i}^{1+\rho}}\\
  &\leq \max_{i}\left(\frac{1}{n\pi_{n,i}}\right)^{1+\rho} \frac{1}{s_n^{\rho/2}\varepsilon^\rho n}\sumn \| \dm(Z_i,\htheta_n)\|^{2+\rho}
  =O_P(s_n^{-\rho/2})=\op.
\end{align*}
According to the Lindeberg-Feller Central Limit Theorem \citep[cf.]{Vaart:98}, given $\Dn$,
\begin{align*}
\sqrt{s_n}\{\Lambda_{n,P}(\htheta_n) \}^{-1/2}\dM^*_P(\htheta_n) \rightarrow \Nor(\0,\I),
\end{align*}
in conditional distribution.
\end{proof}
\begin{lemma}\label{lem:ddMr}
  Under Assumptions \ref{ass:3} and \ref{ass:6}, for any $\bu_{s_n}=\op$, conditional on $\Dn$,
  \begin{equation*}
    \onen\sumn\frac{\nu_i\ddm(Z_i,\htheta_n+\bu_{s_n})}{s_n\pi_{n,i}}
    -\onen\sumn\ddm(Z_i,\htheta_n)=o_P(1).
  \end{equation*}
\end{lemma}
\begin{proof}
First, note that %
\begin{align}\label{eq:lem6Op}
\onen\sumn\frac{\nu_i\psi(Z_i)}{s_n\pi_{n,i}}=O_{P|\Dn}(1),
\end{align}
by Chebyshev's inequality and the fact that
\begin{align*}
\Exp\left(\onen\sumn\frac{\nu_i\psi(Z_i)}{s_n\pi_{n,i}}\Bigg|\Dn\right)
&=\onen\sumn\frac{\psi(Z_i)\Exp(\nu_i|\Dn)}{s_n\pi_{n,i}}
=\onen\sumn\psi(Z_i)=\Exp\{\psi(Z_i)\}+\op,\\
\Var\left(\onen\sumn\frac{\nu_i\psi(Z_i)}{s_n\pi_{n,i}}\Big|\Dn \right)
&=\frac{1}{n^2}\sumn \frac{\psi^2(Z_i)\Var(\nu_i|\Dn)}{s_n^2\pi_{n,i}^2}
\leq\frac{1}{n^2}\sumn \frac{\psi^2(Z_i)\Exp(\nu^2_i)}{s_n^2\pi_{n,i}^2}\\
&=\frac{1}{n^2}\sumn \frac{\psi^2(Z_i)}{s_n\pi_{n,i}}
\leq \frac{1}{s_nn}\sumn\psi^2(Z_i)\max_{i}\frac{1}{n\pi_{n,i}}
=O_P(s_n^{-1}).
\end{align*}

Thus, for every $k,l=1,2,...,d$, from Assumption \ref{ass:3}, we have
\begin{align*}
\left|\onen\sumn\frac{\nu_i\ddm_{k,l}(Z_i,\htheta_n+\bu_{s_n})}{s_n\pi_{n,i}}-\onen\sumn\frac{\nu_i\ddm_{k,l}(Z_i,\htheta_n)}{s_n\pi_{n,i}}\right|
&\leq\frac{\|\bu_{s_n}\|}{n}\sumn\frac{\nu_i\psi(Z_i)}{s_n\pi_{n,i}}
=o_{P}(1).
 \end{align*}
which shows that
\begin{align}\label{eq:lem6step11}
\onen\sumn\frac{\nu_i\ddm(Z_i,\htheta_n+\bu_{s_n})}{s_n\pi_{n,i}}-\onen\sumn\frac{\nu_i\ddm(Z_i,\htheta_n)}{s_n\pi_{n,i}}=o_{P}(1).
\end{align}
According to \eqref{eq:ddmOP1}, for every $k,l=1,2,...,d$
\begin{align*}
\Exp\left(\onen\sumn\frac{\nu_i\ddm_{k,l}(Z_i,\htheta_n)}{s_n\pi_{n,i}}\Bigg|\Dn\right)&=\onen\sumn\ddm_{k,l}(Z_i,\htheta_n),\\
\Var\left(\onen\sumn\frac{\nu_i\ddm_{k,l}(Z_i,\htheta_n)}{s_n\pi_{n,i}} \Bigg|\Dn \right)&=\frac{1}{s_nn^2}\sumn\frac{(1-s_n\pi_{n,i})\ddm^2_{k,l}(Z_i,\htheta_n)}{\pi_{n,i}}
\leq \frac{1}{s_nn^2}\sumn\frac{\ddm_{k,l}^2(Z_i,\htheta_n)}{\pi_{n,i}}\\
&\leq\max_{i}\left(\frac{1}{n\pi_{n,i}}\right)\frac{1}{s_nn}\sumn \ddm_{k,l}^2(Z_i,\htheta_n)
=O_P(s_n^{-1}).
\end{align*}
Thus, Chebyshev's inequality tells us that
\begin{align}\label{eq:lem6step21}
\onen\sumn\frac{\nu_i\ddm(Z_i,\htheta_n)}{s_n\pi_{n,i}}-\onen\sumn\ddm(Z_i,\htheta_n)=O_{P|\Dn}(s_n^{-1/2})=o_P(1).
\end{align}
Therefore, combining \eqref{eq:lem6step11} and \eqref{eq:lem6step21}, we have
\begin{align*}%
\onen\sumn\frac{\nu_i\ddm(Z_i,\htheta_n+\bu_{s_n})}{s_n\pi_{n,i}}-\onen\sumn\ddm(Z_i,\htheta_n)=o_P(1).
\end{align*}

\end{proof}

\begin{proof}[Proof of Theorem~\ref{thm:2}]
Denote
\begin{align*}
\gamma_P(\bu)&=s_nM^*_P(\htheta_n+\bu/\sqrt{s_n})-s_nM^*_P(\htheta_n).
\end{align*}
Under Assumption \ref{ass:2}, $\sqrt{s_n}(\ttheta_{s_n,P}-\htheta_n)$ is the unique maximizer of $\gamma_P(\bu)$ as $\ttheta_{s_n,P}$ is the unique maximizer of $M^*_P(\bu)$. By Taylor's expansion,
\begin{align*}
  \gamma_P(\bu)
    &=\sqrt{s_n}\bu^T\dM^*_P(\htheta_n)
      +\frac{1}{2}\bu^T\ddM^*_P(\htheta_n+\bau/\sqrt{s_n})\bu
\end{align*}
where $\bau$ lies between $\0$ and $\bu$.
From Lemma \ref{lem:dM_PAsim} $\sqrt{s_n}\dM^*_P(\htheta_n)$ is stochastically bounded in conditional probability given $\Dn$. From Lemma \ref{lem:ddMr}, conditional on $\Dn$, $\ddM^*_P(\htheta_n+\bau/\sqrt{s_n})-\ddM_n(\htheta_n)=\op$ and $\ddM_n(\htheta_n)$ converges to a positive-definite matrix. 

Thus from the Basic Corollary in page 2 of \cite{hjort2011asymptotics}, the maximizer of $s_n\gamma_P(\bu)$, $\sqrt{s_n}(\ttheta_{s_n,P} -\htheta_n)$, satisfies that
\begin{align}
  \sqrt{s_n}(\ttheta_{s_n,P} -\htheta_n)
  =\ddM_n^{-1}(\htheta_n)\sqrt{s_n}\dM^*_P(\htheta_n)+\op,
\end{align}
which implies that
\begin{align}
\sqrt{s_n}\{V_{n,P}(\htheta_n) \}^{-1/2}(\ttheta_{s_n,P} -\htheta_n)\rightarrow \Nor(\0,\I),
\end{align}
in conditional distribution, given $\Dn$ in probability.
This finishes the proof. 
\end{proof}

\subsection{Proof of Theorem~\ref{thm:1p}}
\begin{proof}[Proof of Theorem~\ref{thm:1p}]
Letting $S_{n,R}=\sqrt{s_n}(\ttheta_{s_n,R}-\htheta_n)$ and $Y_n=\sqrt{s_n}(\htheta_n-\btheta_0)$, 
we have
\begin{equation*}
    \sqrt{s_n}(\ttheta_{s_n,R}-\btheta_0)=S_{n,R}+Y_n.
\end{equation*}
According to Theorem \ref{thm:1}, we know that under Assumptions \ref{ass:1}-\ref{ass:6} the characteristic function of $S_{n,R}$ given $\Dn$ satisfies that
\begin{equation}\label{eq:42}
\Exp\big(e^{\img\bm{t}\tp S_{n,R}}\big|\Dn\big)
  =e^{-0.5\bm{t}\tp V_{n,R}(\htheta_n)\bm{t}}+\opD(1)
  =e^{-0.5\bm{t}\tp V_{n,R}(\htheta_n)\bm{t}}+\op,
\end{equation}
where %
$\img$ is the imaginary unit. 
For every $k,l=1,2,...,d$, from Lipschitz continuity, we have
\begin{equation*}
  \left|\onen\sumn\ddm_{k,l}(Z_i,\htheta_n)
    -\onen\sumn\ddm_{k,l}(Z_i,\btheta_0)\right|
  \leq\|\htheta_n-\btheta_0\|\onen\sumn\psi(Z_i)=\op.
\end{equation*}
Thus, applying the law of large numbers, we know that $\ddM_n(\htheta_n)=\ddM_n(\btheta_0)+\op=\ddM(\btheta_0)+\op$.

Next we prove that $\Lambda_{n,R}(\htheta_n)=\Lambda_{\pi}(\btheta_0)+\op$. We have
\begin{align*}
  &\|\Lambda_{n,R}(\htheta_n)-\Lambda_{n,R}(\btheta_0)\|\\
    &=\onen\left\|\sumn\frac{\dm(Z_i,\htheta_n)\dm\tp(Z_i,\htheta_n)
    -\dm(Z_i,\btheta_0)\dm\tp(Z_i,\btheta_0)}{n\pi_{n,i}}\right\|\\
  &\leq\max_{i}\left(\frac{1}{n\pi_{n,i}}\right)
    \onen\sumn\left\|\dm(Z_i,\htheta_n)\dm\tp(Z_i,\htheta_n)
    -\dm(Z_i,\htheta_n)\dm\tp(Z_i,\btheta_0)\right\|\\
  &\quad+\max_{i}\left(\frac{1}{n\pi_{n,i}}\right)
    \onen\sumn\left\|\dm(Z_i,\htheta_n)\dm\tp(Z_i,\btheta_0)
    -\dm(Z_i,\btheta_0)\dm\tp(Z_i,\btheta_0)\right\|\\
  &=\max_{i}\left(\frac{1}{n\pi_{n,i}}\right)
    \onen\sumn\left\{\|\dm(Z_i,\htheta_n)\|+\|\dm(Z_i,\btheta_0)\|\right\}
    \|\dm(Z_i,\htheta_n)-\dm(Z_i,\btheta_0)\|.
\end{align*}
Using Taylor's expansion, we obtain
\begin{equation*}
    \dm(Z_i,\htheta_n)=\dm(Z_i,\btheta_0)+B_{n,i}(\htheta_n-\btheta_0),
\end{equation*}
where $B_{n,i}=\int_{0}^{1}\ddm\{Z_i,\btheta_0+\lambda(\htheta_n-\btheta_0)\}\ud\lambda$ satisfies that 
\begin{align*}
  \|B_{n,i}-\ddm(Z_i,\btheta_0)\|
  \leq d\int_{0}^{1}\lambda\psi(Z_i)\|\htheta_n-\btheta_0\|\ud\lambda
  =0.5d\psi(Z_i)\|\htheta_n-\btheta_0\|.
\end{align*}
due to the Lipschitz continuity in Assumption~\ref{ass:3}. This shows that
\begin{align*}
  \|\dm(Z_i,\htheta_n)-\dm(Z_i,\btheta_0)\|
  \le0.5d\psi(Z_i)\|\htheta_n-\btheta_0\|^2
  +\|\ddm(Z_i,\btheta_0)\|\|(\htheta_n-\btheta_0)\|.
\end{align*}
Thus,
\begin{align*}
    &\|\Lambda_{n,R}(\htheta_n)-\Lambda_{n,R}(\btheta_0)\|\\
    &\leq\max_{i}\left(\frac{1}{n\pi_{n,i}}\right)
      \Bigg[0.5d\|\htheta_n-\btheta_0\|^{2}
      \left\{\onen\sumn\|\dm(Z_i,\htheta_n)\|\psi(Z_i)
      +\onen\sumn\|\dm(Z_i,\btheta_0)\|\psi(Z_i)\right\}\\
    &\quad+\|\htheta_n-\btheta_0\|
      \left\{\onen\sumn\|\dm(Z_i,\htheta_n)\|\|\ddm(Z_i,\btheta_0)\|
      +\onen\sumn\|\dm(Z_i,\btheta_0)\|\|\ddm(Z_i,\btheta_0)\|\right\}\Bigg].%
\end{align*}
From H\"older's inequality
\begin{align*}
  \onen\sumn\|\dm(Z_i,\htheta_n)\|\psi(Z_i)
  \le\left\{\onen\sumn\|\dm(Z_i,\htheta_n)\|^4\right\}^{\frac{1}{4}}
  \left\{\onen\sumn\psi(Z_i)^{\frac{4}{3}}\right\}^{\frac{3}{4}}
  =\Op.
\end{align*}
Similarly, we can show that $\onen\sumn\|\dm(Z_i,\btheta_0)\|\psi(Z_i)$, $\onen\sumn\|\dm(Z_i,\htheta_n)\|\|\ddm(Z_i,\btheta_0)\|$, and $\onen\sumn\|\dm(Z_i,\btheta_0)\|\|\ddm(Z_i,\btheta_0)\|$ are all $\Op$. 
Therefore, $\|\Lambda_{n,R}(\htheta_n)-\Lambda_{n,R}(\btheta_0)\|=\op$, and thus (\ref{eq:42}) implies that
\begin{equation*}
  \Exp\big(e^{\img\bm{t}\tp S_{n,R}}\big|\Dn\big)
  =e^{-0.5\bm{t}\tp\ddM^{-1}(\btheta_0)\Lambda_{\pi}(\btheta_0)
    \ddM^{-1}(\btheta_0)\bm{t}} +\op, 
\end{equation*}
where the $\op$ is bounded.

Note that $Y_{n}=\sqrt{\frac{s_n}{n}}\sqrt{n}(\htheta_n-\btheta_0)$. 
Using Proposition \ref{prop1}, we have
\begin{equation*}
  \Exp(e^{\img\bm{t}\tp Y_{n}}) \rightarrow
  e^{-0.5\bm{t}\tp c\ddM^{-1}(\btheta_0)\Lambda(\btheta_0)\ddM^{-1}(\btheta_0)\bm{t}}.
\end{equation*}
Since $Y_{n}$ is $\mathcal{D}_{n}$ measurable, we have
\begin{align*}
  &\left|\Exp\left\{e^{\img\bm{t}\tp(Y_{n}+S_{n,R})}-e^{\img\bm{t}\tp Y_{n}}e^{-0.5\bm{t}\tp\ddM^{-1}(\btheta_0)\Lambda_{\pi}(\btheta_0)\ddM^{-1}(\btheta_0)\bm{t}}\right\}\right|\\
  &=\left|\Exp\left[\Exp\left\{e^{\img\bm{t}\tp(Y_{n}+S_{n,R})}-e^{\img\bm{t}\tp Y_{n}}e^{-0.5\bm{t}\tp\ddM^{-1}(\btheta_0)\Lambda_{\pi}(\btheta_0)\ddM^{-1}(\btheta_0)\bm{t}}\Big|\Dn\right\}\right]\right|\\
  &=\left|\Exp\left[e^{\img\bm{t}\tp Y_{n}}\left\{\Exp\left(e^{\img\bm{t}S_{n,R}}\big|\Dn\right)-e^{-0.5\bm{t}\tp\ddM^{-1}(\btheta_0)\Lambda_{\pi}(\btheta_0)\ddM^{-1}(\btheta_0)\bm{t}}\right\}\right]\right|\\
    &\leq\Exp\left\{\left|\Exp\left(e^{\img\bm{t}\tp S_{n,R}}|\Dn\right)-e^{-0.5\bm{t}\tp\ddM^{-1}(\btheta_0)\Lambda_{\pi}(\btheta_0)\ddM^{-1}(\btheta_0)\bm{t}}\right|\right\}\\
  &=o(1),
\end{align*}
where the last step is from the dominated convergence theorem. 
Therefore,
\begin{equation*}
  \Exp\big\{e^{\img\bm{t}\tp(Y_{n}+S_{n,R})}\big\}
  =\Exp(e^{\img\bm{t}\tp Y_{n}})
  e^{-0.5\bm{t}\tp\ddM^{-1}(\btheta_0)\Lambda_{\pi}(\btheta_{0})
    \ddM^{-1}(\btheta_0)\bm{t}}+o(1)
  \to\Exp\big\{e^{-0.5\bm{t}\tp V_{R}^{U}(\btheta_0)\bm{t}}\big\}.
\end{equation*}
Hence, we obtain that
\begin{equation*}
    \sqrt{s_n}\{V^{U}_{R}(\btheta_{0})\}^{-1/2}(\ttheta_{s_n,R}-\btheta_{0})\cvd\Nor(\0,\I_d).
\end{equation*}
\end{proof}

\subsection{Proof of Theorem~\ref{thm:2p}}
\begin{proof}[Proof of Theorem~\ref{thm:2p}]
The technique of proving Theorem~\ref{thm:2p} is similar to that of proving Theorem~\ref{thm:1p}.   
Denoting $S_{n,P}=\sqrt{s_n}(\ttheta_{s_n,P}-\htheta_n)$, we write $\sqrt{s_n}(\ttheta_{s_n,P}-\btheta_0)=S_{n,P}+Y_n$.  
From Theorem \ref{thm:2}, we know that under Assumptions \ref{ass:1}-\ref{ass:6},
\begin{equation*}
  \Exp\big(e^{\img\bm{t}\tp S_{n,P}}\big|\Dn\big)
  =e^{-0.5\bm{t}\tp V_{n,P}(\htheta_n)\bm{t}}+\opD
  =e^{-0.5\bm{t}\tp V_{n,P}(\htheta_n)\bm{t}}+\op.
\end{equation*}
In the proof of Theorem~\ref{thm:1p}, we have proved that $\ddM(\htheta_n)=\ddM(\btheta_0)+\op$ and $\Lambda_{n,R}(\htheta_n)=\Lambda_{\pi}+\op$. Using a similar approach, we can show that %
\begin{equation*}
  \frac{s_n}{n^{2}}\sumn\dm(Z_i,\htheta_n)\dm\tp(Z_i,\htheta_n)
  =c\Lambda(\btheta_0)+\op.
\end{equation*}
Therefore, 
\begin{equation*}
  \Exp\big(e^{\img\bm{t}\tp S_{n,P}}\big|\Dn\big)
  =e^{-0.5\bm{t}\tp \ddM^{-1}(\btheta_0)
    \{\Lambda_{\pi}(\btheta_0)-c\Lambda(\btheta_0)\}
    \ddM^{-1}(\btheta_0)\bm{t}}+\op.
\end{equation*}
Now we use the same technique used in the proof of Theorem \ref{thm:1p}. Since $Y_{n}$ is $\mathcal{D}_{n}$ measurable, we have
\begin{align*}
    &\left|\Exp\left\{e^{\img\bm{t}\tp(Y_{n}+S_{n,P})}-e^{\img\bm{t}\tp Y_{n}}e^{-0.5\bm{t}\tp\ddM^{-1}(\btheta_0)\{\Lambda_{\pi}(\btheta_0)-c\Lambda(\btheta_0)\}\ddM^{-1}(\btheta_0)\bm{t}}\right\}\right|\\
    &=\left|\Exp\left[e^{\img\bm{t}\tp Y_{n}}\left\{\Exp\left(e^{\img\bm{t}S_{n,P}}|\Dn\right)-e^{-0.5\bm{t}\tp\ddM^{-1}(\btheta_0)\{\Lambda_{\pi}(\btheta_0)-c\Lambda(\btheta_0)\}\ddM^{-1}(\btheta_0)\bm{t}}\right\}\right]\right|\\
      &\leq\Exp\left\{\left|\Exp\left(e^{\img\bm{t}\tp S_{n,P}}|\Dn\right)-e^{-0.5\bm{t}\tp\ddM^{-1}(\btheta_0)\{\Lambda_{\pi}(\btheta_0)-c\Lambda(\btheta_0)\}\ddM^{-1}(\btheta_0)\bm{t}}\right|\right\}
  \rightarrow0,
\end{align*}
where the last step is from the dominated convergence theorem. %
Hence,
\begin{equation*}
  \Exp\big\{e^{\img\bm{t}\tp(Y_{n}+S_{n,P})}\big\}
  =\Exp(e^{\img\bm{t}\tp Y_{n}})e^{-0.5\bm{t}\tp\ddM^{-1}(\btheta_0)\{\Lambda_{\pi}(\btheta_0)-c\Lambda(\btheta_0)\}\ddM^{-1}(\btheta_0)\bm{t}}+o(1)
  \to\Exp\big\{e^{-0.5\bm{t}\tp V_{P}^{U}(\btheta_0)\bm{t}}\big\},
\end{equation*}
and this finishes the proof. 
\end{proof}

\black
\subsection{Proof of Theorem~\ref{thm:3}}
\addcontentsline{toc}{subsection}{Appendix C: Proof of Theorem~\ref{thm:3}}
\begin{proof}[Proof of Theorem~\ref{thm:3}]
  For the result in \eqref{eq:SSPmMSE},
\begin{align*}
  \tr\{\Lambda_{n,R}(\htheta_n)\}
  =&\frac{1}{n^2}\sumn\frac{\|\dm(Z_i,\htheta_n)\|^2}{\pi_{n,i}}
  =\frac{1}{n^2}\sumn\pi_{n,i}
     \sumn\frac{\|\dm(Z_i,\htheta_n)\|^2}{\pi_{n,i}}
     \ge\frac{1}{n^2}\bigg\{\sumn\|\dm(Z_i,\htheta_n)\|\bigg\}^2.
\end{align*}
Here, the last step is from the Cauchy-Schwarz inequality and the equality holds if and only if $\pi_{n,i}\propto\|\dm(Z_i,\htheta_n)\|$.

\end{proof}
\subsection{Proof of Theorem~\ref{thm:4}}
\addcontentsline{toc}{subsection}{Appendix D: Proof of Theorem~\ref{thm:4}}
\begin{proof}
Note that
\begin{align*}
\tr\{\Lambda_{n,P}(\htheta_n)\}&=\tr\left\{\frac{1}{n^2}\sumn\frac{(1-s_n\pi_{n,i})\dm(Z_i,\htheta_n)\dm\tp(Z_i,\htheta_n)}{\pi_{n,i}}
\right\}\\
&=\frac{1}{n^2}\left[\sumn\frac{\|\dm(Z_i,\htheta_n)\|^2}{\pi_{n,i}}
-s_n\sumn\|\dm(Z_i,\htheta_n)\|^2\right].
\end{align*}
Thus, minimizing $\tr\{\Lambda_{n,P}(\htheta_n)\}$ is equal to minimizing $\sumn\frac{\|\dm(Z_i,\htheta_n)\|^2}{\pi_{n,i}}$. For $i=1, ..., n$, let $t_i=\|\dm(Z_i,\htheta_n)\|$ and let $t_{(i)}$ denote the order statistics of $\|\dm(Z_i,\htheta_n)\|$, i.e., $t_{(i)}=\|\dm(Z,\htheta_n)\|_{(i)}$. The optimization problem of minimizing $\tr\{\Lambda_{n,P}(\htheta_n)\}$ subject to the constraints on $\pi_{n,i}$ can be presented as minimizing
\begin{align}\label{eq:object}
  T(\pi_1,\pi_2,...,\pi_n) &=\sumn\frac{t_{(i)}^2}{\pi_{n,i}},\\
\text{ subject to }
  \sumn\pi_{n,i}&=1\quad\text{and}\quad
  0\leq\pi_{n,i}\leq \oner,i=1,2,...,n.\notag
\end{align}
Defining slack variables $\omega_1^2,\omega_2^2,...,\omega_n^2$, to use Lagrangian multiplier method, we can construct
\begin{align*}
H(\pi_1,...,\pi_n,\tau,\mu_1,...,\mu_n,\omega_1,...,\omega_n)=\sumn\frac{t_{(i)}^2}{\pi_{n,i}}+\tau\left( \sumn\pi_{n,i}-1\right)+\sumn\mu_i\left(\pi_{n,i}+\omega_i^2-\oner\right).
\end{align*}
By taking the derivatives, the Karush–Kuhn–Tucker (KKT) conditions \citep{NumericalOptimization1999} are
\begin{align}
  \frac{\partial H}{\partial \pi_{n,i}}
  &=-\frac{t_{(i)}^2}{\pi_{n,i}^2}+\tau+\mu_i=0,
  &i=1,2,...,n. \label{eq:K1}\\
  \frac{\partial H}{\partial \tau}
  &=\sumn \pi_{n,i}-1=0,\label{eq:K2}\\
  \frac{\partial H}{\partial \mu_i}
  &=\pi_{n,i}+\omega_i^2=\oner,
  &i=1,2,...,n.\label{eq:K3}\\
  \frac{\partial H}{\partial \omega_i}
  &=2\mu_i\omega_i=0,
  &i=1,2,...,n.\label{eq:K4}\\
  \mu_i&\geq 0,&i=1,2,...,n.
\end{align}
From \eqref{eq:K1}, we have
\begin{align}\label{eq:K5}
\pi_{n,i}=\frac{t_{(i)}}{\sqrt{\tau+\mu_i}},\quad i=1,2,...,n.
\end{align}
Combining it with \eqref{eq:K3}, we have
\begin{align}\label{eq:K6}
\frac{t_{(i)}}{\sqrt{\tau+\mu_i}}+\omega_i^2=\oner,\quad i=1,2,...,n.
\end{align}
According to \eqref{eq:K4}, at least one of $\mu_i$ and $\omega_i$ must be 0. From \eqref{eq:K5} and \eqref{eq:K6},
\begin{align}
  &\text{if } t_{(i)}<\frac{\sqrt{\tau}}{s},
  && \mu=0 \text{ and }
  \pi_{n,i}=\frac{t_{(i)}}{\sqrt{\tau}}<\oner;\label{eq:K7}\\
  &\text{if } t_{(i)}\ge\frac{\sqrt{\tau}}{s},
  && \omega_i=0 \text{ and }
  \pi_{n,i}=\frac{t_{(i)}}{\sqrt{\tau+\mu_i}}=\oner.\label{eq:K8}
\end{align}
Thus, letting $g$ be the number of cases that $t_{(i)}\ge\frac{\sqrt{\tau}}{s}$, from \eqref{eq:K2} and the fact that $t_{(i)}$ is non-decreasing in $i$,
\begin{align}
  1=\sumn\pi_{n,i}
  =\sum_{i=1}^{n-g}\frac{t_{(i)}}{\sqrt{\tau}}
  +\sum_{i=n-g+1}^n\oner
  =\frac{\sum_{i=1}^{n-g}t_{(i)}}{\sqrt{\tau}}+\frac{g}{s},
\end{align}
which shows that
\begin{align}\label{eq:K10}
 \sqrt{\tau}=\frac{s}{s-g}\sum_{i=1}^{n-g}t_{(i)}.
\end{align}
Combining\eqref{eq:K7}, \eqref{eq:K8}, and \eqref{eq:K10}, 
\begin{numcases}{\pi_{n,i}=}
  \frac{t_{(i)}(s_n-g)}{s_n\sum_{i=1}^{n-g}t_{(i)}},
  & for $i=1,2,...,n-g;$\label{eq:K11}\\
  \oner, & for $i=n-g+1,...,n.$\label{eq:K12}
\end{numcases}
From \eqref{eq:K10},
\begin{align}
  H=\frac{\sum_{i=1}^{n-g}t_{(i)}}{s_n-g}=\frac{\sqrt{\tau}}{s_n},
\end{align}
Thus, from~\eqref{eq:K7} and~\eqref{eq:K8}, we know $t_{(i)}<H$ for $i=1,2,...,n-g$, and $t_{(i)}\ge H$, for $i=n-g+1, ..., n$. Therefore
\begin{align}
  \sumn(t_{(i)}\wedge H)=\sum_{i=1}^{n-g}t_{(i)}
  +\sum_{i=n-g+1}^nH=s_nH
\end{align}
Thus, from~\eqref{eq:K11}, for $i=1,2,...,n-g$,
\begin{align}
  \pi_{n,i}=\frac{t_{(i)}}{s_nH}=\frac{t_{(i)}\wedge H}{\sumn(t_{(i)}\wedge H)};
\end{align}
from~\eqref{eq:K12}, for $i=n-g+1, ..., n$,
\begin{align}
  \pi_{n,i}=\frac{H}{s_nH}=\frac{t_{(i)}\wedge H}{\sumn(t_{(i)}\wedge H)}.
\end{align}

For the result under the A-optimality, define $t_{(i)}=\|\ddM^{-1}_n(\htheta_n)\dm(Z,\htheta_n)\|_{(i)}$ and the proof is the same as the used for the L-optimality.
\end{proof}

\subsection{Proof of Theorem~\ref{thm:5}}
\addcontentsline{toc}{subsection}{Appendix E: Proof of Theorem~\ref{thm:5}}
The proof of Theorem~\ref{thm:5} relies on Lemmas~\ref{lem:two-step1} and \ref{lem:PoiDistributionofdM} below. %
\begin{lemma}\label{lem:two-step1}
Under Assumption \ref{ass:3}, if $\|\ttheta_{s_n,R}^\alpha-\htheta_n\|=o_P(1)$, then conditional on $\Dn$ and $\ttheta_{s_0,R}^{0*}$,
\begin{align}
B_{s_n}^{\ttheta_{s_0,R}^{0*}}-\ddM_n(\htheta_n)=o_P(1),
\end{align}
where
\begin{align*}
B_{s_n}^{\ttheta_{s_0,R}^{0*}}&=\int_{0}^{1}\oner \sumrr\frac{\ddm\left\{Z_i^*,\htheta_n+\lambda(\ttheta_{s_n,R}^{\alpha}-\htheta_n) \right\}}{n\tilde{\pi}_{n,R\alpha i}^{\mvc*}}\,\ud\lambda.
\end{align*}
\end{lemma}
\begin{proof}
For every $k,l=1,2,...,d$, from Lipschitz continuity, we have
\begin{align}\label{eq:50}
&\left|\oner\sumrr\frac{\ddm_{k,l}\{Z_i^*,\htheta_n+\lambda(\ttheta_{s_n,R}^{\alpha}-\htheta_n)\}}{n\tilde{\pi}_{n,R\alpha i}^{\mvc*}}-
\oner\sumrr\frac{\ddm_{k,l}(Z_i^*,\htheta_n)}{n\tilde{\pi}_{n,R\alpha i}^{\mvc*}} \right|\notag\\
&\leq\oner\sumrr\frac{
\psi(Z_i^*)\|\lambda(\ttheta_{s_n,R}^{\alpha}-\htheta_n)\|}{n\tilde{\pi}_{n,R\alpha i}^{\mvc*}} \notag\\
&\leq\lambda\|\ttheta_{s_n,R}^{\alpha}-\htheta_n\|\oner\sumrr\frac{\psi(Z_i^*)}{\alpha} %
=\|\ttheta_{s_n,R}^{\alpha}-\htheta_n\|O_P(1)%
=\op.
\end{align}

According to \eqref{eq:ddmOP1}, we have
\begin{align*}
\Exp\left(\oner\sumrr\frac{\ddm_{k,l}(Z_i^*,\htheta_n)}{n\tilde{\pi}_{n,R\alpha i}^{\mvc*}}\bigg|\Dn,\ttheta_{s_0,R}^{0*}\right)&=\onen\sumn\ddm_{k,l}(Z_i,\htheta_n),\\
\Var\left(\oner\sumrr\frac{\ddm_{k,l}(Z_i^*,\htheta_n)}{n\tilde{\pi}_{n,R\alpha i}^{\mvc*}}\bigg|\Dn,\ttheta_{s_0,R}^{0*}\right)&=\oner\sumn\frac{\ddm_{k,l}^2(Z_i,\htheta_n)}{n^2\pi_{n,R\alpha i}^{\mvc}}
\leq\frac{1}{\alpha s_nn}\sumn \ddm_{k,l}^2(Z_i,\htheta_n)
=O_P(s_n^{-1}).
\end{align*}
Thus, by Chebyshev's inequality, similar to \eqref{eq:ddm}, we have
\begin{align}\label{eq:51}
\left\|\oner\sumrr\frac{\ddm(Z_i^*,\htheta_n)}{n\tilde{\pi}_{n,R\alpha i}^{\mvc*}}-\onen\sumn\ddm(Z_i,\htheta_n)\right\|=o_{P|\Dn,\ttheta_{s_0,R}^{0*}}(1).
\end{align}

Combining \eqref{eq:50} and \eqref{eq:51}, we have
\begin{align*}
\left\|B_{s_n}-\onen\sumn\ddm(Z_i,\htheta_n)\right\|
&\leq \int_{0}^{1}\left\|\oner \sumrr\frac{\ddm\{Z_i^*,\htheta_n+\lambda(\ttheta_{s_n,R}^{\alpha}-\htheta_n) \}}{n\tilde{\pi}_{n,R\alpha i}^{\mvc*}}-\onen\sumn\ddm(Z_i,\htheta_n)\right\|\,\ud\lambda\\
&\leq \int_{0}^{1}
\Bigg[\left\|\oner\sumrr\frac{\ddm\{Z_i^*,\htheta_n+\lambda(\ttheta_{s_n,R}^{\alpha}-\htheta_n)\}}{n\tilde{\pi}_{n,R\alpha i}^{\mvc*}}-
\oner\sumrr\frac{\ddm(Z_i^*,\htheta_n)}{n\tilde{\pi}_{n,R\alpha i}^{\mvc*}} \right\|\\
&\quad+\left\|\oner\sumrr\frac{\ddm(Z_i^*,\htheta_n)}{n\tilde{\pi}_{n,R\alpha i}^{\mvc*}}-\onen\sumn\ddm(Z_i,\htheta_n)\right\|\Bigg]\,\ud\lambda
=o_P(1),
\end{align*}
which finishes the proof.
\end{proof}
\begin{lemma}\label{lem:PoiDistributionofdM}
If Assumption \ref{ass:4} %
hold, then given $\Dn$ and ${\ttheta_{s_0,R}^{0*}}$ in probability,
\begin{align}\label{eq:lem4}
  \sqrt{s_n}\{\Lambda_R^\alpha(\ttheta_{s_0,R}^{0*})\}^{-1/2}
  \dM_{R\alpha}^*(\htheta_n)\rightarrow
  \Nor\left(\0, \I\right),
\end{align}
in conditional distribution, where
\begin{align*}
  \dM_{R\alpha}^*(\htheta_n)=\frac{1}{ns_n}\sumrr \frac{\dm(Z_i^*,\htheta_n)}{\tilde{\pi}_{n,R\alpha i}^{\mvc*}},
    \quad\text{ and }\quad
\Lambda_R^\alpha(\ttheta_{s_0,R}^{0*})=\frac{1}{n^2}\sumn\frac{\dm(Z_i,\htheta_n)\dm\tp(Z_i,\htheta_n)}{\tilde{\pi}_{n,R\alpha i}^{\mvc}}.
\end{align*}
\end{lemma}
\begin{proof}
Note that
\begin{equation}\label{dMr*2}
\sqrt{s_n}\dM_{R\alpha}^*(\htheta_n)=\frac{1}{\sqrt{s_n}}\sumrr\frac{\dm(Z_i^*,\htheta_n)}{n\tilde{\pi}_{n,R\alpha i}^{\mvc*}}\equiv\frac{1}{\sqrt{s_n}}\sumrr\bm{\eta}_i^{\ttheta_{s_0,R}^{0*}}.
\end{equation}
Given $\Dn$ and $\ttheta_{s_0,R}^{0*}$, $\bm{\eta}_1^{\ttheta_{s_0,R}^{0*}},...,\bm{\eta}_{s_n}^{\ttheta_{s_0,R}^{0*}}$ are i.i.d, with
\begin{align}
\Exp(\bm{\eta}_i^{\ttheta_{s_0,R}^{0*}}|\Dn,\ttheta_{s_0,R}^{0*})
&=\onen\sumn\dm(Z_i,\htheta_n)=\0,\label{Expeta2}\quad\text{and}\\
  \Var(\bm{\eta}_i^{\ttheta_{s_0,R}^{0*}}|\Dn,{\ttheta_{s_0,R}^{0*}})
  &=\Exp\left\{\frac{\dm(Z_i^*,\htheta_n)\dm\tp(Z_i^*,\htheta_n)}{n^2(\tilde{\pi}_{n,R\alpha i}^{\mvc*})^2} \Bigg|\Dn,{\ttheta_{s_0,R}^{0*}} \right\}\\
&=\frac{1}{n^2}\sumn\frac{\dm(Z_i,\htheta_n)\dm\tp(Z_i,\htheta_n)}{\tilde{\pi}_{n,R\alpha i}^{\mvc*}}=\Lambda_R^\alpha(\ttheta_{s_0,R}^{0*}).\label{Vareta2}%
\end{align}
Meanwhile, for every $\varepsilon>0$ and some $\delta\in(0,2]$,
\begin{align*}
&\oner\sumrr\Exp\left\{\| \bm{\eta}_i^{\ttheta_{s_0,R}^{0*}}\|^2I(\| \bm{\eta}_i^{\ttheta_{s_0,R}^{0*}}\|>s_n^{1/2}\varepsilon)\Big|\Dn,{\ttheta_{s_0,R}^{0*}}\right\}\\
&\leq\frac{1}{s_n^{1+\delta/2}\varepsilon^\delta}\sumrr\Exp\left\{\| \bm{\eta}_i^{\ttheta_{s_0,R}^{0*}}\|^{2+\delta}I(\| \bm{\eta}_i^{\ttheta_{s_0,R}^{0*}}\|>s_n^{1/2}\varepsilon)\Big|\Dn,{\ttheta_{s_0,R}^{0*}}  \right\}\\
&\leq\frac{1}{s_n^{1+\delta/2}\varepsilon^\delta}\sumrr\Exp\left(\| \bm{\eta}_i^{\ttheta_{s_0,R}^{0*}}\|^{2+\delta}\Big|\Dn,{\ttheta_{s_0,R}^{0*}}  \right)\\
&\leq\frac{1}{s_n^{\delta/2}n^{2+\delta}\varepsilon^\delta}\sumn\frac{\|\dm(Z_i,\htheta_n)\|^{2+\delta}}{(\tilde{\pi}_{n,R\alpha i}^{\mvc*})^{1+\delta}}\\
&\leq\frac{1}{s_n^{\delta/2}\alpha^{1+\delta}\varepsilon^\delta}\onen \sumn\|\dm(Z_i,\htheta_n)\|^{2+\delta}
=O_P(s_n^{-\delta/2})=\op.
\end{align*}
where the second last equality is from Assumption \ref{ass:4}. This show that Lindeberg's condition is satisfied in probability. From \eqref{dMr*2}, \eqref{Expeta2} and \eqref{Vareta2}, by the Lindeberg-Feller central limit theorem (Proposition 2.27 of \cite{Vaart:98}), conditional on $\Dn, {\ttheta_{s_0,R}^{0*}}$, we obtain \eqref{eq:lem4}.
\end{proof}

\begin{proof}[Proof of Theorem~\ref{thm:5}]
By direct calculation, we have
\begin{align*}
\Exp\left\{M_{R\alpha}^*(\btheta)\Big|\Dn,{\ttheta_{s_0,R}^{0*}}\right\}&=M_n(\btheta),\\
\Var\left\{M_{R\alpha}^*(\btheta)\Big|\Dn,{\ttheta_{s_0,R}^{0*}}\right\}&\leq\frac{1}{s_nn^2}\sumn\frac{m^2(Z_i,\btheta)}{\tilde{\pi}_{n,R\alpha i}^{\mvc*}}
\leq\frac{1}{s_nn}\sumn\frac{m^2(Z_i,\btheta)}{\alpha}
=O_P(s_n^{-1}).
\end{align*}
By Chebyshev's inequality, for each $\btheta$, we have
\begin{align*}
M_{R\alpha}^*(\btheta)-M_n(\btheta)=o_{P|\Dn,\ttheta_{s_0,R}^{0*}}(1).
\end{align*}
Under Assumptions \ref{ass:1} and \ref{ass:2}, the parameter space is compact and $\htheta_n$ is the unique global maximum of the continuous concave function $M_n(\btheta)$. Thus from Theorem 5.9 and its remark of \cite{Vaart:98}, conditionally on $\Dn$ and ${\ttheta_{s_0,R}^{0*}}$,
\begin{align*}
\|\ttheta_{s_n,R}^\alpha-\htheta_n\|=\op.
\end{align*}

By Taylor expansion
\begin{align*}
  0= \dM_{R\alpha}^*(\ttheta_{s_n,R}^\alpha)
 &=\dM_{R\alpha}^*(\htheta_n)+B_{s_n}^{\ttheta_{s_0,R}^{0*}}(\ttheta_{s_n,R}^\alpha-\htheta_n),
 \end{align*}
so
\begin{align*}
  \ttheta_{s_n,R}^\alpha-\htheta_n
  &=-\Big(B_{s_n}^{\ttheta_{s_0,R}^{0*}}\Big)^{-1}\dM_{R\alpha}^*(\htheta_n)\\
&=-\frac{1}{\sqrt{s_n}}\Big(B_{s_n}^{\ttheta_{s_0,R}^{0*}}\Big)^{-1}\{\Lambda_R^\alpha(\ttheta_{s_0,R}^{0*})\}^{1/2}\sqrt{s_n}\{\Lambda_R^\alpha(\ttheta_{s_0,R}^{0*}) \}^{-1/2}\dM_{R\alpha}^*(\htheta_n).
\end{align*}
Therefore, from Lemma \ref{lem:two-step1} and Lemma \ref{lem:PoiDistributionofdM}, conditional on $\Dn,{\ttheta_{s_0,R}^{0*}}$, by Slutsky's theorem
\begin{align}\label{eq:19}
  \sqrt{s_n}\{\Lambda_R^\alpha(\ttheta_{s_0,R}^{0*})\}^{-1/2}\ddM_n(\htheta_n)(\ttheta_{s_n,R}^\alpha-\htheta_n)\rightarrow \Nor\left(\0, \I\right),
\end{align}
in conditional distribution. 

Next, we check the distance bewtween $ \Lambda_R^\alpha(\ttheta_{s_0,R}^{0*})$ and $\Lambda_R^\alpha(\htheta_n)$.
\begin{align}\label{eq:Lamn-Lamn}
&\|\Lambda_R^\alpha(\ttheta_{s_0,R}^{0*})-\Lambda_R^\alpha(\htheta_n)\|\notag\\
&=\left\|\frac{1}{n^2}\sumn\frac{\dm(Z_i,\htheta_n)\dm\tp(Z_i,\htheta_n)}{(1-\alpha)\pi_{R i}^{\mvc}({\ttheta_{s_0,R}^{0*}})+\alpha\onen}-\frac{1}{n^2}\sumn\frac{\dm(Z_i,\htheta_n)\dm\tp(Z_i,\htheta_n)}{(1-\alpha)\pi_{R i}^{\mvc}(\htheta_n)+\alpha\onen} \right\|\notag\\
&\leq\frac{1}{n^2}\sumn \|\dm(Z_i,\htheta_n)\|^2
    \left|\frac{1}{(1-\alpha)\pi_{R i}^{\mvc}({\ttheta_{s_0,R}^{0*}})+\alpha\onen}
    -\frac{1}{(1-\alpha)\pi_{R i}^{\mvc}(\htheta_n)+\alpha\onen} \right|\notag\\
&< \frac{1}{\alpha^2}\sumn \|\dm(Z_i,\htheta_n)\|^2\ \Big|
\pi_{R i}^{\mvc}({\ttheta_{s_0,R}^{0*}})-\pi_{R i}^{\mvc}(\htheta_n)\Big|\notag\\
  &\le\frac{1}{\alpha^2}\sumn
    \|\dm(Z_i,\htheta_n)\|^2\Bigg\{
    \frac{\Big|\|\dm(Z_i,\ttheta_{s_0,R}^{0*})\|
    -\|\dm(Z_i,{\htheta_n})\|\Big|}
    {\sumjn\|\dm(Z_j,\ttheta_{s_0,R}^{0*})\|}\notag\\
  &\hspace{6cm}
 +\|\dm(Z_i,\htheta_n)\|\frac{\sumjn\Big|\|\dm(Z_j,\ttheta_{s_0,R}^{0*})\|-\|\dm(Z_j,\htheta_n)\|\Big|}
     {\sumjn\|\dm(Z_j,{\htheta_n})\|\sumjn\|\dm(Z_j,{\ttheta_{s_0,R}^{0*}})\|}\Bigg\}\notag\\
&\equiv\frac{1}{\alpha^2}\sumn \|\dm(Z_i,\htheta_n)\|^2\left(\Delta_{1i}+\Delta_{2i} \right).
\end{align}
Under Assumption \ref{ass:3}, for any $j=1,2,...,n$
\begin{align}\label{eq:opteq2}
&\left|\|\dm(Z_j,\htheta_n)\|-\|\dm(Z_j,{\ttheta_{s_0,R}^{0*}})\| \right|
\leq\|\dm(Z_j,\htheta_n)-\dm(Z_j,{\ttheta_{s_0,R}^{0*}}) \|\notag\\
&\leq\sqrt{\sum_{k=1}^{d}\{\dm_k(Z_j,\htheta_n) -\dm_k(Z_j,{\ttheta_{s_0,R}^{0*}})\}^2}
\leq\sum_{k=1}^{d}\left|\dm_k(Z_j,\htheta_n) -\dm_k(Z_j,{\ttheta_{s_0,R}^{0*}})\right| \notag\\
&\leq\sum_{k=1}^{d}\left|\ddm_{k}^T(Z_j,\xi_k)(\htheta_n-{\ttheta_{s_0,R}^{0*}})\right|
\leq\|\htheta_n-{\ttheta_{s_0,R}^{0*}}\|\sum_{k=1}^{d}\|\ddm_k(Z_j,\xi_k)\|
\equiv\|\htheta_n-{\ttheta_{s_0,R}^{0*}}\|h(Z_j),
\end{align}
where $\dm_k(Z_j,\htheta_n)$ is the $k$th element of $\dm(Z_j,\htheta_n)$, $\ddm_k(Z_j,\htheta_n)$ is the $k$th column of $\ddm(Z_j,\htheta_n)$, and all $\xi_k$ are between $\htheta_n$ and $\ttheta_{s_0,R}^{0*}$.  
Thus,
\begin{align}
  \Delta_{1i}\le\frac{\|\htheta_n-\ttheta_{s_0,R}^{0*}\|h(Z_i)}
  {\sumjn\|\dm(Z_j,\ttheta_{s_0,R}^{0*})\|},
\end{align}
and
\begin{align}
  &\Delta_{2i}\le 
    \frac{\|\dm(Z_i,\htheta_n)\|
    \|\htheta_n-{\ttheta_{s_0,R}^{0*}}\|\sumjn h(Z_j)}
    {\sumjn\|\dm(Z_j,{\htheta_n})\|\sumjn\|\dm(Z_j,{\ttheta_{s_0,R}^{0*}})\|}
\end{align}

From \eqref{eq:ddmOP1} and Assumption \ref{ass:3}
\begin{align}
&\onen\sumjn h^2(Z_j)
\leq d\onen\sumjn\sum_{k=1}^{d}\|\ddm_k(Z_j,\xi_k)\|^2
= d\onen\sumjn\sum_{k=1}^{d}\sum_{l=1}^{d}\ddm_{k,l}^2(Z_j,\xi_k)\notag \\
&\leq d\onen\sumjn\sum_{k=1}^{d}\sum_{l=1}^{d}\left(2\ddm_{k,l}^2(Z_j,\htheta_n)+2\psi^2(Z_j)\|\ttheta_{s_0,R}^{0*}-\htheta_n\|^2\right)
=O_P(1)\label{eq:h(Z)OP1}
\end{align}
which also implies that $\onen\sumjn h(Z_j)=\Op$. Thus, 

\begin{align}
  \sumn\|\dm(Z_i,\htheta_n)\|\Delta_{1i}
  & \le\frac{O_P(\|\htheta_n-\ttheta_{s_0,R}^{0*}\|)}{n}
    \sumn\|\dm(Z_i,\htheta_n)\|^2h(Z_i)\notag\\
  &\le O_P(\|\htheta_n-\ttheta_{s_0,R}^{0*}\|)
    \bigg\{\onen\sumn\|\dm(Z_i,\htheta_n)\|^4\bigg\}^{\frac{1}{2}}
    \bigg\{\onen\sumn h^2(Z_i)\bigg\}^{\frac{1}{2}},\label{eq:Delta1}
\end{align}
and
\begin{align}
  \sumn\|\dm(Z_i,\htheta_n)\|\Delta_{2i}
  & =O_P(\|\htheta_n-\ttheta_{s_0,R}^{0*}\|)\onen
    \sumn\|\dm(Z_i,\htheta_n)\|^2\label{eq:Delta2}
\end{align}

Combining \eqref{eq:Lamn-Lamn}, \eqref{eq:Delta1}, and \eqref{eq:Delta2}, %
we obtain that for large $s_0$, $s_n$ and $n$,
\begin{align*}
&\|\Lambda_R^\alpha(\ttheta_{s_0,R}^{0*})-\Lambda_R^\alpha(\htheta_n)\|
=\|\htheta_n-{\ttheta_{s_0,R}^{0*}}\|\Op=\op.
\end{align*}

Thus, Slutsky's theorem and \eqref{eq:19} indicate that given $\Dn$ and $\ttheta_{s_0,R}^{0*}$, as $s_0, s_n$ and $n\to\infty$
\begin{align*}
  \sqrt{s_n}\{V_{n,R}^\alpha(\htheta_n)\}^{-1/2}(\ttheta_{s_n,R}^\alpha-\htheta_n)\rightarrow \Nor\left(\0, \I\right),
\end{align*}
in conditional distribution. %
\end{proof}

\subsection{Proof of Theorem \ref{thm:6}}
\addcontentsline{toc}{subsection}{Appendix F:Proof of Theorem \ref{thm:6}}
The proof of Theorem~\ref{thm:6} relies on Lemmas~\ref{lem:plem1}, \ref{lem:poiddMr} and \ref{lem:H}.
\begin{lemma}\label{lem:plem1}
Under Assumptions \ref{ass:4},  conditional on $\Dn$ and ${\ttheta_{s_0,R}^{0*}}$, then
  \begin{align*}
    \sqrt{s_n}\{\Lambda_{n,P}^\alpha(\ttheta_{s_0,P}^{0*}) \}^{-1/2}\dM_{P\alpha}^*(\htheta_n) \rightarrow \Nor(\0,\I),
  \end{align*}
in conditional distribution, where
\begin{align*}
  \Lambda_{n,P}^\alpha(\ttheta_{s_0,P}^{0*})
  &=\frac{s_n}{n^2}\sumn
    \frac{\{1-(s_n\tilde{\pi}_{n,P\alpha i}^{\mvc})\wedge1\}
    \dm(Z_i,\htheta_n)\dm\tp(Z_i,\htheta_n)}
    {(s_n\tilde{\pi}_{n,P\alpha i}^{\mvc})\wedge1}.%
\end{align*}
\end{lemma}
\begin{proof}
 For the sake of readability, in the sequel, we redefine $\nu_i$ as $\nu_i=I(u_i\leq s_n\pi_{n,P\alpha i}^\opt)$ and let 
\begin{align}
  \sqrt{s_n}\dM_{P\alpha}^*(\htheta_n)
  =\sumn\frac{\nu_i\sqrt{s_n}\dm(Z_i,\htheta_n)}
  {n\{(s_n\tilde{\pi}_{n,P\alpha i}^{\mvc})\wedge1\}}
  \equiv \sumn\bm{\eta}_i^{\ttheta_{s_0,P}^{0*}}.
\end{align}
From direct calculation and the definition of $\htheta_n$, we have
\begin{align*}
  \Exp\Big(\sqrt{s_n}\dM_{P\alpha}^*(\htheta_n)\Big|\Dn,\ttheta_{s_0,P}^{0*}\Big)
&=\frac{\sqrt{s_n}}{n}\sumn\dm(Z_i,\htheta_n)=\0,
\end{align*}
and
  \begin{align*}
  \Var\Big(\sqrt{s_n}\dM_{P\alpha}^*(\htheta_n)\Big|\Dn,\ttheta_{s_0,P}^{0*}\Big)
    &=\frac{s_n}{n^2}\sumn\frac{\Var(\nu_i|\Dn,{\ttheta_{s_0,P}^{0*}})
      \dm(Z_i,\htheta_n)\dm\tp(Z_i,\htheta_n)}
      {\{(s_n\tilde{\pi}_{n,P\alpha i}^{\mvc})\wedge1\}^2}\\
    &=\frac{s_n}{n^2}\sumn\frac{
      \{1-(s_n\tilde{\pi}_{n,P\alpha i}^{\mvc})\wedge1\}
      \dm(Z_i,\htheta_n)\dm\tp(Z_i,\htheta_n)}
      {(s_n\tilde{\pi}_{n,P\alpha i}^{\mvc})\wedge1}\\
  &\leq\frac{1}{\alpha n}\sumn\dm(Z_i,\htheta_n)\dm\tp(Z_i,\htheta_n)
  =\Op.
\end{align*}
Next, we check Lindeberg's condition. For any $\epsilon>0$ and $\rho\in (0,2]$, 
\begin{align*}
  &\Exp\left\{\sumn\|\bm{\eta}_i^{\ttheta_{s_0,P}^{0*}}\|
    I(\|\bm{\eta}_i^{\ttheta_{s_0,P}^{0*}}\|>\varepsilon)
    \Bigg|\Dn,{\ttheta_{s_0,P}^{0*}}\right\}\\
  &\leq\frac{1}{\varepsilon^\rho}\sumn
    \Exp\left\{\|\bm{\eta}_i^{\ttheta_{s_0,P}^{0*}}\|^{2+\rho}
    I(\|\bm{\eta}_i^{\ttheta_{s_0,P}^{0*}}\|>\varepsilon)
    \Big|\Dn,{\ttheta_{s_0,P}^{0*}}\right\}\\
  &\leq\frac{1}{\varepsilon^\rho}\sumn
    \Exp\left(\|\bm{\eta}_i^{\ttheta_{s_0,P}^{0*}}\|^{2+\rho}
    \Big|\Dn,{\ttheta_{s_0,P}^{0*}}\right)
  =\frac{s_n^{1+\rho/2}}{\varepsilon^{\rho}n^{2+\rho}}\sumn
    \frac{\|\dm(Z_i,\htheta_n)\|^{2+\rho}}
    {\{(s_n\tilde{\pi}_{n,P\alpha i}^{\mvc})\wedge1\}^{1+\rho}}\\
  &\le\frac{s_n^{1+\rho/2}}{\varepsilon^{\rho}n^{2+\rho}}\sumn
    \frac{\|\dm(Z_i,\htheta_n)\|^{2+\rho}}
    {(s_n\alpha/n)^{1+\rho}}
 = \frac{1}{\alpha^{1+\rho}\varepsilon^\rho {s_n}^{\rho/2}}\onen\sumn\|\dm(Z_i,\htheta_n)\|^{2+\rho}
  =O_P({s_n}^{-\rho/2}).
\end{align*}
Thus, from the Lindeberg-Feller Central Limit Theorem \citep[cf.][]{Vaart:98}, Lemma~\ref{lem:plem1} follows. %
\end{proof}

\begin{lemma}\label{lem:poiddMr}
  Under Assumption \ref{ass:3},  for any $\bu_{s_n}=\op$, conditional on $\Dn$ and ${\ttheta_{s_0,P}^{0*}}$, 
\begin{equation*}
  \onen\sumn\frac{\nu_i\ddm(Z_i,\htheta_n+\bu_{s_n})}
  {(s_n\tilde{\pi}_{n,P\alpha i}^{\mvc})\wedge1}
  -\onen\sumn\ddm(Z_i,\htheta_n)=o_P(1).
\end{equation*}
\end{lemma}
\begin{proof}
  First, using an approach similar to the one used to prove \eqref{eq:lem6Op}, we can show that given $\Dn$ and $\ttheta_{s_0,P}^{0*}$, 
\begin{align}
  \onen\sumn\frac{\nu_i\psi(Z_i)}
  {(s_n\tilde{\pi}_{n,P\alpha i}^{\mvc})\wedge1}=\Op.
\end{align}
For every $k,l=1,2,...,d$, from Lipschitz continuity, we have
\begin{align}\label{eq:lem6step1}
  \left|\onen\sumn\frac{\nu_i\ddm_{k,l}(Z_i,\htheta_n+\bu_{s_n})}
  {(s_n\tilde{\pi}_{n,P\alpha i}^{\mvc})\wedge1}
  -\onen\sumn\frac{\nu_i\ddm_{k,l}(Z_i,\htheta_n)}
  {(s_n\tilde{\pi}_{n,P\alpha i}^{\mvc})\wedge1}\right|
  \leq\onen\sumn\frac{\nu_i\psi(Z_i)\|\bu_{s_n}\|}
  {(s_n\tilde{\pi}_{n,P\alpha i}^{\mvc})\wedge1}
  =o_P(1).
\end{align}
For each $k,l=1,2,...,d$, direct calculations show that
\begin{align*}
  \Exp\left\{\onen\sumn\frac{\nu_i\ddm_{k,l}(Z_i,\htheta_n)}
  {(s_n\tilde{\pi}_{n,P\alpha i}^{\mvc})\wedge1}
  \Bigg|\Dn,{\ttheta_{s_0,P}^{0*}}\right\}
  &=\onen\sumn\ddm_{k,l}(Z_i,\htheta_n),\\
  \Var\left\{\onen\sumn\frac{\nu_i\ddm_{k,l}(Z_i,\htheta_n)}
  {(s_n\tilde{\pi}_{n,P\alpha i}^{\mvc})\wedge1}
  \Bigg|\Dn,{\ttheta_{s_0,P}^{0*}} \right\}
  &\leq \frac{1}{s_nn^2}\sumn\frac{\ddm^2_{k,l}(Z_i,\htheta_n)}
    {(s_n\tilde{\pi}_{n,P\alpha i}^{\mvc})\wedge1}
  \leq \frac{1}{\alpha s_n n}\sumn h^2(Z_i)
=O_{P}(s_n^{-1}).
\end{align*}
According to Chebyshev's inequality, we obtain
\begin{align}\label{eq:lem6step2}
  \onen\sumn\frac{\nu_i\ddm(Z_i,\htheta_n)}
  {(s_n\tilde{\pi}_{n,P\alpha i}^{\mvc})\wedge1}
  -\onen\sumn\ddm(Z_i,\htheta_n)=O_{P}({s_n}^{-1/2}).
\end{align}
Therefore, combining \eqref{eq:lem6step1} and \eqref{eq:lem6step2}, we have
\begin{equation*}%
\onen\sumn\frac{\nu_i\ddm(Z_i,\htheta_n+\bu_{s_n})}{s_n\pi_{n,P\alpha i}^{\mvc}}-\onen\sumn\ddm(Z_i,\htheta_n)=o_P(1).
\end{equation*}
\end{proof}

\begin{lemma}\label{lem:H}
  Under Assumptions \ref{ass:3} and \ref{ass:4},
  \begin{enumerate}[1)]
  \item if $\varrho_n=s_n/(bn)\rightarrow\varrho\in(0,1)$, then $H^{0*}-H_{\varrho_n}=\op$;
  \item $\Psi^{0*}-\Psi_{\varrho_n}=\op$, where
  \begin{align}
      &\Psi_{\varrho_n}
  =\onen\sumn\{\|\dm(Z_i,\htheta_n)\|\wedge H_{\varrho_n}\};
  \end{align}
\item if $s_n/(bn)\rightarrow\varrho=0$, then $\Psi^{0*}-\Psi_{\infty}=\op$.
\end{enumerate}
\end{lemma}
\begin{proof}
  Note that $H^{0*}$ is the $\ceil{s_0^*-s_0^*s_n/b/n}$-th order statistics of %
 $\|\dm(Z_i^{0*},\ttheta_{s_0,P}^{0*})\|$, $i=1, ..., s_0^*$.
  For any $\rho>0$, let $\tilde{H}_{\rho}$ be the $\ceil{n(1-\rho)}$-th order statistics of %
  $\|\dm(Z_i,\ttheta_{s_0,P}^{0*})\|$, $i=1, ..., n$.
Let $\nu_{(i)}^0=1$ if $\|\dm(Z,\ttheta_{s_0,P}^{0*})\|_{(i)}$ is included in $\|\dm(Z_1^{0*},\ttheta_{s_0,P}^{0*})\|, ..., \|\dm(Z_{s_{0}^*}^{0*},\ttheta_{s_0,P}^{0*})\|$, and $\nu_{(i)}^0=0$ otherwize. 
For any $\varrho_+>\varrho$, 
\begin{align}
  \Pr(H^{0*}\le\tilde{H}_{\varrho_+})
  =\Pr\Bigg(\sum_{i=1}^{\ceil{n(1-\varrho_+)}}\nu_{(i)}^0
  \ge\ceil{s_0^*-s_0^*s_n/b/n}\Bigg).
\end{align}
Note that
\begin{align}
  \frac{1}{s_0}\sum_{i=1}^{\ceil{n(1-\varrho_+)}}\nu_{(i)}^0
  =1-\varrho_++\op\quad\text{and}\quad
  \frac{\ceil{s_0^*-s_0^*s_n/b/n}}{s_0} =1-\varrho+\op.
\end{align}
Thus,
\begin{align}\label{eq:20}
  \Pr(H^{0*}\le\tilde{H}_{\varrho_+})\rightarrow0.
\end{align}
Similarly, we obtain that for any $\varrho_-<\varrho$, 
\begin{align}\label{eq:21}
  \Pr(H^{0*}\le\tilde{H}_{\varrho_-})\rightarrow1.
\end{align}
Note that $\tilde{H}_{\varrho_+}$ is between the $\ceil{n(1-\varrho_+)}-s_0^*$-th and the $\ceil{n(1-\varrho_+)}$-th order statistics of $\|\dm(Z_i,\ttheta_{s_0,P}^{0*})\|$'s that are not included in $\|\dm(Z_1^{0*},\ttheta_{s_0,P}^{0*})\|, ..., \|\dm(Z_{s_{0}^*}^{0*},\ttheta_{s_0,P}^{0*})\|$. The joint distribution of these $\|\dm(Z_i,\ttheta_{s_0,P}^{0*})\|$'s are exchangeable, and $s_0^*/n\rightarrow0$ in probability. Therefore, both the $\ceil{n(1-\varrho_+)}-s_0^*$-th and the $\ceil{n(1-\varrho_+)}$-th order statistics of these $\|\dm(Z_i,\ttheta_{s_0,P}^{0*})\|$'s converge to the $\varrho_+$-quantile of the distribution of $\|\dm(Z_i,\btheta_0)\|$ in probability \citep{chanda1971asymptotic}, where %
$\btheta_0=\arg\max_\btheta\Exp\{\m(Z,\btheta)\}$. As a result, $\tilde{H}_{\varrho_+}$ converge in probability to the $\varrho_+$-quantile of the distribution of $\|\dm(Z,\btheta_0)\|$, say $\zeta_{\varrho+}$. Similarly, $\tilde{H}_{\varrho_-}$ converge in probability to the $\varrho_-$-quantile of the distribution of $\|\dm(Z,\btheta_0)\|$, say $\zeta_{\varrho-}$. Thus, \eqref{eq:20} and \eqref{eq:21} together imply that for any $\epsilon>0$,
\begin{align}
  \Pr(\zeta_{\varrho_+}-\epsilon< H^{0*}
  <\zeta_{\varrho_-}+\epsilon)\rightarrow1.
\end{align}
Since the distribution of $Z$ is continuous and so is that of $\|\dm(Z,\btheta_0)\|$, we can choose $\varrho_+$ and $\varrho_-$ close enough to $\varrho$ such that $\zeta_{\varrho_-}-\zeta_{\varrho}<\epsilon$ and $\zeta_{\varrho}-\zeta_{\varrho_+}<\epsilon$, which implies that
\begin{align}
  \Pr(\zeta_{\varrho}-2\epsilon< H^{0*}
  <\zeta_{\varrho}+2\epsilon)\rightarrow1,
\end{align}
for any $\epsilon$. Thus, $H^{0*}=\zeta_{\varrho}+\op$. 
Since $\|\dm(Z_1,\htheta_n)\|, ..., \|\dm(Z_n,\htheta_n)\|$ are exchangeable, $H_{\varrho_n}=\zeta_{\varrho}+\op$, where $H_{\varrho_n}$ is the $\ceil{n(1-\varrho_n)}$-th order statistics of %
$\|\dm(Z_i,\htheta_n)\|$, $i=1, ..., n$. Therefore, $H^{0*}-H_{\varrho_n}=\op$.

Now we prove 2) of Lemma~\ref{lem:H}. If $\varrho=0$ and $\|\dm(Z,\btheta)\|$ is bounded, then 
\begin{align*}
  \Psi^{0*}
  &=\sum_{i=1}^{\ceil{s_0^*-s_0^*s_n/b/n}}
  \frac{\|\dm(Z^{0*},\ttheta_{s_0,P}^{0*})\|_{(i)}}{s_0^*}
  +\frac{s_0^*-\ceil{s_0^*-s_0^*s_n/b/n}}{s_0^*}H^{0*}\\
  &=\sum_{i=1}^{s_0^*}\frac{\|\dm(Z_i^{0*},\ttheta_{s_0,P}^{0*})\|}{s_0^*}
    +\op,
\end{align*}
and similarly,
\begin{align*}
  \Psi_{\varrho_n}=\onen\sumn\|\dm(Z_i,\htheta_n)\|+\op.
\end{align*}
Thus the proof reduce to prove that
\begin{align*}
  \sum_{i=1}^{s_0^*}\frac{\|\dm(Z_i^{0*},\ttheta_{s_0,P}^{0*})\|}{s_0^*}
  =\onen\sumn\|\dm(Z_i,\htheta_n)\|+\op,
\end{align*}
which can be proved by Taylor's expansion and Markov's inequality. 
To prove other cases, let $\nu_i^0=1$ if the $i$-th observation is included in the pilot subsample and $\nu_i^0=0$ otherwise; then $\Psi^{0*}$ can be written as 
\begin{align*}
  \Psi^{0*}=\frac{1}{s_0^*}\sumn\nu_i^0
  \{\|\dm(Z_i,\ttheta_{s_0,P}^{0*})\|\wedge H^{0*}\}.
\end{align*}
Define
\begin{align*}
  \Psi^{0*}_{H_{\varrho_n}}
  &=\frac{1}{s_0^*}\sumn\nu_i^0\{\|\dm(Z_i,\ttheta_{s_0,P}^{0*})\|\wedge
    H_{\varrho_n}\}\quad\text{and}\quad
  \Psi^{0*}_{\htheta_n}
  =\frac{1}{s_0^*}\sumn\nu_i^0\{\|\dm(Z_i,\htheta_n)\|\wedge
    H_{\varrho_n}\}.
\end{align*}
If $\varrho>0$, then
\begin{align*}
  |\Psi^{0*}-\Psi^{0*}_{H_{\varrho_n}}|
  &=\frac{1}{s_0^*}\sumn\nu_i^0\Big|\|\dm(Z_i,\ttheta_{s_0,P}^{0*})\|\wedge
    H^{0*}-\|\dm(Z_i,\ttheta_{s_0,P}^{0*})\|\wedge H_{\varrho_n}\Big|\\
  &\le\frac{|H^{0*}-H_{\varrho_n}|}{s_0^*}\sumn\nu_i^0
    I\Big\{\|\dm(Z_i,\ttheta_{s_0,P}^{0*})\|\ge
    H^{0*}\wedge H_{\varrho_n}\Big\}
    \le|H^{0*}-H_{\varrho_n}|=\op.
\end{align*}

If $\varrho=0$ and $\|\dm(Z,\btheta)\|$ is unbounded, then $H^{0*}\wedge H_{\varrho_n}\rightarrow\infty$ in probability. Under Assumptions \ref{ass:3} and \ref{ass:4}, it can be shown that $\frac{1}{s_0^*}\sumn\nu_i^0\|\dm(Z_i,\ttheta_{s_0,P}^{0*})\|^2=\OpD(1)$. Thus,
\begin{align}
  |\Psi^{0*}-\Psi^{0*}_{H_{\varrho_n}}|
  &\le\frac{1}{s_0^*}\sumn\nu_i^0\|\dm(Z_i,\ttheta_{s_0,P}^{0*})\|
    I\Big\{\|\dm(Z_i,\ttheta_{s_0,P}^{0*})\|\ge
    H^{0*}\wedge H_{\varrho_n}\Big\}\notag\\
  &+\frac{H^{0*}}{s_0^*}\sumn\nu_i^0
    I\Big\{\|\dm(Z_i,\ttheta_{s_0,P}^{0*})\|\ge H^{0*}\Big\}
  +\frac{H_{\varrho_n}}{s_0^*}\sumn\nu_i^0
    I\Big\{\|\dm(Z_i,\ttheta_{s_0,P}^{0*})\|\ge H_{\varrho_n}\Big\}\notag\\
  &\le\Bigg\{\frac{1}{H^{0*}\wedge H_{\varrho_n}}
    +\frac{1}{H^{0*}}+\frac{1}{H_{\varrho_n}}\Bigg\}
    \frac{1}{s_0^*}\sumn\nu_i^0\|\dm(Z_i,\ttheta_{s_0,P}^{0*})\|^2
    =\op.
\end{align}
Furthermore, we can show that
\begin{align*}
  |\Psi^{0*}_{H_{\varrho_n}}-\Psi^{0*}_{\htheta_n}|
  &\le\frac{1}{s_0^*}\sumn\nu_i^0\{\|\dm(Z_i,\ttheta_{s_0,P}^{0*})\|
    -\|\dm(Z_i,\htheta_n)\|\}\\
  &\le\frac{\|\htheta_n-\ttheta_{s_0,P}^{0*}\|}{s_0^*}\sumn\nu_i^0
  h(Z_i)=\op
\end{align*}
and 
\begin{align*}
  |\Psi^{0*}_{\htheta_n}-\Psi_{\varrho_n}|=\op,
\end{align*}
where the last two $\op$ are obtained by mean and variance calculations under the conditional distribution of $\nu_i^0$'s. Thus, we have that 
\begin{align}
  |\Psi^{0*}-\Psi_{\varrho_n}|=\op. 
\end{align}

With 2) of Lemma~\ref{lem:H} proved, in order to prove 3), we only need to show that $\Psi_{\infty}-\Psi_{\varrho_n}=\op$ if $s_n/(bn)\rightarrow\varrho=0$. This is true because if $\|\dm(Z,\btheta)\|$ is bounded, then
\begin{align*}
  |\Psi_{\infty}-\Psi_{\varrho_n}|
  &\le\onen\sumn\Big|\|\dm(Z_i,\htheta_n)\|-
    \|\dm(Z_i,\htheta_n)\|\wedge H_{\varrho_n}\Big|\\
  &\le\frac{n-\ceil{n(1-\varrho_n)}}{n}\|\dm(Z,\htheta_n)\|_{(n)}
    =\op;
\end{align*}
otherwise,
\begin{align*}
  |\Psi_{\infty}-\Psi_{\varrho_n}|
  &\le\onen\sumn\Big|\|\dm(Z_i,\htheta_n)\|-
    \|\dm(Z_i,\htheta_n)\|\wedge H_{\varrho_n}\Big|\\
  &\le\onen\sumn\|\dm(Z_i,\htheta_n)\|I\Big\{
    \|\dm(Z_i,\htheta_n)\|\ge H_{\varrho_n}\Big\}\\
  &\le\frac{1}{nH_{\varrho_n}}\sumn\|\dm(Z_i,\htheta_n)\|^2
    =\op.
\end{align*}

\end{proof}

\begin{proof}[Proof of Theorem~\ref{thm:6}]
  For Algorithm \ref{alg:4}, $M_{P\alpha}^*(\btheta)$ can be written as
\begin{align*}
  M_{P\alpha}^*(\btheta)%
  &=\onen\sumn\frac{\nu_im(Z_i,\btheta)}
    {(s_n\tilde{\pi}_{n,P\alpha i}^{\mvc})\wedge1}.
\end{align*}
Denote
\begin{align*}
  \gamma_{{\ttheta_{s_0,P}^{0*}} P}(\bu)
  &=s_nM_{P\alpha}^*(\htheta_n+\bu/\sqrt{s_n})-s_nM_{P\alpha}^*(\htheta_n).
\end{align*}
Under Assumption \ref{ass:2}, $\sqrt{s_n}({\ttheta_{s_n,P}^\alpha}-\htheta_n)$ is the unique maximizer of $\gamma_{{\ttheta_{s_0,P}^{0*}} P}(\bu)$. 
By Taylor's expansion,
\begin{align*}
  \gamma_{{\ttheta_{s_0,P}^{0*}} P}(\bu)
  &=\sqrt{s_n}\bu^T\dM_{P\alpha}^*(\htheta_n)
    +\frac{\bu^T\ddM_{P\alpha}^*(\htheta_n+\bau/\sqrt{s_n})\bu}{2}
\end{align*}
where $\bau$ lies between $\0$ and $\bu$.
From Lemma \ref{lem:plem1}, $\sqrt{s_n}\dM_{P\alpha}^*(\htheta_n)$ is stochastically bounded in conditional probability given $\Dn$ and $\ttheta_{s_0,P}^{0*}$; from Lemma \ref{lem:poiddMr}, conditional on $\Dn$ and ${\ttheta_{s_0,P}^{0*}}$, $\ddM_{P\alpha}^*(\htheta_n+\bau/\sqrt{s_n})-\ddM_n(\htheta_n)=\op$ and $\ddM_n(\htheta_n)$ converges to a positive-definite matrix.  
Thus, from the Basic Corollary in page 2 of \cite{hjort2011asymptotics}, the minimizer of $s_n\gamma(\bu), \sqrt{s_n}(\ttheta_{s_n,P}^\alpha -\htheta_n)$, satisfies that
\begin{align}
  \sqrt{s_n}(\ttheta_{s_n,P}^\alpha -\htheta_n)
  =\ddM_n^{-1}(\htheta_n)\sqrt{s_n}\dM^*_P(\htheta_n)+\op,
\end{align}
which implies that
\begin{align}
  \sqrt{s_n}\{\Lambda_{n,P}^\alpha(\ttheta_{s_0,P}^{0*})\}^{-1/2}
  \ddM_n(\htheta_n)(\ttheta_{s_n,P}^\alpha -\htheta_n)\rightarrow \Nor(\0,\I),
\end{align}
in conditional distribution given $\Dn$ and $\ttheta_{s_0,P}^{0*}$.

Next, we will check the distance between $\Lambda_{n,P}^\alpha(\ttheta_{s_0,P}^{0*})$ and $\Lambda_{n,P}^\alpha(\htheta_n)$. 
Let $\Lambda_{P\varrho_n}^{\alpha}(\htheta_n)$ have the same expression as $\Lambda_{n,P}^{\alpha}(\htheta_n)$ in \eqref{eq:30} except that $\pi_{n,Pi}^{\mvc}(\htheta_n)$ in the denominator is replaced by
\begin{equation*}
  \pi_{n,P\alpha i}^{\varrho_n}(\htheta_n)
  =(1-\alpha)\pi_{n,Pi}^{\varrho_n}(\htheta_n)+\alpha\onen
  \quad\text{ with }\quad
  \pi_{n,Pi}^{\varrho_n}
  =\frac{\|\dm(Z_i,\htheta_n)\|\wedge H_{\varrho_n}}
  {\sumjn\{\|\dm(Z_j,\htheta_n)\|\wedge H_{\varrho_n}\}}.
\end{equation*}
We have that 
\begin{align}
  \|\Lambda_{n,P}^\alpha(\ttheta_{s_0,P}^{0*})
    -\Lambda_{P\varrho_n}^\alpha(\htheta_n)\|\notag
  &\leq\frac{s_n}{n^2}\sumn\left|\frac{\|\dm(Z_i,\htheta_n)\|^2}
    {\{s_n\tilde{\pi}_{n,P\alpha i}^{\mvc}(\ttheta_{s_0,P}^{0*})\}\wedge1}
    -\frac{\|\dm(Z_i,\htheta_n)\|^2}
    {\{s_n\pi_{n,P\alpha i}^{\varrho_n}(\htheta_n)\}\wedge1}
    \right|\notag\\
  &=\frac{s_n}{n^2}\sumn\|\dm(Z_i,\htheta_n)\|^2\left|\frac{1}
    {\{s_n\tilde{\pi}_{n,P\alpha i}^{\mvc}(\ttheta_{s_0,P}^{0*})\}\wedge1}
    -\frac{1}{\{s_n\pi_{n,P\alpha i}^{\varrho_n}(\htheta_n)\}\wedge1}
    \right|\notag\\
  &=\frac{s_n}{n^2}\sumn\|\dm(Z_i,\htheta_n)\|^2\left|
    \frac{\{s_n\tilde{\pi}_{n,P\alpha i}^{\mvc}(\ttheta_{s_0,P}^{0*})\}\wedge1
    -\{s_n\pi_{n,P\alpha i}^{\varrho_n}(\htheta_n)\}\wedge1}
    {[\{s_n\tilde{\pi}_{n,P\alpha i}^{\mvc}(\ttheta_{s_0,P}^{0*})\}\wedge1]
    [\{s_n\pi_{n,P\alpha i}^{\varrho_n}(\htheta_n)\}\wedge1]}
    \right|\notag\\
  &<\frac{1}{\alpha^2}\sumn\|\dm(Z_i,\htheta_n)\|^2\left|
    \tilde{\pi}_{n,Pi}^{\mvc}(\ttheta_{s_0,P}^{0*})
    -\pi_{n,Pi}^{\varrho_n}(\htheta_n)\right|\label{eq:70}
\end{align}
If $\varrho>0$, then from
\begin{align*}
  &n\left|\tilde{\pi}_{n,Pi}^{\mvc}(\ttheta_{s_0,P}^{0*})
    -\pi_{n,Pi}^{\varrho_n}(\htheta_n)\right|\notag\\
  &=\left|\frac{\|\dm(Z_i,\ttheta_{s_0,P}^{0*})\|\wedge H^{0*}}
    {\Psi^{0*}_{\varrho_n}}
    -\frac{\|\dm(Z_i,\htheta_n)\|\wedge H_{\varrho_n}}
    {\Psi_{\varrho_n}}\right|\notag\\
  &\le\left|
    \frac{\|\dm(Z_i,\ttheta_{s_0,P}^{0*})\|\wedge H^{0*}
    -\|\dm(Z_i,\htheta_n)\|\wedge H_{\varrho_n}}
    {\Psi^{0*}_{\varrho_n}}\right|
    +\{\|\dm(Z_i,\htheta_n)\|\wedge H_{\varrho_n}\}\left|
    \frac{\Psi^{0*}_{\varrho_n}-\Psi_{\varrho_n}}
    {\Psi^{0*}_{\varrho_n}\Psi_{\varrho_n}}\right|\notag\\
  &\le\frac{\Big|\|\dm(Z_i,\ttheta_{s_0,P}^{0*})\|
    -\|\dm(Z_i,\htheta_n)\|\Big|}{\Psi^{0*}_{\varrho_n}}
    +\frac{\big|H^{0*}-H_{\varrho_n}\big|}
    {\Psi^{0*}_{\varrho_n}}
    +\{\|\dm(Z_i,\htheta_n)\|\wedge H_{\varrho_n}\}
    \frac{\big|\Psi^{0*}_{\varrho_n}-\Psi_{\varrho_n}\big|}
    {\Psi^{0*}_{\varrho_n}\Psi_{\varrho_n}},
\end{align*}
we have that
\begin{align*}
  &\|\Lambda_{n,P}^\alpha(\ttheta_{s_0,P}^{0*})
    -\Lambda_{P\varrho_n}^\alpha(\htheta_n)\|\notag\\
  &<\frac{1}{\alpha^2}\sumn\|\dm(Z_i,\htheta_n)\|^2\left|
    \tilde{\pi}_{n,Pi}^{\mvc}(\ttheta_{s_0,P}^{0*})
    -\pi_{n,Pi}^{\varrho_n}(\htheta_n)\right|\\
  &\le\frac{1}{\alpha^2\Psi^{0*}_{\varrho_n}}
    \sumn\|\dm(Z_i,\htheta_n)\|^2
    \Big|\|\dm(Z_i,\ttheta_{s_0,P}^{0*})\|-\|\dm(Z_i,\htheta_n)\|\Big|\\
  &\quad+\frac{\big|H^{0*}-H_{\varrho_n}\big|}
    {\alpha^2\Psi^{0*}_{\varrho_n}}
    \sumn\|\dm(Z_i,\htheta_n)\|^2
    +\frac{\big|\Psi^{0*}_{\varrho_n}-\Psi_{\varrho_n}\big|}
    {\alpha^2\Psi^{0*}_{\varrho_n}\Psi_{\varrho_n}}
    \sumn\|\dm(Z_i,\htheta_n)\|^3=\op,
\end{align*}
by \eqref{eq:Delta1} and Lemma~\ref{lem:H}. 

If $\varrho=0$, then, 
\begin{align}
  &n\left|\tilde{\pi}_{n,Pi}^{\mvc}(\ttheta_{s_0,P}^{0*})
    -\pi_{n,Pi}^{\varrho_n}(\htheta_n)\right|\notag\\
  &\le\frac{\Big|\|\dm(Z_i,\ttheta_{s_0,P}^{0*})\|\wedge H^{0*}
    -\|\dm(Z_i,\htheta_n)\|\wedge H_{\varrho_n}\Big|}
    {\Psi^{0*}_{\varrho_n}}
    +\|\dm(Z_i,\htheta_n)\|\left|
    \frac{\Psi^{0*}_{\varrho_n}-\Psi_{\varrho_n}}
    {\Psi^{0*}_{\varrho_n}\Psi_{\varrho_n}}\right|\notag\\
  &\le\frac{\Big|\|\dm(Z_i,\ttheta_{s_0,P}^{0*})\|
    -\|\dm(Z_i,\htheta_n)\|\Big|}{\Psi^{0*}_{\varrho_n}}
   +\|\dm(Z_i,\htheta_n)\|
    \frac{\big|\Psi^{0*}_{\varrho_n}-\Psi_{\varrho_n}\big|}
    {\Psi^{0*}_{\varrho_n}\Psi_{\varrho_n}}\notag\\
  &\quad+\frac{\|\dm(Z_i,\ttheta_{s_0,P}^{0*})\|}{\Psi^{0*}_{\varrho_n}}
    I\Big\{\|\dm(Z_i,\htheta_n)\|\ge H_{\varrho_n}\Big\}
  +\frac{H_{\varrho_n}}{\Psi^{0*}_{\varrho_n}}
    I\Big\{\|\dm(Z_i,\htheta_n)\|\ge H_{\varrho_n}\Big\}\notag\\
  &\quad+\frac{\|\dm(Z_i,\htheta_n)\|}{\Psi^{0*}_{\varrho_n}}
    I\Big\{\|\dm(Z_i,\ttheta_{s_0,P}^{0*})\|\ge H^{0*}\Big\}
  +\frac{H^{0*}}{\Psi^{0*}_{\varrho_n}}
    I\Big\{\|\dm(Z_i,\ttheta_{s_0,P}^{0*})\|\ge H^{0*}\Big\}\notag\\
  &\equiv\Delta_{3i}+\Delta_{4i}+\Delta_{5i}+\Delta_{6i}
    +\Delta_{7i}+\Delta_{8i}.\label{eq:22}
\end{align}
From \eqref{eq:Delta1} and Lemma \ref{lem:H}, we know that
\begin{align}\label{eq:23}
  &\onen\sumn\|\dm(Z_i,\htheta_n)\|^2\Delta_{3i}=\op
    \quad\text{and}\quad
  \onen\sumn\|\dm(Z_i,\htheta_n)\|^2\Delta_{4i}=\op.
\end{align}
Note that
\begin{align*}
  &\onen\sumn\|\dm(Z_i,\htheta_n)\|^2\|\dm(Z_i,\ttheta_{s_0,P}^{0*})\|
    I\Big\{\|\dm(Z_i,\htheta_n)\|\ge H_{\varrho_n}\Big\}\\
  &\le\bigg\{\onen\sumn\|\dm(Z_i,\htheta_n)\|^4\bigg\}^{\frac{1}{2}}
  \bigg\{\onen\sumn\|\dm(Z_i,\ttheta_{s_0,P}^{0*})\|^4\bigg\}^{\frac{1}{4}}
  \bigg[\onen\sumn I\Big\{\|\dm(Z_i,\htheta_n)\|\ge
    H_{\varrho_n}\Big\}\bigg]^{\frac{1}{4}}\\
  &=\op,
\end{align*}
because
\begin{align*}
  &\onen\sumn I\big\{\|\dm(Z_i,\htheta_n)\|\ge H_{\varrho_n}\big\}=\op, \quad \onen\sumn\|\dm(Z_i,\htheta_n)\|^4=\Op, \\
  &\quad\text{and}\quad
\onen\sumn\|\dm(Z_i,\ttheta_{s_0,P}^{0*})\|^4=\Op.
\end{align*}
Thus, 
\begin{align}\label{eq:24}
  &\onen\sumn\|\dm(Z_i,\htheta_n)\|^2\Delta_{5i}=\op.
\end{align}
If $\|\dm(Z,\btheta)\|$ is bounded, then
\begin{equation*}
  \frac{H_{\varrho_n}}{n}\sumn\|\dm(Z_i,\htheta_n)\|^2
    I\Big\{\|\dm(Z_i,\htheta_n)\|\ge H_{\varrho_n}\Big\}
    \le\frac{n-\ceil{n(1-\varrho_n)}}{n}\|\dm(Z,\htheta_n)\|_{(n)}^3
    =\op;
\end{equation*}
otherwise
\begin{align*}
  \frac{H_{\varrho_n}}{n}\sumn\|\dm(Z_i,\htheta_n)\|^2
    I\Big\{\|\dm(Z_i,\htheta_n)\|\ge H_{\varrho_n}\Big\}
  \le\frac{1}{nH_{\varrho_n}}\sumn\|\dm(Z_i,\htheta_n)\|^4
  =\op.
\end{align*}
Thus we know that
\begin{align}\label{eq:25}
  &\onen\sumn\|\dm(Z_i,\htheta_n)\|^2\Delta_{6i}=\op.
\end{align}
Similarly, we can obtain that
\begin{align}\label{eq:26}
  &\onen\sumn\|\dm(Z_i,\htheta_n)\|^2\Delta_{7i}=\op
    \quad\text{and}\quad
  \onen\sumn\|\dm(Z_i,\htheta_n)\|^2\Delta_{8i}=\op.
\end{align}
Combining \eqref{eq:70}, \eqref{eq:22}, \eqref{eq:23}, \eqref{eq:24}, \eqref{eq:25}, and \eqref{eq:26}, we know that 
\begin{align*}
  &\|\Lambda_{n,P}^\alpha(\ttheta_{s_0,P}^{0*})
    -\Lambda_{P\varrho_n}^\alpha(\htheta_n)\|=\op.
\end{align*}

To finish the proof for the case of $\varrho=0$, we only need to show that 
$\|\Lambda_{P\varrho_n}^\alpha(\htheta_n)-\Lambda_R^\alpha(\htheta_n)\|=\op$. Let $\Psi_{\infty}=\onen\sumn\{\|\dm(Z_i,\htheta_n)\|\}$. We notice that
\begin{align}
  &n\left|\pi_{n,Pi}^{\varrho_n}(\htheta_n)
    -\pi_{n,Ri}^{\mvc}(\htheta_n)\right|\notag\\
  &\le\frac{\Big|\|\dm(Z_i,\htheta_n)\|\wedge H_{\varrho_n}
    -\|\dm(Z_i,\htheta_n)\|\Big|}{\Psi_{\varrho_n}}
    +\|\dm(Z_i,\htheta_n)\|\left|
    \frac{\Psi_{\varrho_n}-\Psi_{\infty}}
    {\Psi_{\varrho_n}\Psi_{\infty}}\right|\notag\\
  &\le\frac{\|\dm(Z_i,\htheta_n)\|}{\Psi_{\varrho_n}}
    I\Big\{\|\dm(Z_i,\htheta_n)\|\ge H_{\varrho_n}\Big\}
   +\|\dm(Z_i,\htheta_n)\|
    \frac{\big|\Psi_{\varrho_n}-\Psi_{\infty}\big|}
    {\Psi_{\varrho_n}\Psi_{\infty}}\notag\\
  &\le\frac{\|\dm(Z_i,\htheta_n)\|^2}{\Psi_{\varrho_n}H_{\varrho_n}}
   +\|\dm(Z_i,\htheta_n)\|
    \frac{\big|\Psi_{\varrho_n}-\Psi_{\infty}\big|}
    {\Psi_{\varrho_n}\Psi_{\infty}}
  \equiv\Delta_{9i}+\Delta_{10i}.
\end{align}
With this result, it can be shown that 
\begin{align*}
  &\onen\sumn\|\dm(Z_i,\htheta_n)\|^2\Delta_{9i}=\op
    \quad\text{and}\quad
  &\onen\sumn\|\dm(Z_i,\htheta_n)\|^2\Delta_{10i}=\op,
\end{align*}
which indicates that $\|\Lambda_{P\varrho_n}^\alpha(\htheta_n)-\Lambda_R^\alpha(\htheta_n)\|=\op$.

From Slutsky's theorem, we know that given $\Dn$ and ${\ttheta_{s_0,P}^{0*}}$, as $s_0$, $s_n$, and $n$ go to infinity,
\begin{align*}
  \sqrt{s_n}\{V_{n,P}^\alpha(\htheta_n)\}^{-1/2}(\ttheta_{s_n,P}^\alpha-\htheta_n)\rightarrow \Nor\left(\0, \I\right),
\end{align*}
in conditional distribution. %
\end{proof}

\begin{proof}[Proof of Remark~\ref{remark8}]
  Since $\Lambda_R^\opt(\htheta_n)$ has the minimum trace among all choices of sampling probabilities, if $\alpha\neq 0$ then $\tr\{\Lambda_R^\opt(\htheta_n)\} < \tr\{\Lambda_R^{\alpha}(\htheta_n)\}$. 
On the other hand, 
\begin{align*}
\tr\{\Lambda^{\alpha}_R(\htheta_n)\}
  =\frac{1}{n^2}\sumn\frac{\|\dm(Z_i,\htheta_n)\|^2}
  {(1-\alpha)\pi_{n,i}^{R\mvc}+\alpha\onen}
  &<\frac{1}{n^2}\sumn\frac{\|\dm(Z_i,\htheta_n)\|^2}
    {(1-\alpha)\pi_{n,i}^{R\mvc}}\\
&=\frac{1}{(1-\alpha)n^2}\left\{\sumn \|\dm(Z_i,\htheta_n)\|\right\}^2
=\frac{\tr_{opt}\{\Lambda_{n,R}(\htheta_n)\}}{1-\alpha},
\end{align*}
and this finishes the proof for $\Lambda_R^{\alpha}(\htheta_n)$ from subsampling with replacement. For $\Lambda_{n,P}^{\alpha}(\htheta_n)$ from Poisson subsampling, the proof is similar.

\end{proof}

\section{Additional examples on optimal structural results}
\label{sec:addit-exampl-optim}
\begin{example}[Least-squares]\normalfont
  \label{example2}
Consider least-squares estimator
\begin{equation*}
  \htheta_n=\arg\min_\btheta\sumn\{y_i-g(\x_i,\btheta)\}^2,
\end{equation*}
where $y_i$ is the response, $\x_i$ is the covariate, %
and $g(\x_i,\btheta)$ is a smooth function.
The least-squares estimator of $\btheta$
 can be presented in our framework by letting
$Z_i=(\x_i, y_i)$ and defining %
\begin{equation*}
  m(Z_i,\btheta)=-0.5\{y_i-g(\x_i,\btheta)\}^2.
\end{equation*}
From direct calculation, we have
\begin{align}\label{eq:34}
  \dm(Z_i,\htheta_n)
  &=\hat\varepsilon_i\dot{g}(\x_i,\htheta_n), \quad\text{ and }\quad
  \ddm(Z_i,\htheta_n)
  =\hat\varepsilon_i\ddot{g}(\x_i,\htheta_n)
  -\dot{g}(\x_i,\htheta_n)\dot{g}\tp(\x_i,\htheta_n),%
\end{align}
where $\hat\varepsilon_i=y_i-g(\x_i,\htheta_n)$, $\dot{g}(\x_i,\htheta_n)$ and $\ddot{g}(\x_i,\htheta_n)$ are the gradient and Hessian matrix of $g(\x_i,\btheta)$, respectively, evaluated at $\htheta_n$. Note that $\onen\sumn\hat\varepsilon_i\ddot{g}(\x_i,\htheta_n)$ is a small term, so there is no need to calculate the Hessian matrix $\ddot{g}(\x_i,\htheta_n)$, and $\ddM_n(\htheta_n)$ %
can be replaced by
\begin{equation}\label{eq:35}
    \ddM_n^a(\htheta_n)=
    -\onen\sumn\dot{g}(\x_i,\htheta_n)\dot{g}\tp(\x_i,\htheta_n).
\end{equation}
From~\eqref{eq:34} and~\eqref{eq:35}, we obtain optimal sampling probabilities by using
\begin{align}\label{eq:36}
  \|\dm(Z_i,\htheta_n)\|
  &=|\hat\varepsilon_i|\|\dot{g}(\x_i,\htheta_n)\|,
  \quad\text{ or}\quad%
   \|\dm(Z_i,\htheta_n)\|_L
  =n|\hat\varepsilon_i|\Big\|L\{\ddM_n^a(\htheta_n)\}^{-1}
  \dot{g}(\x_i,\htheta_n)\Big\|,
\end{align}
to replace $\|\dm(Z_i,\htheta_n)\|$ %
in Theorems~\ref{thm:3} and \ref{thm:4} for different subsampling procedures.

Specifically for ordinary least-squares (OLS) in linear regression, $g(\x_i,\btheta)=\x_i\tp\btheta$, $\dot{g}(\x_i,\htheta_n)=\x_i$, and $\ddot{g}(\x_i,\htheta_n)=\0$. Therefore, the expression in \eqref{eq:36} is simplified to
\begin{align}\label{eq:37}
  \|\dm(Z_i,\htheta_n)\|
  &=|\hat\varepsilon_i|\|\x_i\|,
  \quad\text{ or}\quad
   \|\dm(Z_i,\htheta_n)\|_L
    =n|\hat\varepsilon_i|\big\|L(\X\X\tp)^{-1}\x_i\big\|,
\end{align}
where $\X=(\x_1, ..., \x_n)\tp$.

With $\|\dm(Z_i,\htheta_n)\|=|\hat\varepsilon_i|\|\x_i\|$
inserted into~\eqref{eq:SSPmMSE}, the sampling probabilities reduce to gradient-based sampling probabilities \citep{Zhu2016}. Furthermore, %
if we take $L=\{-n\ddM_n(\htheta_n)\}^{1/2}=(\X\tp\X)^{1/2}$ in \eqref{eq:37}, the optimal probabilities for subsampling with replacement satisfy that
\begin{equation}\label{eq:39}
  \pi_{n,Ri}^{\mvc}\propto{|\hat\varepsilon_i|\sqrt{h_i}},%
  \quad i=1, ..., n,
\end{equation}
where $h_i$'s are statistical leverage scores of $\x_i$'s, i.e., diagonal elements of $\X(\X\tp\X)^{-1}\X\tp$.
This clearly shows the connection between leverage scores and the L optimality. 

Form \eqref{eq:39} and Theorem~\ref{thm:4}, optimal probabilities for Poisson subsampling and subsampling with replacement differ if there are data points such that $\frac{s_n}{n}|\hat\varepsilon_i|\sqrt{h_i}>\onen\sumjn|\hat\varepsilon_j|\sqrt{h_j}$. This is more likely to happen if $|\hat\varepsilon_i|$'s or $\sqrt{h_i}$'s are more nonuniform. %
\cite{yang2015explicit} showed that if statistical leverage scores are very nonuniform, then using the square roots of statistical leverage scores to construct subsampling probabilities yields better approximation than using the original leverage scores. An intuitive explanation for their conclusion is that taking score roots on leverage scores has some shrinkage effect on the resulting probabilities toward the uniform subsampling probability. Our results echos their conclusion, and further indicates that for optimal Poisson subsampling it may be necessary to perform truncation for high leverage scores. 

\end{example}

\begin{example}[Generalized linear models]\normalfont
Let $y_i$ be the response and $\x_i$ be the corresponding covariate. A generalized linear model (GLM) assumes that the conditional mean of the response $y_i$ given the covariate $\x_i$, $\Exp(y_i|\x_i)$, satisfies
  \begin{equation*}
    g\{\Exp(y_i|\x_i)\}=\x_i\tp\bbeta,
  \end{equation*}
  where $g$ is the link function, $\x_i\tp\bbeta$ is the linear predictor, and $\bbeta$ is the regression coefficient. For most of the commonly used GLMs, it is assumed that the distribution of the response $y_i$ given the covariate $\x_i$ belongs to the exponential family, namely,
\begin{equation*}
  f(y_i|\x_i;\bbeta,\phi)=a(y_i,\phi)\exp\Big[
  \frac{y_ib(\x_i\tp\bbeta)-c(\x_i\tp\bbeta)}{\phi}\Big],
\end{equation*}
where $a$, $b$ and $c$ are known scalar functions, and $\phi$ is the dispersion parameter. In the framework of GLM. If the link function $g$ is selected such that $b$ is the identity function, i.e., $b(\x_i\tp\bbeta)=\x_i\tp\bbeta$, then the link function is called the canonical link. With a canonical link function, $g\{\Exp(y_i|\x_i)\}=c'(\x_i\tp\bbeta)$ where $c'$ is the derivative function of $c$.

Let $Z_i=(\x_i, y_i)$. If both the regression coefficient $\bbeta$ and the dispersion parameter $\phi$ are of interest, then let $\btheta=(\bbeta\tp,\phi)\tp$. The MLE of $\btheta$ corresponds to
\begin{equation*}
  m(Z_i,\btheta)
  =\frac{y_ib(\x_i\tp\bbeta)-c(\x_i\tp\bbeta)}{\phi}
  +\log\{a(y_i,\phi)\}.
\end{equation*}
If $\bbeta$ is the only parameter of interest, then $\btheta=\bbeta$, and the MLE of $\btheta$ corresponds to
\begin{equation*}
  m(Z_i,\btheta)=y_ib(\x_i\tp\bbeta)-c(\x_i\tp\bbeta).
\end{equation*}
For this case, direct calculations give us that
\begin{align}\label{eq:41}
  \dm(Z_i,\btheta)
  =\{y_ib'(\x_i\tp\bbeta)-c'(\x_i\tp\bbeta)\}\x_i
  \ \text{and}\ \
  \ddm(Z_i,\btheta)
  =\{y_ib''(\x_i\tp\bbeta)-c''(\x_i\tp\bbeta)\}\x_i\x_i\tp,
\end{align}
where $b'$ and $b''$ are the first and second derivative functions of $b$, and and $c''$ is the second derivative function of $c$. Thus, optimal sampling probabilities under the L-optimality can be obtained by using the expressions in \eqref{eq:41} for Theorems~\ref{thm:3} and \ref{thm:4}. If the canonical link is used, then the expressions in \eqref{eq:41} simplify to
\begin{align*}
  \dm(Z_i,\btheta)=\{y_i-c'(\x_i\tp\bbeta)\}\x_i
  \quad\text{and}\quad
  \ddm(Z_i,\btheta)=-c''(\x_i\tp\bbeta)\x_i\x_i\tp.
\end{align*}
The following list gives the forms of $\m(Z_i,\btheta)$,  $\dm(Z_i,\btheta)$, and $\ddm(Z_i,\btheta)$ for commonly used GLMs with the canonical links.

\begin{itemize}
\item {\bf Normal distribution}, $y_i|\x_i\sim \Nor(\mu_i,\sigma^2)$.
  \begin{itemize}
  \item Canonical link: $g(\mu_i)=\mu_i=\x_i\tp\bbeta$.
  \item Parameter $\btheta=(\bbeta\tp,\sigma^2)\tp$:
    \begin{itemize}
    \item $m(Z_i,\btheta)
      =\frac{-(y_i-\x_i\tp\bbeta)^2}{2\sigma^2}
      -\frac{\log(\sigma^2)}{2}$. 
    \item $\dm(Z_i,\htheta_n) =\frac{1}{\hat\sigma^2}
      \begin{bmatrix}
        \hat\varepsilon_i\x_i\\[2mm]
        \frac{\hat\varepsilon_i^2-\hat\sigma^2}{2\hat\sigma^2}
      \end{bmatrix}$, and
      $\ddM_n(\htheta_n)=\frac{-1}{n\hat\sigma^2}
      \begin{bmatrix}
        \X\tp\X & \0 \\[1mm] \0 & \frac{n}{2\hat\sigma^2}
      \end{bmatrix}$,\\
      where $\X=(\x_1, ..., \x_n)\tp$,
      $\hat\varepsilon_i=y_i-\x_i\tp\hbeta$ and
      $\hat\sigma^2=\onen\sumn\hat\varepsilon_i^2$.
    \end{itemize}
  \item Parameter $\btheta=\bbeta\tp$ when $\sigma^2$ is not of
    interest:
    \begin{itemize}
    \item $m(Z_i,\btheta)=-(y_i-\x_i\tp\bbeta)^2$.
    \item $\dm(Z_i,\htheta_n)$ and $\ddM_n(\htheta_n)$ are the same to
      case of OLS in Example~\ref{example2}.
    \end{itemize}
  \end{itemize}

\item {\bf Binomial distribution}, $y_i|\x_i\sim \mathbb{BIN}(k_i,p_i)$. The problem is often converted to model the ratio $y_i^r=y_i/k_i$.
  \begin{itemize}
  \item Canonical link: $g(p_i)=\log(\frac{p_i}{1-p_i})=\x_i\tp\bbeta$.
  \item Parameter $\btheta=\bbeta$:
    \begin{itemize}
    \item $m(Z_i,\btheta) =
      k_i\{y_i^r\x_i\tp\bbeta-\log(1+e^{\x_i\tp\bbeta})\}$. 
    \item $\dm(Z_i,\htheta_n)=k_i(y_i^r-\hat{p}_i)\x_i,
  \quad\text{ and }\quad
  \ddM_n(\htheta_n)=-\onen\sumn k_i\hat{p}_i(1-\hat{p}_i)\x_i\x_i\tp$,\\
  where $\hat{p}_i=e^{\x_i\tp\hbeta}/(1+e^{\x_i\tp\hbeta})$.
    \end{itemize}
  \end{itemize}
  If $k_i=1$ for all $i$, the results reduce to the case of logistic regression in Example \ref{example1}.

\item {\bf Poisson distribution}, $y_i|\x_i\sim \mathbb{POI}(\mu_i)$.
  \begin{itemize}
  \item Canonical link: $g(\mu_i)=\log(\mu_i)=\x_i\tp\bbeta$.
  \item Parameter $\btheta=\bbeta$:
    \begin{itemize}
    \item $m(Z_i,\btheta)=y_i\x_i\tp\bbeta-e^{\x_i\tp\bbeta}$. 
    \item $\dm(Z_i,\htheta_n)=(y_i-e^{\x_i\tp\bbeta})\x_i$, and
      $\ddM_n(\htheta_n)=-\onen\sumn e^{\x_i\tp\bbeta}\x_i\x_i\tp$
    \end{itemize}
  \end{itemize}

\item {\bf Gamma distribution},
  $y_i|\x_i\sim \mathbb{GAM}(\nu, \mu_i)$, with density function 
  \begin{equation}\label{eq:18}
    f(y_i)=\frac{\nu^\nu}{\Gamma(\nu)\mu_i^\nu}\
    y_i^{\nu-1}\ e^{-\frac{\nu y_i}{\mu_i}}, \;\; y_i>0,
  \end{equation}
  where $\nu$ is the shape parameter and $\mu_i$ is the mean parameter.\footnote{A Gamma distribution is also often parameterized in terms of the shape and rate parameters or the shape and scale parameters. With our notations here, the shape and rate parameters are $\nu$ and $\nu/\mu_i$, respectively, and the shape and scale parameters are $\nu$ and $\mu_i/\nu$, respectively.}
  \begin{itemize}
  \item Canonical link: $g(\mu_i)=\frac{-1}{\mu_i}=\x_i\tp\bbeta$.
  \item Parameter $\btheta=(\bbeta\tp, \nu)\tp$:
    \begin{itemize}
    \item $m(Z_i,\btheta)=\nu y_i\x_i\tp\bbeta
      +\nu \log(-\x_i\tp\bbeta)+\nu\log\nu
      +(\nu-1)\log(y_i)-\log\{\Gamma(\nu)\}$.
    \item $\dm(Z_i,\htheta_n)=
      \begin{bmatrix}
        \hat\nu\Big(y_i+\frac{1}{\x_i\tp\hbeta}\Big)\x_i \\
        y_i\x_i\tp\hbeta+\log(-\x_i\tp\hbeta)
        +\log(\hat{\nu})+1+\log(y_i)
        -\frac{\Gamma'(\hat{\nu})}{\Gamma(\hat{\nu})}
      \end{bmatrix}$, \\
      and $\ddM_n(\htheta_n)=
      \begin{bmatrix}
        -\frac{\hat\nu}{n}\sumn\frac{1}{(\x_i\tp\hbeta)^2}\x_i\x_i\tp
        & \0 \\ \0 & \frac{1}{\hat{\nu}}
        -\frac{\Gamma''(\hat{\nu})\Gamma(\hat{\nu})-\{\Gamma''(\hat{\nu})\}^2}
        {\{\Gamma(\hat{\nu})\}^2}\end{bmatrix}$,
      where $\Gamma'(\hat{\nu})$ and $\Gamma''(\hat{\nu})$ are the first and second derivative of $\Gamma(\nu)$ evaluated at $\hat{\nu}$.
    \end{itemize}
  \item Parameter $\btheta=\bbeta$:
    \begin{itemize}
    \item $m(Z_i,\btheta)=y_i\x_i\tp\bbeta+\log(-\x_i\tp\bbeta)$.
    \item $\dm(Z_i,\htheta_n)=\Big(y_i+\frac{1}{\x_i\tp\hbeta}\Big)\x_i$, and $\ddM_n(\htheta_n)=-\onen\sumn \frac{1}{(\x_i\tp\hbeta)^2}\x_i\x_i\tp$.
    \end{itemize}
  \end{itemize}
   If $\nu=1$ in \eqref{eq:18}, then the Gamma distribution reduces to an exponential, and thus the results reduce to the case of exponential regression. For inverse Gamma distribution, one can use the reciprocal transformation, i.e., $1/y_i$, to convert the problem to Gamma distribution. 
\end{itemize}
\end{example}

\bibliographystyle{agsm}
\bibliography{ref}

\end{document}